\documentclass[12pt]{article}

\usepackage{cite,epic,eepic,euscript,verbatim,amsmath,amssymb,amsthm,amscd,amsfonts,afterpage,float,bm,stmaryrd,paralist,mathrsfs}
\usepackage{booktabs,color,csvsimple,multirow,cancel,cases}
\usepackage{gensymb} 
\usepackage{graphicx,graphics,epsf,epsfig,colordvi,subfigure,caption,wrapfig,epstopdf}
\usepackage{algpseudocode}
\usepackage{authblk}
\usepackage{lipsum}
\usepackage[titletoc]{appendix}
\usepackage{setspace}
\usepackage{empheq}
\usepackage{ulem}

\usepackage{fancyhdr}              
\usepackage{hyperref}              
\usepackage{makeidx}
\usepackage{multirow}

\usepackage{enumerate}

\usepackage{listings}
\usepackage{enumitem}

\usepackage{wrapfig}

\theoremstyle{plain}
\newtheorem{definition}{Definition}[section]
\newtheorem{theorem}{Theorem}[section]
\newtheorem{lemma}[theorem]{Lemma}
\newtheorem{proposition}[theorem]{Proposition}
\newtheorem{corollary}[theorem]{Corollary}

\theoremstyle{definition}

\theoremstyle{remark}
\newtheorem{rmk}{Remark}

\numberwithin{equation}{section}

\newcommand{\R}{\mathbb{R}}

\newcommand{\bn}{\mathbf{n}}
\newcommand{\br}{\mathbf{r}}
\newcommand{\bv}{\mathbf{v}}

\graphicspath{{figs/}}

\begin{document}

\title{Kinetics of an Expanding Bacterial Colony: Continuum Modeling and Analysis  }
\author[1]{Bo Li}
\author[2]{Mykhailo Potomkin}

\affil[1]{Department of Mathematics,
University of California San Diego\\

Email: bli@math.ucsd.edu} 

\affil[2]{Department of Mathematics, 
University of California Riverside\\

Email: mykhailp@ucr.edu}

\maketitle

\begin{abstract}
  We study the spatiotemporal dynamics of the 
  expansion of a bacterial colony on a hard substrate. Nutrient from the substrate diffuses into the colony and 
  is taken up by the bacterial cells for them to grow and divide, expanding the initial monolayer
  and then pancake-shaped colony. The concentration of nutrient
  determines the local cell growth rate in the colony. Mass conservation relates such local growth rate
  with velocity which is approximated to be proportional to the pressure gradient by  Darcy's law. 
  Altogether, the growing colony
  is modeled as a moving-boundary problem with the nutrient concentration and pressure solving a reaction-diffusion equation and 
  Laplace's equation, respectively. 
   We analyze the self-consistent moving-boundary model with respect to different geometrical setting. For a 
    general three-dimensional cylindrically symmetric model, we study the steady-state nutrient concentration. 
    We construct and analyze a 
   one-dimensional model for vertical expansion and a two-dimensional disk model for radial expansion 
   of the colony. Our analysis finds that 
   the nutrient depletes into the colony and slows
   down the vertical expansion of the colony. The vertical level where the nutrient concentration
   reaches the Monod constant, a threshold below which individual bacteria hardly grow, lowers down 
   exponentially fast. 
   We also estimate the 
   asymptotic radial expansion rate. Moreover, we establish that the region where the nutrient concentration 
   is above the threshold, allowing bacteria to grow and the colony to expand radially, is a ring-shaped peripheral region of fixed thickness.
   All these are consistent with experiment and agent-based simulations reported in literature. 

   \vspace{4 mm}


\end{abstract}

\tableofcontents

\allowdisplaybreaks

\section{Introduction}
\label{s:Introduction}

Bacterial colonies are building blocks and model systems of biofilms that are 
heterogeneous structures 
often attached to surfaces with extracellular polymeric substances excreted by bacterial cells. 
Biofilms are found almost everywhere on earth, in soils, oceans, and animal and human bodies.  Understanding
bacterial biofilms has therefore significant impacts on our health and environment
\cite{StewartFranklin_Rev2008,Kolter_BiofilmRev_TrendsMicrobiol2005,Costerton_Rev1995,Greenberg_Science1999,Flemming_Rev2010,Potera_Science2008,Monds_Review_TrendsMicrobiol2009,Weitz_ComplexFluids_MRSBulletin2011}. 
Modeling and their mathematical analysis of colony expansion provide a framework for understanding the dynamics of complex biofilm systems.

In this work, we consider the expansion of a bacterial colony on a substrate such as an agar plate. 
Individual bacterial cells in the colony take nutrients from the substrate and the 
surrounding environment to grow and divide, generating 
the force that expands the colony. 
After an initial period of exponential growth ($\sim 5$  hours) starting 
from a few cells, the colony keeps growing
from a monolayer to vertical expansion, forming a pancake-shaped three-dimensional structure, lasting
for a few days; 
cf.\ Figure~\ref{f:colony}. 
The expanding bacterial colony exhibits complex multiscale spatiotemporal patterns and dynamics, 
resulting from the interplay between
individual cell activities such as growth, division, and movement, the cell-cell and 
cell-environment mechanical interactions, and the metabolic interactions
 of nutrients and other biochemical substances such as glucose, oxygen, nitrogen, 
and wastes produced by bacteria 
\cite{Wimpenny_1979,Wimpenny_1981,Wimpenny_1983,Rieck_1973,Warren_eLife2019,Kannan_NatCommun2025,
Coombs_1987,Nemenman_PLoS2017,Schulten_BMC2015,Syu_2D3D_PLoSONE2012,Drescher_eLife2021,Reyrolle_1979,Cooper_1968}.

\begin{figure}[tp]
\begin{center}
\includegraphics[width=0.375\textwidth]{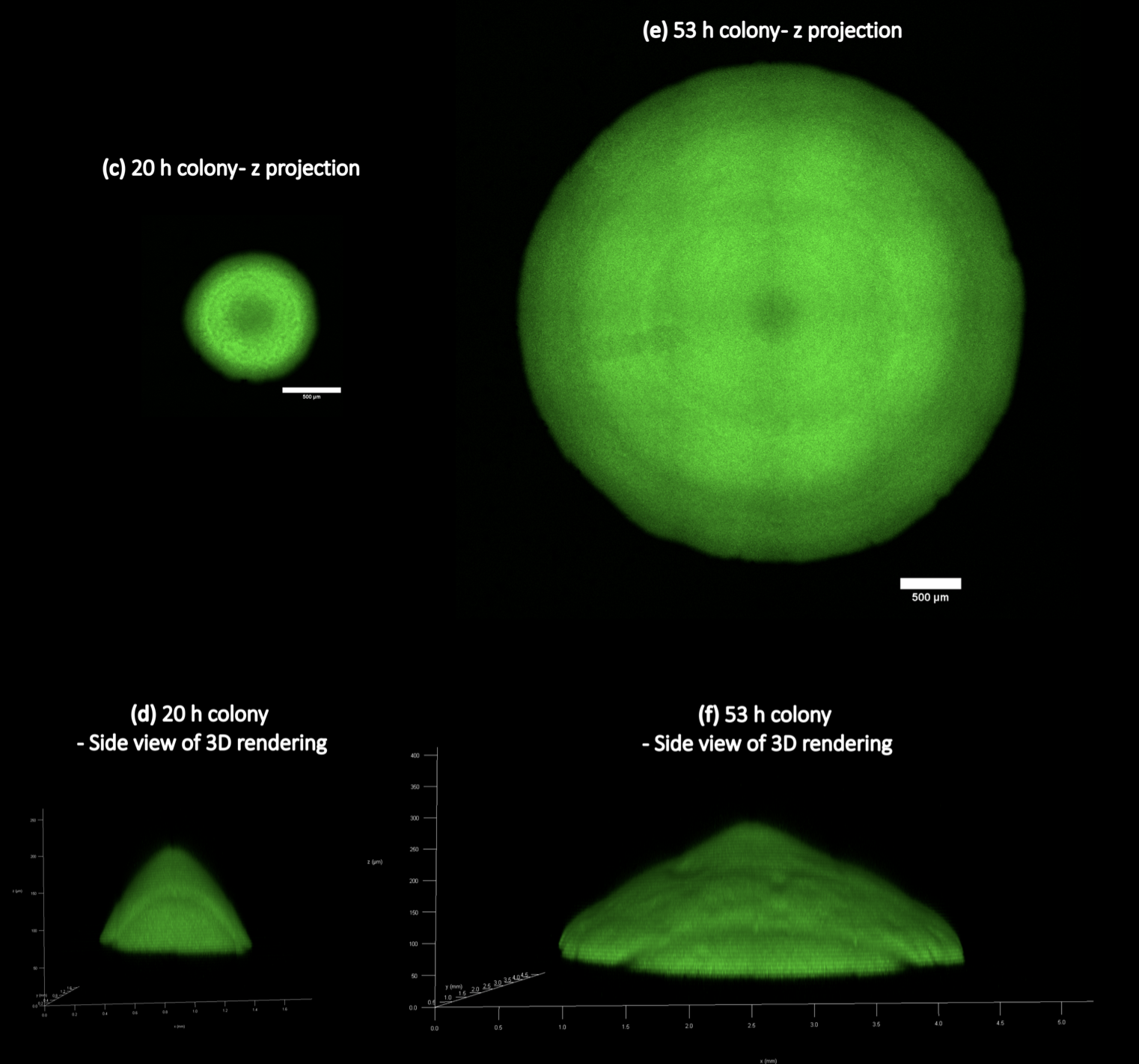}
\includegraphics[width=0.6\textwidth]{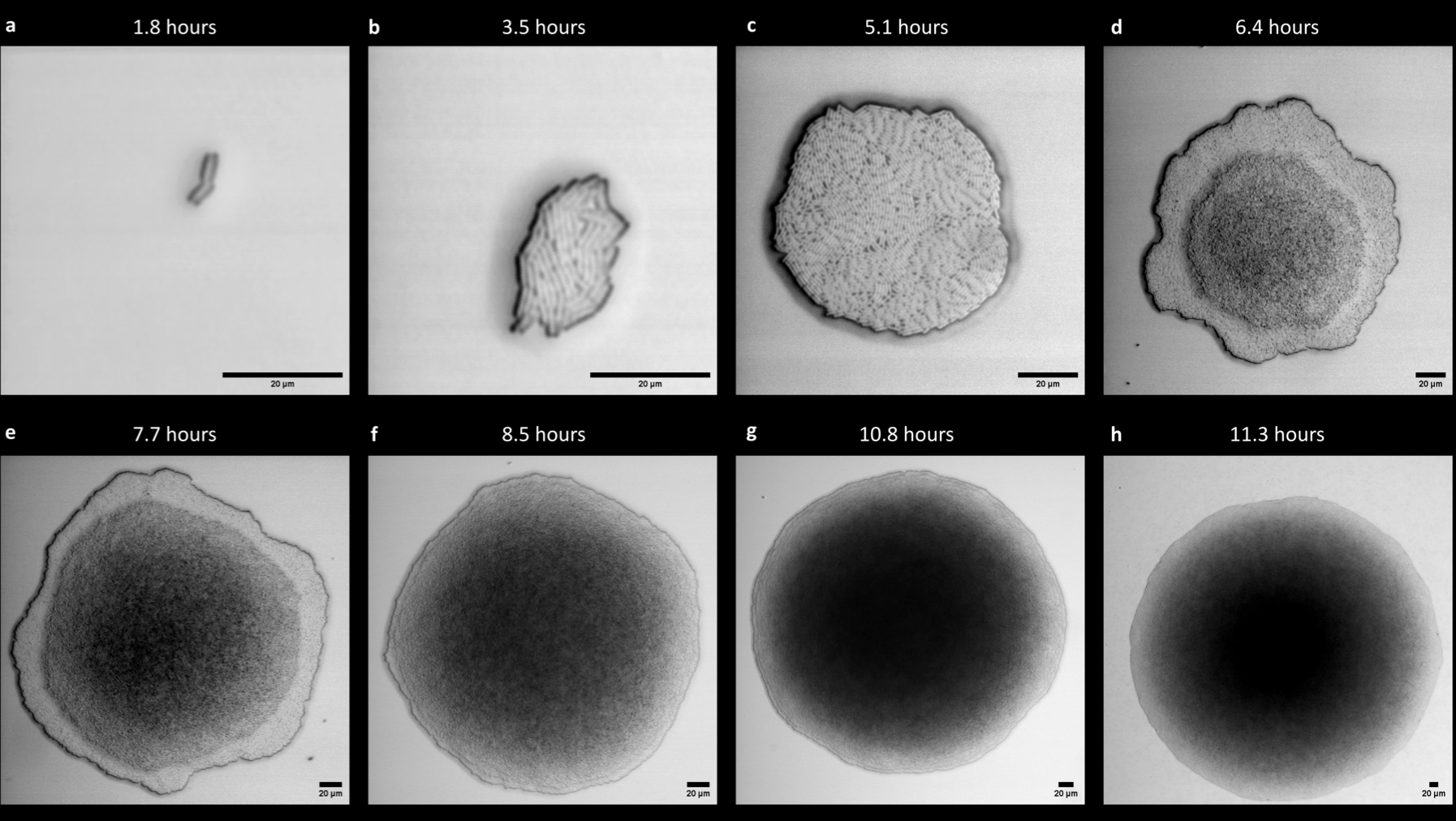}
\end{center}
\caption{Experiment: Confocal microscopy images of {\it E.~coli} colonies (Left) and the 
dynamics of early colony expansion 
(Right). Adapted from Fig.~S1 and Fig.~S2 of \cite{Kannan_NatCommun2025} (open access through Nature Communications).} 
\label{f:colony}
\end{figure}

Despite of enormous  complexities of an expanding colony comprised of millions to billions of individual bacterial cells,
robust kinetic behaviors of the macroscopic colony dynamics have long been observed experimentally. 
For instance, after the initial period of  exponential expansion, 
the colony radius increases with a constant speed for up to $\sim$ 65 hours while 
the colony height increases at a nearly  constant speed for 10 - 15 hours but then slows down significantly 
\cite{Pirt_1967,Wimpenny_1979,Palumbo_JGenMicrobiol1971,Kannan_NatCommun2025,Warren_eLife2019,Wimpenny_1981,Syu_2D3D_PLoSONE2012,Yunker_PNAS2023}. (These experiments are conducted with a small colony-forming unit, neglecting the influence from neighboring 
colonies. For otherwise the kinetic behaviors can be different; cf.\ e.g., \cite{Ernebjerg_2012}.)

Pirt (1967) \cite{Pirt_1967}, Cooper et al.\ (1968) \cite{Cooper_1968}, and Lewis and Wimpenny (1981) \cite{Wimpenny_1981} 
were among the early works discovering the kinetic behaviors and investigating 
the underlying mechanisms of the bacterial colony expansion. 
They argued that 
the constant radial expansion results from the exponential growth of cells in the colony 
peripheral region where the nutrients are abundant and the space is unlimited, while the vertical 
expansion is controlled by the nutrients penetrating into the colony with the vertical slowdown resulting
from the nutrient depletion. 

More recently, 
Warren et al.\ (2019) \cite{Warren_eLife2019} and Kannan et al.\ (2025) 
\cite{Kannan_NatCommun2025} developed a hybrid approach
combining the agent-based model 
(cf.\ also 
\cite{Haseloff_ACSSynthBiol2012,Tsimring_Ordering_PNAS2008,Levine_PhaseSeparation_PNAS2015,Aranson_CommunPhys2025,Levine_PRE2017}) that tracks individual cell activities and 
reaction-diffusion equations for modeling the metabolic dynamics, 
with the two components connected through the local cell growth rate.
The agent-based model describes the individual cell growth and division, 
cell-cell and cell-environment mechanical interactions, and the cell movement governed by Newton's law of motion. 
The hybrid simulations
successfully predicted the experimentally observed kinetic behaviors of expanding 
{\it E.~coli} colonies on agar for up to $\sim 65$ hours. Their 
quantitative findings indicate that the onset of mechanical buckling 
during the monolayer growth that leads to the transition from the planar
to three-dimensional colony expansion and 
the existence of a ring-shaped peripheral region with a fixed
width ($\sim 20\, \mu m$) of a monolayer of growing cells 
together determine the constant
speed of the radial expansion for up to $ \sim 65$ hours. 
Clearly, the nutrient depletion contributes to the vertical slowdown. 
These studies further found that colony maintenance consumes nutrients, eventually leading to nutrient depletion, cell starvation, and cell death, which ultimately halt the colony growth.
They also showed that in various regions of the colony, 
such as the top and the interface near the agar, cells continue to grow either aerobically or anaerobically.

While the hybrid modeling with the agent-based description  can examine detailed structures 
and dynamics of an expanding colony from individual cells, it is limited in terms of 
computational power by the size of an underlying system. 
 (This is in fact realized by Kannan et al.\
\cite{Kannan_NatCommun2025} where a reduced (1+1) geometry, one horizontal
substrate dimension and one vertical growth dimension,  is used.)

Here, we construct and analyze a moving-boundary reaction-diffusion model for the expansion of a bacterial colony on a substrate \cite{Warren_eLife2019,Kannan_NatCommun2025,martinez2022morphological,He_BMB2023,
Tsimring_eLife2020,Fu_PNAS2025,Klapper_SIAMRev2010,Yunker_NatPhys2024}. 
With the geometry  shown in Figure~\ref{f:agarcolony}, 
this model consists of three components. 

\begin{figure}[H]
\centering
\includegraphics[width=0.62\textwidth]{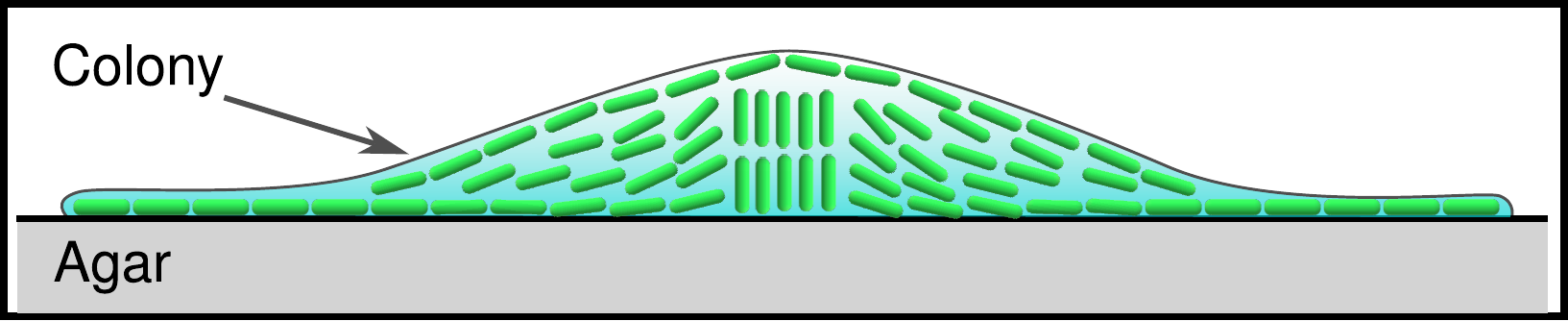}
\caption{{ Schematic of the substrate-colony system.}
}
\label{f:agarcolony}
\end{figure}

(1) {\it A set of reaction-diffusion
equations} for the concentrations $c_j = c_j(\br, t) $ $(j = 1, \dots, M)$ 
of nutrients and other chemicals (such as glucose, oxygen, nitrogen, acetate, etc.)


\begin{equation}
\label{eq:diffusion_interface_twocol}
\left\{
\begin{array}{lll}
\partial_t c_j
= D_{j,+} \Delta c_j
+ P_j(c_1,\dots,c_M)
- Q_j(c_1,\dots,c_M),
& \mbox{in colony}, & \\[0.4em]
\partial_t c_j
= D_{j,-} \Delta c_j,
& \mbox{in substrate}, & \\[0.6em]
\left.\begin{array}{l}
c_{j,+} = c_{j,-} \\
D_{j,+}\partial_n c_{j,+}
= D_{j,-}\partial_n c_{j,-}
\end{array}\right\}
&
 \begin{array}{c}
 \hspace{-60pt}\text{on colony-agar interface.}
 \end{array} &
\end{array}
\right.
\end{equation}
Here, $D_{j,\pm}$ are the diffusion constants
for the $j$th species in the colony and substrate, respectively, $P_j$ and $Q_j$ are the
rate of production and consumption of the $j$th species, and in the last equation
$c_{j,\pm} $ denote the values of $c_j$ from the colony and substrate, respectively, and $\partial_n$ is
the normal derivative at the colony-agar interface. 

(2) {\it The local growth rate, velocity field, and pressure field.}
The explicit form of the local growth rate 
\begin{equation} 
\label{generallambda}
\lambda = \lambda( c_1, \dots, c_M) \qquad \mbox{in colony}
\end{equation}
needs to be determined. For a single nutrient concentration 
$c $, we can define the local growth rate by the Monod kinetic form: $\lambda(c) \propto c / ( c + 1).$ 
An example of such $\lambda$ for concentrations
of three different nutrients or biochemical species is given in Kannan et al.\ (2025) \cite{Kannan_NatCommun2025}. 
Mass conservation leads to 
\[
\nabla \cdot (\rho \bv ) = \rho \lambda \qquad \mbox{in conlony},
\]
 where $\rho$ is the cell density. We will assume that the density $\rho$
 is a constant as suggested by agent-based simulations \cite{Warren_eLife2019,Kannan_NatCommun2025}. 
 Darcy's law provides the approximation 
 \[
 \bv = - \xi \nabla p, 
 \]
with $\xi > 0$ the friction constant. Combining all these, we obtain
\begin{equation}
\label{generalpressure}
- \xi  \Delta p = \lambda (c_1, \dots, c_M) \qquad \mbox{in colony}. 
\end{equation}

(3) {\it  The motion law for the moving colony surface.} 
The normal velocity $V_n$ is given by 
\begin{equation}
\label{generalVn}
V_n = \bv \cdot \bn - \gamma H = - \xi  \partial_n p -\gamma H \qquad 
\mbox{on the colony surface},
\end{equation}
where $\bn $ and $H$ are the exterior unit normal and the mean curvature of the colony surface and 
$\gamma> 0$ is the surface tension constant. 

We remark that the individual cell activities are not tracked in the continuum model 
\eqref{eq:diffusion_interface_twocol}--\eqref{generalVn} but the nutrient dynamics modeled more
efficiently by the reaction-diffusion equations
determines the rate of local cell growth which in turn determines the consumption of nutrient and 
the colony expansion dynamics.


To understand the essential properties of metabolic dynamics 
of a growing colony, 
we consider here in the model only a single bacterial species and only one type of nutrient which is the glucose
during a developmental phase of the growth, approximately between 6 -- 24 hours
starting from incubation, same setup as  in  \cite{Warren_eLife2019}.
We study the kinetics of an expanding bacterial colony in several steps. 
\begin{compactenum}
\item[(1)]
We show that, with a fixed configuration of the substrate-colony system, there is a unique equilibrium
concentration.  This is done by realizing that such a concentration field minimizes uniquely the convex
functional 
\[
{\mathcal E}[c] = \int \left[ \frac{1}{2} |\nabla c|^2 + \Lambda(c) \right] dV 
\]
among all the admissible concentrations, where the integral region includes both the colony and the substrate, 
and in the dimensionless form  $\Lambda'(c) = \lambda (c) = c / (1 + c).$ 
\item[(2)]
We construct a 
one-dimensional moving-boundary model that is reduced from the 
full three-dimensional model to study the kinetics of the vertical expansion in details.
We analyze the model to show that the nutrient 
concentration $c = c(z)$
decays quadratically first and then exponentially in the colony region $z > 0$ and that the 
point $z_\ast > 0$ defined by $c(z_\ast) = K$, where $K$ is the Monod constant, 
moves exponentially fast towards $z = 0$, the interface between the substrate and the colony region; 
cf.\ Figure~\ref{fig:0} (a). We also show that the speed $h'(t)$ of the colony height $h = h(t)$ has a constant 
limit as $t \to \infty.$
\item[(3)]  
We construct 
a two-dimensional disk model, again reduced from the full three-dimensional model,
for the monolayer growth of a bacterial colony, as illustrated in Figure~\ref{fig:0} (b). 
We analyze the nutrient 
concentration profile and the radial expansion kinetics. 
We show the existence of a ring-shaped peripheral region of the colony
with sufficient nutrient allowing cells in this region to grow for long time. We also
prove that the radius $R(t) = O(1)$ as $t\to \infty$, indicating the constant radial expansion. 
\end{compactenum}

\begin{figure}[t]
\centerline{\includegraphics[width=0.5\textwidth]{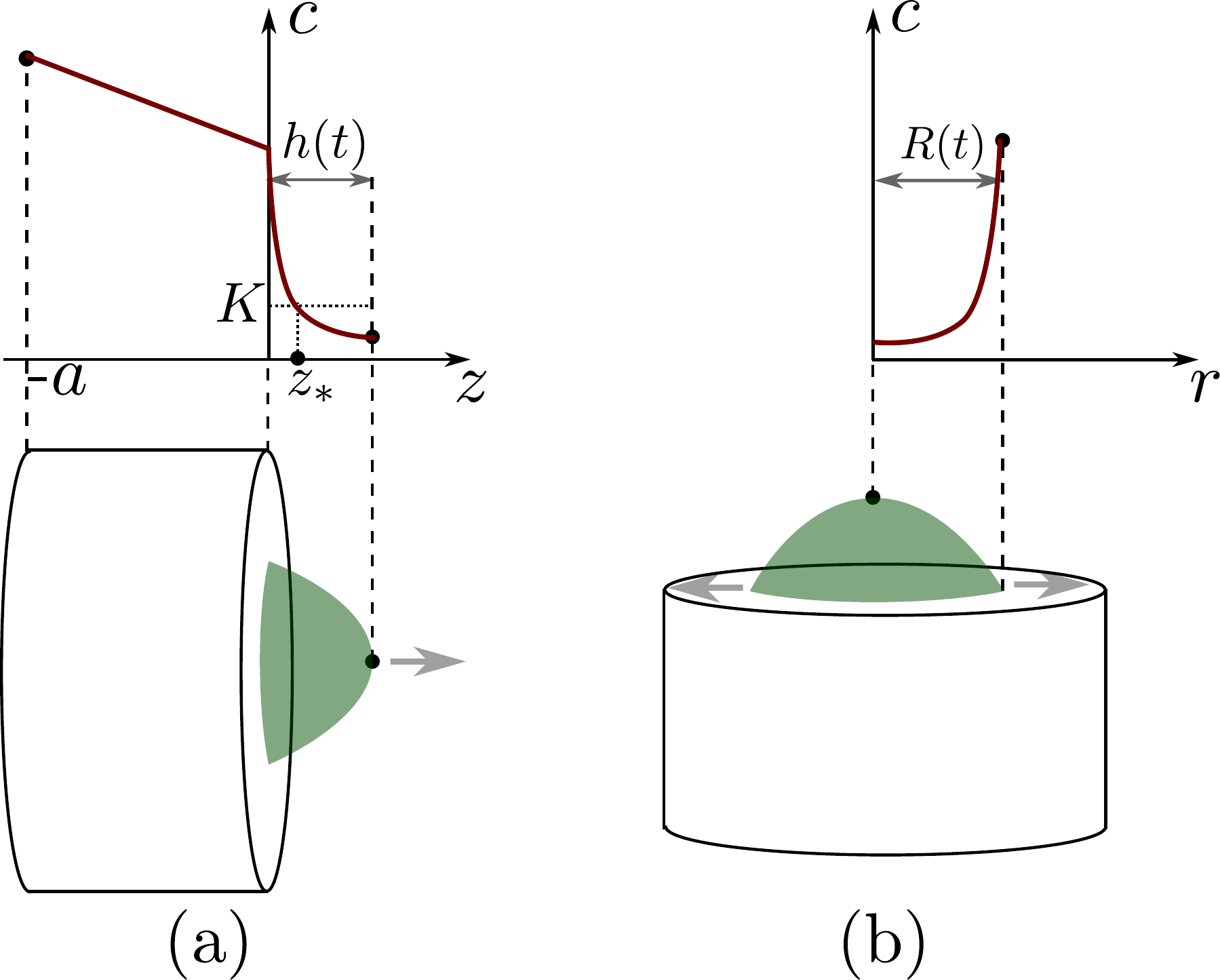}}
\caption{
Schematic of a colony-substrate configuration at some time $t$. (a) Definition of the 
colony height $h(t)$ and the level $z_\ast$ at which the concentration reaches the Monod 
constant~$K$. (b) Definition of the colony radius $R(t)$.} 
\label{fig:0}
\end{figure}

The rest of the paper is organized as follows: 
In section~\ref{s:ContinuumModel}, we describe the general, moving-boundary model in the cylindrically symmetric, 
three-dimensional setting
for the expansion of 
a bacterial colony and non-dimensionalize the model.  We also prove the existence, uniqueness, and some properties 
of the steady-state solution to the reaction-diffusion equation of the nutrient concentration with a fixed
configuration of colony.
In section~\ref{s:1D}, we construct a one-dimensional model for the vertical expansion of the colony, 
perform the related numerical simulations and asymptotic analysis, and analyze the nutrient concentration profile
and eventually the rate of vertical expansion.  
In section~\ref{s:disk}, we  construct a two-dimensional disk model for the radial expansion kinetics, perform numerical 
simulations and asymptotic analysis, and analyze the 
the nutrient profile and the asymptotic behaviors of the radial expansion.  
Finally, in section~\ref{s:Conclusion}, we draw conclusions of our work and discuss several issues for further studies.


\section{A three-dimensional model and steady-state nutrient concentrations}
\label{s:ContinuumModel}

\subsection{Model formulation}
\label{ss:3Dmodel}

We denote the system region that depends on time $t$ by $\Omega = \Omega(t)$.
This region consists of two parts. One 
is the colony region $\Omega_+ = \Omega_+(t)$ which expands with time
and the other is the substrate (i.e., agar) region $\Omega_-$ that is assumed to be fixed all the time;
cf.\ Figure~\ref{f:SystemRegion}. 
We shall assume the system is cylindrically symmetric.

\begin{figure}[h]
\begin{center}
\includegraphics[width=1.0\textwidth]{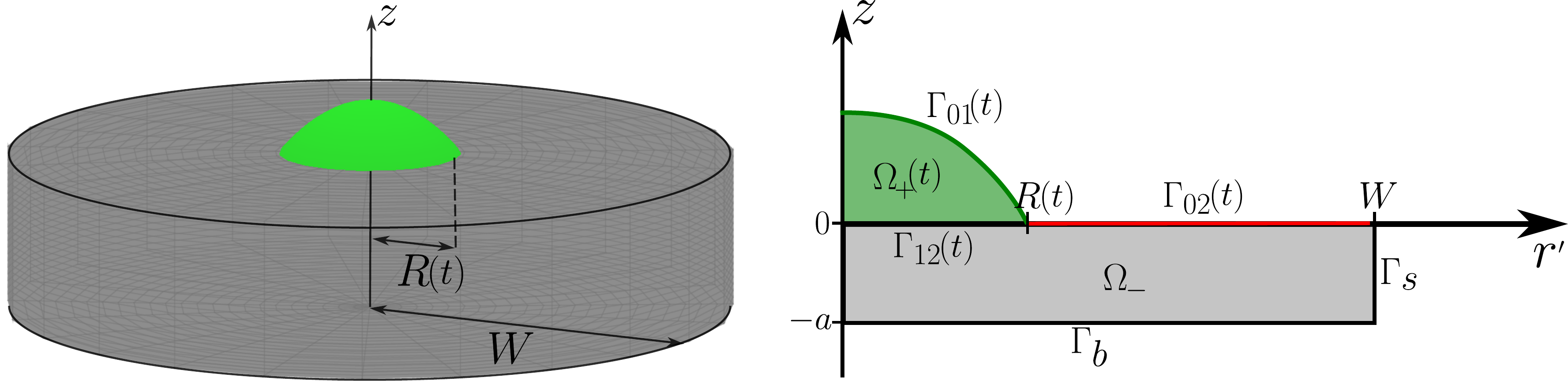}

$ \qquad \qquad \qquad $ (a) $\qquad \qquad \qquad \qquad \qquad  \qquad \qquad \qquad \qquad $  (b) $\qquad \qquad \qquad \qquad $
\end{center}
\caption{ 
{The geometry of 
the three-dimensional continuum model of an expanding bacterial colony.}
(a) The system region $\Omega = \Omega (t)$ at time $t$ consisting of 
the colony region $\Omega_+ = \Omega_+(t)$ (in green) and the substrate region $\Omega_-$ (in gray) which is fixed throughout.
(b) The same region in cylindrical coordinates $(r',z)$ with identification of the boundary portions $\Gamma_{01}(t)$, $\Gamma_{02}(t)$, $\Gamma_{12}(t)$, $\Gamma_{\rm b}$, and $\Gamma_{\rm s}$.}
\label{f:SystemRegion}
\end{figure}

 With the notations
\[
\br = (x, y, z) = (\br', z),  \quad \br' = (x, y), \quad r = | \br |=\sqrt{x^2+y^2+z^2}, 
 \quad r' = |\br'| = \sqrt{x^2+y^2}, 
\]
we define 
\begin{align*}
&\Omega_- = \{ \br = (\br', z): r'< W   \mbox{ and }  -a < z < 0\}, \\
&\Omega_+(t) = \{ \br = (\br', z) : 0 < z < h( r',  t)
\ \mbox{and} \ r'< R(t) \} \qquad \forall t > 0. 
\end{align*}
Here, $a$ and $W$ are two given positive numbers, and the function
$h = h(r', t) > 0$ is a $C^1$-function defined for $r' \in [0, R(t)]$ and $t > 0$ and satisfies 
the following properties: 
\[
\partial_{r'} h (0, t) = 0, \quad  \partial_{r'} h(r', t) < 0 \quad 
\forall r' \in (0, R(t)), \quad  h(R(t), t) = 0, \quad \mbox{and} \quad R'(t) > 0
\]
for all $t > 0.$
We denote the air-colony, air-agar, and 
colony-agar interfaces by $\Gamma_{01} = \Gamma_{01}(t)$, $\Gamma_{02}= \Gamma_{02}(t)$, 
and $\Gamma_{12} = \Gamma_{12}(t)$, respectively:
\begin{align*}
   & \Gamma_{01}(t) = \{ \br =  (\br', z): z = h(r', t) \mbox{ and } r' < R(t) \}, \\
   & \Gamma_{02}(t) = \{ \br = (\br', z): R(t) < r' < W \mbox{ and } z = 0 \}, \\
   & \Gamma_{12}(t) = \{ \br = (\br', z): r' < R(t) \mbox{ and } z = 0 \}. 
\end{align*}
The initial colony region $\Omega_+(0)$ and the initial agar-colony interface 
$\Gamma_{01}(0)$ are defined by
\begin{align*}
    &\Omega_+(0)  = \{ \br = (\br', z) : 0 < z < h_0( r')
\ \mbox{and} \ r'< R_0 \}, \\
& \Gamma_{01}(0) = \{ \br =  (\br', z): z = h_0(r') \mbox{ and } r' < R_0 \}, 
\end{align*}
where $R_0 > 0$ and $h_0 \in C^1([0, R_0])$ are given with 
$h_0(r')$ decreases from $r'=0$ to $r' = R_0$, $h_0'(0) = 0$, and $h_0 (R_0) = 0.$ 
We denote the lateral side and the bottom face of the boundary of agar region 
by $\Gamma_{\rm s}$ and $\Gamma_{\rm b}$, respectively: 
\begin{align*}
    & \Gamma_{\rm s} = \{\br = (\br', z): r' = W \mbox{ and } -a < z < 0\}, \\
    & \Gamma_{\rm b} = \{ \br = (\br', z): r' < W \mbox{ and } z = -a \}. 
\end{align*}

The basic components of the model include the normal velocity of the moving colony surface, 
the nutrient concentration that is governed by a reaction-diffusion equation, and the pressure
in the colony that is related to the normal velocity of the colony surface. 
Let us denote by $c = c(\br, t)$ the local concentration of
the nutrient at a spatial point $\br$ and time $t$. It is governed by the 
reaction-diffusion equation 
\begin{equation}
\label{RDE}
\partial_t c = \nabla \cdot D \nabla c - 
\frac{\rho}{Y}\chi_+ \lambda(c)  \qquad \forall \br \in \Omega(t) \  \mbox{ and } \ t > 0, 
\end{equation}
where $D$ is the diffusion coefficient defined by 
\begin{equation}
\label{D}
D = \left\{ 
\begin{aligned}
    & D_+  \quad &  & \mbox{in }  \Omega_+(t), \\
    & D_- \quad & & \mbox{in } \Omega_-,
\end{aligned}
\right. 
\end{equation}
and $\chi_+ = \chi_{\Omega_+(t)}$ is the indicator function of the region $\Omega_+(t)$. Here
and below, 
$\rho> 0$ is the constant cell density in the colony, $Y > 0$ is a constant yield factor, 
and $\lambda: [0, \infty) \to [0, \infty)$ is the local growth rate of cell mass given  by 
the Monod kinetic form as follows: 
\begin{equation}
\label{lambda}
\lambda(c) = \lambda_{\rm B} \frac{c}{c+K}, 
\end{equation}
where $\lambda_{\rm B} > 0$ is the constant batch culture growth rate of cell mass  and 
$K > 0$ is a constant, the Monod constant \cite{Warren_eLife2019}. 
Note that Eq.~\eqref{RDE} is equivalent to 
\begin{subequations}
\label{GeneralInterfacProb}
  \begin{empheq}[left=\empheqlbrace]{align}
    \label{Inteface+}
    & \partial_t c = D_+ \Delta c - \frac{\rho}{Y} \lambda(c) & &  \forall \br \in \Omega_+(t) \ \mbox{ and } \ \forall t > 0, \\
    \label{Interface-}
    & \partial_t c = D_- \Delta c  & &  \forall \br \in \Omega_- \ \mbox{ and }\  \forall t > 0, \\
    \label{InterfaceJump}
    & D_+ \partial_z c|_{ \{z >0\}}  = D_- \partial_z c|_{\{z <0\}} & &  \mbox{on } \Gamma_{12}(t) \ \mbox{ and }\  \forall t > 0. 
\end{empheq}
\end{subequations}
As in \cite{Warren_eLife2019}, we
 impose the boundary conditions for the nutrient concentration to be 
\begin{equation}
    \label{BC4c}
    c = c_{\rm s} \quad \forall \br \in \Gamma_{\rm s} \ \mbox{ and } \ t > 0 \qquad \mbox{and} \qquad 
\partial_n c = 0  \quad \forall \br \in \partial\Omega(t) \setminus \Gamma_{\rm s}, 
\end{equation}
where $c_{\rm s} > 0$ is a constant and $\partial_n$ denotes the normal derivative with $\bn$ the unit exterior normal at the 
boundary of $\Omega(t)$. The initial condition for the nutrient concentration is given by 
\begin{equation}
    \label{IC4c}
    c(\br, 0) = \left\{ 
    \begin{aligned}
    & 0 & &  \quad \forall \br \in \Omega_+(0), \\
    & c_{\rm ini} & & \quad \forall \br \in \Omega_-, 
    \end{aligned}
    \right.
\end{equation}
where $c_{\rm ini} > 0$ is a constant. 

The concentration $c = c(\br, t)$ determines the local growth rate of cell mass by \eqref{lambda}. 
Let us denote by
$\bv = \bv (\br, t)$ the local velocity at the spatial point $\br$ in the colony and time $t$. It follows 
from the mass conservation and the assumption that the cell density is a constant that 
\begin{equation}
\label{massconservation}
\nabla \cdot \bv = \lambda \qquad \forall \br \in \Omega_+(t) \ \mbox{ and }\  t > 0. 
\end{equation}
Invoking Darcy's law, we obtain 
\begin{equation}
\label{Darcy}
\bv = - \xi \nabla p \qquad \forall \br \in \Omega_+(t) \ \mbox{ and }\ t > 0, 
 \end{equation}
where $p  = p(\br, t)$ is the pressure at $\br \in \Omega_+$ and $t > 0$, and $\xi > 0$ is 
the friction constant. 
Combining \eqref{lambda}--\eqref{Darcy}, we obtain
\begin{equation}
    \label{pressure}
    \Delta p = - \frac{1}{\xi} \lambda(c) \qquad \forall \br \in \Omega_+(t) \ \mbox{ and } \ t > 0. 
\end{equation}
We impose the boundary condition 
\begin{equation}
\label{BC4pressure}
p = 0 \quad \forall \br \in \Gamma_{01}(t) \ \mbox{ and } \ t > 0 \qquad \mbox{and} \qquad 
\partial_n p = 0  \quad \mbox{on } \Gamma_{12}(t) \quad \forall t > 0. 
\end{equation}

The colony surface $\Gamma_{01}$ is moving with the normal velocity 
\begin{equation}
\label{Vn}
V_n = \bv \cdot \bn - \gamma H \qquad \forall \br \in \Gamma_{01}(t) \ \mbox{ and }\  t > 0,
\end{equation}
where $\bn$ is the unit normal at the surface $\Gamma_{01}(t)$ exterior to the colony $\Omega_+(t)$, 
$\gamma > 0$ is a constant which is the surface tension constant, and $H$ is the mean curvature of 
the surface. 
By the representation of the  surface $\Gamma_{01}(t)$ with the function $h$, we have the unit normal 
and the mean curvature 
\begin{align*}
    & \bn = \bn(\br', t) =  - \frac{(\nabla_{\br'} h(r',t), -1)}{\sqrt{1+ (\partial_{r'} h(r', t))^2}}, \\
    & 2 H = 2 H(\br', t) = \frac{1}{\sqrt{1+(\partial_{r'} h(r', t))^2 }}
    \left[  \frac{\partial_{r'}^2 h(r',t) }{1 + (\partial_{r'} h(r', t))^2 } 
    + \frac{\partial_{r'}h (r', t)}{r'}  \right]. 
\end{align*}
Consequently, by \eqref{Darcy}, we have 
 
\begin{equation}
{
\partial_t h = \xi(\partial_{r'}p\,\partial_{r'} h-\partial_z p )+\dfrac{\gamma}{2}\left[\dfrac{\partial_{r'}^2 h}{1+(\partial_{r'} h)^2}+\dfrac{\partial_{r'} h}{r'}\right]
}.\label{FinalVn2}
\end{equation}


\medskip 

{\it In summary}, the model consists of the reaction-diffusion equation \eqref{RDE} for the nutrient concentration 
$c = c(\br, t)$ together with the boundary conditions \eqref{InterfaceJump}, \eqref{BC4c} and initial conditions \eqref{IC4c}, the Laplace equation \eqref{pressure}
for the pressure $p = p(\br, t)$ together with the boundary conditions \eqref{BC4pressure}, and the motion law 
\eqref{Vn} or \eqref{FinalVn2} for the normal velocity of the moving colony surface. 

\begin{rmk}\label{rmk:edge-condition}
For a fixed domain $\Omega$, the problems for concentration $c$ and pressure $p$ are
well posed, as we show later in this section. For the moving-boundary problem, the formulation above is incomplete 
and one would not expect uniqueness, and possibly not even a closed evolution law for $\Omega(t)$. 
An additional equation, that is for the contact line, where the colony domain boundary meets the substrate (or, equivalently, an equation for $R(t)$) is 
possibly needed. 
Such an equation arises in the study of capillarity problems \cite{finn1986,huh1971hydrodynamic}. Defining
such an equation is known  to be  subtle as it needs to be consistent with the equation for the normal 
velocity for the colony surface \eqref{Vn} and also creates no singularities. In this work, we will focus on the analysis of reduced models in Sections~\ref{s:1D} and \ref{s:disk},  where this issue does not exist.     

\end{rmk}

We now derive a  dimensionless form of our three-dimensional (3D) model. 
Let us introduce the  dimensionless variables
\[
\hat{\boldsymbol{r}} = \dfrac{\boldsymbol{r}}{\mathcal{L}}, \qquad \hat{t}=\dfrac{t}{\mathcal{T}}, \qquad \hat{c}=\dfrac{c}{K}, \qquad 
\hat{p}=\dfrac{p}{\mathcal{P}},
\]
where $\mathcal{L}$, $\mathcal{T}$, and $\mathcal{P}$ are the characteristic scales that are determined below. 
Rewrite  equations \eqref{Inteface+}, \eqref{Interface-}, \eqref{InterfaceJump}, \eqref{pressure}, and \eqref{Vn}  in these 
variables:
\begin{align*}
    & \dfrac{K}{\mathcal{T}}\partial_{\hat{t}} \hat{c} = \dfrac{D_+K}{\mathcal{L}^2} \Delta_{\hat{\boldsymbol{r}}} \hat{c} - \frac{\rho\lambda_B}{Y} \dfrac{\hat{c}}{\hat{c}+1},\\
    & \dfrac{K}{\mathcal{T}}\partial_{\hat{t}} \hat{c} = \dfrac{D_-K}{\mathcal{L}^2} \Delta_{\hat{\boldsymbol{r}}} \hat{c},   \\
    &  \dfrac{KD_+}{\mathcal{L}} \partial_{\hat{z}} \hat{c}|_{ \hat{z} > 0}  =\dfrac{KD_+}{\mathcal{L}} \partial_{\hat{z}} \hat{c}|_{\hat{z} < 0},\\
    &\dfrac{\mathcal{P}}{\mathcal{L}^2}\Delta_{\hat{\boldsymbol{r}}}\hat{p} = - \dfrac{\lambda_B}{\xi} \dfrac{\hat{c}}{\hat{c}+1},\\
    &
    \dfrac{\mathcal{L}}{\mathcal{T}}\hat{V}_n = -\dfrac{\xi \mathcal{P}}{\mathcal{L}} \nabla_{\hat{\boldsymbol{r}}} \hat{p} \cdot \hat{\boldsymbol{n}} - \dfrac{\gamma}{\mathcal{L}} \hat{H},
\end{align*}
where $\hat{\boldsymbol{n}}$ is defined using the dimensionless variables. 
The regions $\hat{\Omega}_+(\hat{t})$ and $\hat{\Omega}_-$, and the boundaries 
$\hat{\Gamma}_{01}(\hat{t}\,)$, $\hat{\Gamma}_{02}(\hat{t}\,)$, $\hat{\Gamma}_{12}(\hat{t}\,)$, 
$\hat{\Gamma}_{\rm b}$, and $\hat{\Gamma}_{\rm s}$ are defined similarly. 
For the boundary and initial conditions, we set 
\[
\hat{c}_{\rm s} = \frac{c_{\rm s}}{K}, \qquad 
\hat{c}_{\rm ini}  = \frac{c_{\rm ini}}{K}, \qquad 
\hat{R}_0 = \frac{R_0}{\mathcal L}, \qquad 
\hat{h}_0 = \frac{h_0}{\mathcal L}, 
\]
with the  characteristic scales 
\[
\mathcal{L}=\sqrt{\dfrac{YKD_+}{\rho \lambda_B}}, \qquad \mathcal{P} = \dfrac{\lambda_B\mathcal{L}^2}{\xi}=\dfrac{YKD_+}{\rho \xi}, \qquad \mathcal{T}=\dfrac{1}{\lambda_B}.
\]
Introducing
\[
\eta = \frac{\mathcal{L}^2}{\mathcal{T}D_+} = \frac{YK}{\rho}, \qquad 
d_0=\frac{D_-}{D_+},\qquad  \zeta =  \frac{\gamma \mathcal{T}}{\mathcal{L}^2}, 
\]
and dropping all the hats, we obtain
the dimensionless form of the system
\begin{subequations}\label{eq:nondimensional_model}
\begin{empheq}[left=\empheqlbrace]{align}
    & \eta\partial_{{t}} {c} =\Delta {c} -  \dfrac{{c}}{{c}+1}\quad \text{ in }\Omega_{+}(t),\\
    & \eta\partial_{{t}} {c} = d_0 \Delta {c}\quad \text{ in }\Omega_{-},  \\
    & \qquad  \partial_{{z}} {c}|_{ \{z > 0\}}  =d_0 \partial_{{z}} {c}|_{\{z < 0\}} \quad  \text{ on }\Gamma_{12}(t),\\
    & \qquad c = c_{\rm s} \ \text{ on }\Gamma_{\rm s} \quad \text{ and }\quad \partial_n c = 0\ \text{ on }\partial\Omega(t)\setminus \Gamma_{\rm s},\\
    &\qquad c (t=0)= 0 \ \text{in } \Omega_- \quad \mbox{ and } \quad c (t=0)= c_{\rm ini} \ \mbox{in } \Omega_+(0), 
    \\
    &\Delta{p} = -  \dfrac{{c}}{{c}+1} \quad \text{ in }\Omega_+(t),\label{nondim:p-pde} \\
    & \qquad p=0 \ \text{ on }\Gamma_{01}(t) \quad \text{ and }\quad \partial_n p = 0 \ \text{ on }\Gamma_{12}(t), \label{nondim:p-bc}
    \\&
    {V}_n = -\partial_n  {p}  - \zeta {H}\quad \text{ on }\Gamma_{01}(t), \label{nondimVn}\\
    & \qquad \Gamma_{01}(0) = \{ (r', z): z = h_0(r') \ \mbox{ and } \ 0 \le r' < R_0\}. \label{nondimInitialG01}
\end{empheq}
\end{subequations}

\subsection{Steady-state nutrient concentrations}
\label{s:SteadyState}

In this subsection, 
we study the boundary-value
problem of the steady-state reaction-diffusion equation (cf.\ \eqref{RDE} and 
its dimensionless form in \eqref{eq:nondimensional_model}). 
Solutions of such problems represent steady-state nutrient concentrations and are often used to model bacterial colony expansion under the assumption that cell division is negligibly slow.
We state our results in the three-dimensional setting but it is readily verified that
these results also hold true for lower dimensional cases. We shall also consider the 
dimensionless form \eqref{eq:nondimensional_model}.

For simplicity, we denote the agar-colony interface by 
$\Gamma = \Gamma_{12}$. The agar and colony regions are $\Omega_-$ and $ \Omega_+$, respectively. 
We divide the boundary $\partial\Omega$ into two disjoint parts $\Gamma_{\rm D}=\Gamma_{\rm s}$
and  $\Gamma_{\rm N}=\partial \Omega \setminus \Gamma_{\rm s}$ with $\Gamma_{\rm D} \subset 
\partial \overline{\Omega}_-.$
Both $\Gamma_D$ and $\Gamma_N$ are assumed to be smooth and have positive surface measure. 
We consider the 
boundary-value problem
\begin{subequations}
\label{SimpleBVP}
\begin{empheq}[left=\empheqlbrace]{align}
\label{eq:simple}
& \nabla \cdot D_\Gamma \nabla  c - \chi_+ 
\lambda(c) = 0   \quad \text{in }\Omega,\\
\label{eq:simpleBC}
& c = c_0 \quad \text{on } \Gamma_{\rm D} \quad \mbox{and} \quad 
\partial_n c = 0 \quad \mbox{on } \Gamma_{\rm N},
\end{empheq}
\end{subequations}
where the diffusion coefficient is $D_\Gamma = d_0+\chi_+(1-d_0)$ with $\chi_+ = \chi_{\Omega_+}$, 
\begin{equation}
\lambda(c)=\dfrac{c}{1+c}\label{nondim-def-lambda}; 
\end{equation}
cf.\ Figure~\ref{f:lambdaLambda} (Left), 
$c_0\in H^1(\Omega)$ is a given  function 
such that $c_0 \ge 0$ in $\Omega$, 
and  $\{ \br \in \Gamma_{\rm D}: c_0(\br) > 0\}$ has a positive surface measure. 

\begin{figure}[H]
\begin{center}
\includegraphics[width=0.75\textwidth]{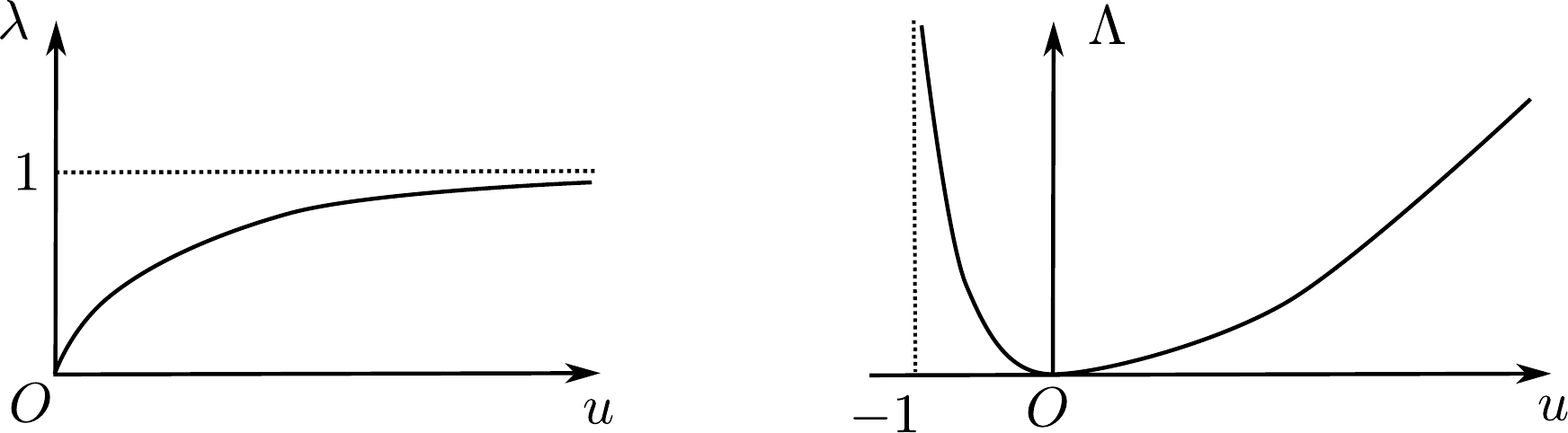}
\end{center}

\vspace{-4 mm}

\caption{The graph of the function $\lambda = \lambda(u)$ (Left) and $\Lambda = \Lambda (u)$ (Right)
 defined in \eqref{nondim-def-lambda} and \eqref{Fc}, respectively. 
}
\label{f:lambdaLambda}
\end{figure}

We denote
\begin{align*}
&\mathcal{A} = \{u \in H^1(\Omega): u \ge 0 \ \mbox{a.e. } \Omega
\mbox{ and } u = c_0 \mbox{ on } \Gamma_{\rm D} \}, \\
& \mathcal{A}_0 = \{ \xi \in H^1(\Omega): \xi  = 0 \ \mbox{ on } \ \Gamma_{\rm D} \}. 
\end{align*}
Here and below, the boundary value is understood as the trace. 
We call $c \in \mathcal{A}$ a weak solution to the boundary-value problem \eqref{eq:simple} if 
\begin{equation}
    \label{eq:weak}
\int_\Omega \left[ D_\Gamma \nabla c \cdot \nabla \xi + \chi_+ 
\lambda (c) \xi \,\right] \text{d}\boldsymbol{r} = 0
\qquad \forall \xi \in \mathcal{A}_{0}. 
\end{equation}
 Formally, 
the equation in \eqref{eq:simple} is the Euler--Lagrange equation of the functional 
$\mathcal{E}: \mathcal{A} \to \R $ defined by 
\begin{equation}
\mathcal{E}[u]=\int_\Omega\left[ \dfrac{D_\Gamma}{2}|\nabla u|^2 + \chi_+ 
\Lambda (u) \right] d\br, 
\label{eq:energyPre}
\end{equation}
where 
\begin{equation}
    \label{Fc}
    \Lambda (u) = u - \ln (1 + u) \qquad \forall u > -1; 
\end{equation}
cf.\ Figure~\ref{f:lambdaLambda} (Right). 
Note that that $\Lambda \in C^\infty( (-1, \infty))$ and it satisfies
\begin{align*}
&\Lambda'(u) = \lambda(u) > 0 \quad \mbox{and} \quad 
 \Lambda''(u) = \frac{1}{(u+1)^2 } > 0 \quad \forall u > -1,  \\
 & \Lambda(0) = 0 < \Lambda(u) \quad \forall u > -1 \ \mbox{ and } \ u \ne 0, \\
 &\Lambda (-1^+) = \lim_{u\to -1^+} \Lambda(u) = +\infty, 
 \quad \Lambda(+\infty) = \lim_{u \to +\infty} \Lambda(u) = +\infty, \quad \mbox{and} \quad
  \lim_{u\to +\infty} \frac{\Lambda(u)}{ u} = 1. 
\end{align*}

\begin{theorem}
    \label{t:3Dminimizer}
    Let $\Omega$ be a bounded domain in $\R^2$ or $\R^3$ with a Lipschitz-continuous boundary $\partial \Omega.$ 
    Assume a disjoint union $\Omega = \Omega_+ \cup \Omega_- \cup \Gamma$ with $\Omega_\pm$ bounded domains and 
    $\Gamma = \overline{\Omega_+} \cap \overline{\Omega_-}$ Lipschitz-continuous. 
    Assume also $\partial \Omega$ is the disjoint union of $\Gamma_D$ and $\Gamma_N$ each of which is 
    Lipschitz-continuous and has positive surface measure, and $\Gamma_{\rm D} \subset \overline{\Omega}_-.$
    There exists a unique $\hat{c} \in \mathcal{A}$ that minimizes $\mathcal{E}: \mathcal{A} \to \R$.
    Moreover, 
    the following statements are equivalent: 
    \begin{compactenum}
        \item[\rm (i)] $c\in \mathcal{A}$ is the minimizer of $\mathcal{E}: \mathcal{A} \to \R$; 
        \item[\rm (ii)] $c\in \mathcal{A} $ is the unique weak solution defined by \eqref{eq:weak}; and 
        \item[\rm (iii)]  
        $c\in \mathcal{A}$ satisfies $ c|_{\Omega_\pm} \in H^2({\Omega_\pm})\cap
        C^\infty(\Omega_\pm)$ and 
        $c$ is the unique solution to the following elliptic interface problem
   \begin{subequations}\label{eq:elliptic-interface-problem}
    \begin{empheq}[left=\empheqlbrace]{align}
       \label{D+Dc}
        &\Delta c - \lambda(c) = 0   &&\mbox{ in } \Omega_+,  \\
        \label{D-Dc}
        & d_0 \Delta c  = 0   && \mbox{ in } \Omega_-,  \\ 
        \label{cInterface}
        & \qquad  \partial_z c|_{\Omega_+} = d_0 \partial_z c|_{\Omega_-} \mbox{ on } \Gamma, \\
        \label{GammaDN}
        & \qquad c = c_0  \ \text{ on } \Gamma_{\rm D} \quad \mbox{and} \quad   
\partial_n c = 0 \ \mbox{ on } \Gamma_{\rm N}.
    \end{empheq}
    \end{subequations}
    \end{compactenum}
In particular, the minimizer $\hat{c}$ of $\mathcal E$ over $\mathcal A$ is positive point-wise in 
$\Omega_\pm.$
\end{theorem}

\begin{proof} We divide our proof into three steps. 

\smallskip


{\it Step 1. Proof of the existence and uniqueness of a minimizer.}
It follows from the properties of the function $\Lambda: 
(-1, \infty) \to \R$ and 
Poincar{\'e}'s inequality  that there exist
constant $K_1 > 0$ and $K_2 \in \R$ such that 
\begin{equation}
    \label{boundE}
     \mathcal{E}[u] \ge K_1 \| u\|_{H^1(\Omega)}^2 + K_2 \qquad \forall u \in \mathcal{A}. 
\end{equation}
    Let $\beta = \inf_{\mathcal{A}} \mathcal{E} \in \R$ and $c_k\in \mathcal{A}$ $ (k=1,2,\dots)$ with 
    $\mathcal{E}[c_k] \to \beta.$ By \eqref{boundE} and the Sobolev compact embedding
    \cite{GilbargTrudinger98,Adams75,EvansBook2010},
    up to a subsequence that is not relabeled, 
    $c_k \rightharpoonup \hat{c}$ (weak convergence) in $H^1(\Omega)$, 
    $c_k \to \hat{c}$ in $L^2(\Omega)$, and $c_k \to \hat{c}$ a.e.\ in $\Omega$, for some 
    $\hat{c} \in H^1(\Omega).$ Clearly, $\hat{c} \in \mathcal{A}$. Now, for each $k \ge 1,$
    \begin{align*}
    &|\nabla c_k|^2 = |\nabla c_k - \nabla \hat{c}|^2 + 2 \nabla c_k \cdot \nabla \hat{c}   - |\nabla \hat{c}|^2 
    \ge 2 \nabla c_k \cdot \nabla \hat{c}  - |\nabla \hat{c}|^2, \\
    & |\Lambda (c_k) - \Lambda (\hat{c})|\le 
    |c_k - \hat{c}|.
    \end{align*}
    Consequently, 
$
\liminf_{k\to \infty} \mathcal{E}[c_k] \ge \mathcal{E}[\hat{c}]. 
$
Hence $\hat{c}$ is a minimizer of 
$\mathcal{E}: \mathcal{A} \to \R$. The uniqueness of such a minimizer follows from the fact
that the functional $\mathcal{E}$ is strictly convex. 




{\it Step 2. Proof of that statements {\rm (i), (ii),} and {\rm (iii)} are equivalent.} 
Let us denote 
\[
\tilde{\mathcal{A}} = \{ u \in H^1(\Omega): u > 
-1 \mbox{ a.e. } \Omega \mbox{ and } u= c_0 \mbox{ on }
\Gamma_{\rm D} \}
\]
and extend $\mathcal{E}$ to be a functional (using the same notation) $\mathcal{E}: \tilde{\mathcal{A}} \to 
\R \cup \{ + \infty \}.$
Clearly, $\mathcal{A} \subset \tilde{\mathcal{A}}.$ Moreover, for any $ u \in \tilde{\mathcal{A}}$, 
$u^+ = \max(u, 0) \in \mathcal{A}$, and $\mathcal{E}[u] \ge \mathcal{E}[u^+]$, since 
$\Lambda (u) \ge \Lambda (u^+)$ in $\Omega. $ Therefore, 
the minimizer $\hat{c} \in \mathcal{A}$ of $\mathcal{E}$ over $\mathcal{A}$ is also a minimizer of
$\mathcal{E}$ over $\tilde{\mathcal{A}}.$


(i) $\Longrightarrow$ (ii).  Assume $c = \hat{c}$ is the minimizer of $\mathcal{E}$ over $\mathcal{A}.$ 
Let $\xi \in C^1(\overline{\Omega}) \cap \mathcal{A}_{0}.$
Then, $ \hat{c} + t \xi \in \tilde{\mathcal{A}}$ for any $t \in \R$ with $|t| \ll 1.$
Since $\hat{c}\in \mathcal{A}$ minimizes $\mathcal{E}$ over $\mathcal{A}$ and hence over 
$\tilde{\mathcal{A}},$ we have 
$\sigma'(0) = 0$, where $\sigma(t) = \mathcal{E}[c+t \xi]$ $(|t| \ll 1).$ Note that 
\[
\left| \ln \left(1+\hat{c}+ t \xi  \right) \right| 
\le \hat{c} + | \xi | \quad \mbox{a.e.\ in } \Omega \quad \mbox{if }  | t | \ll 1
\]
and that $\Lambda'(u) = \lambda (u)$ if $u > -1$.
Consequently, by direct calculations using
the Lebesgue Dominated Convergence Theorem, we obtain the equation \eqref{eq:weak}
with $\hat{c}$ replacing $c$ and $\xi\in C^1(\overline{\Omega}) \cap \mathcal{A}_{0}.$
Since $C^1(\overline{\Omega})\cap \mathcal{A}_{0}$  is dense 
in $\mathcal{A}_{0}$ with respect to the $H^1(\Omega)$-norm 
and $|\lambda (\hat{c})|\le 1$
in $\Omega$, we obtain
\eqref{eq:weak} with $\hat{c}$ replacing $c$. Hence (ii) is true. 


(ii) $\Longrightarrow$ (i). Assume $c \in \mathcal{A}$ satisfies \eqref{eq:weak}. Since the minimizer 
$\hat{c}$ of $\mathcal{E}: \mathcal{A} \to \R$ also satisfies \eqref{eq:weak} with $\hat{c}$
replacing $c$, we have that $u := c - \hat{c} \in \mathcal{A}_{0}$. Setting 
the test function $\xi = u\in \mathcal{A}_{0}$ in the weak formulation \eqref{eq:weak} for both $c$ and $\hat{c}$ we obtain that 
\[
\int_\Omega \left[ D_\Gamma |\nabla u|^2 + \chi_+ 
( \lambda(c) - \lambda (\hat{c}) ) u \right] \text{d}\boldsymbol{r} = 0.
\]
Since 
\[
\lambda (c) - \lambda (\hat{c}) = \frac{u}{(c+1) (\hat{c}+1)}
\]
and both $c$ and $\hat{c}$ are nonnegative in $\Omega$, 
we infer $\nabla u = 0$ a.e.\ in $\Omega$. 
Hence $ u= 0$ and $c = \hat{c}.$ Thus, (i) is true. 



(ii) $\Longrightarrow$ (iii). Suppose $c \in \mathcal{A}$ satisfies \eqref{eq:weak}. By considering 
test functions $\xi \in \mathcal{A}_{0}$ that are supported in $\Omega_+$, 
we see that \eqref{D+Dc} holds true in the weak sense. 
By the bootstrapping argument and the elliptic regularity theory \cite{GilbargTrudinger98,EvansBook2010}, 
we see that $c|_{\Omega_+}
\in H^2(\Omega_+)\cap C^\infty(\Omega_+)$ and \eqref{D+Dc} holds true point wise. Similarly, $c|_{ \Omega_-}
\in H^2(\Omega_-)\cap C^\infty(\Omega_-)$ and \eqref{D-Dc} holds true. 

Consider any $\xi \in  \mathcal{A}_{0}$. Choose the unit normal $\boldsymbol{n}$ on the interface  $\Gamma$ such that it points 
from $\Omega_-$ to $\Omega_+$, and it is outward for $\Gamma_{\text{D}}$ and $\Gamma_{\text{N}}$. The corresponding normal derivative is denoted by $\partial_{n}$.
By \eqref{eq:weak} we have 
\begin{eqnarray}
    0 & =& \int_{\Omega_+}  \nabla c \cdot \nabla \xi \, \text{d}\boldsymbol{r} + 
    \int_{\Omega_-} d_0 \nabla c \cdot \nabla \xi \, \text{d}\boldsymbol{r}+ \int_{\Omega_+} 
    \lambda(c) \xi \, \text{d}\boldsymbol{r}\nonumber \\
    &=& - \int_{\Omega_+} \left[ \Delta c - 
    \lambda(c) \right] \, \xi \, \text{d}\boldsymbol{r} 
    - \int_{\Omega_-} d_0 \Delta c \, \xi \, \text{d}\boldsymbol{r} 
    -\int_\Gamma \left(\partial_n c|_{\Omega_+}- d_0 \partial_n c|_{\Omega_-} \right) \xi \, \text{d}S
    \nonumber
  \\
    &&  
    + \int_{\Gamma_{\rm N}\cap \partial \Omega_+ } \partial_n c|_{\Omega_+} \xi\, \text{d}S 
    + \int_{\Gamma_{\rm N} \cap \partial \Omega_-} d_0 \partial_n c|_{\Omega_-} \xi\, \text{d}S. \label{variation-weak-solution}
\end{eqnarray}
 Note that we used the fact that $\partial_n c|_{\Omega_\pm} \in L^2(\partial \Omega_\pm)$, since 
$D_\Gamma \nabla c \in L^2(\Omega)$ and $c \in H^1(\Omega)$ \cite{Temam84}.
In view of \eqref{D+Dc} and \eqref{D-Dc}, the first two terms 
on the right-hand side of \eqref{variation-weak-solution} vanish. 
Next, by choosing $\xi \in C_0^1(\Omega) \subset \mathcal{A}_{0}$, we obtain
\[
\int_\Gamma \left( \partial_n c|_{\Omega_+} - d_0 \partial_n c|_{\Omega_-} \right) \xi \, \text{d}S = 0.
\]
This implies \eqref{cInterface}, since $\xi$ is arbitrary. Consequently, 
\[
    \int_{\Gamma_{\rm N}\cap \partial \Omega_+ } \partial_n c|_{\Omega_+} \, \xi\, \text{d}S 
    + \int_{\Gamma_{\rm N} \cap \partial \Omega_-} d_0 \partial_n c|_{\Omega_-} \, \xi \, \text{d}S = 0
    \qquad \forall \xi \in \mathcal{A}_{0}. 
\]
This leads to $\partial_n c = 0$ on $\Gamma_{\rm N}.$ Thus, (iii) holds true. 


(iii) $\Longrightarrow$ (ii). This follows  directly from \eqref{variation-weak-solution} and equations in 
\eqref{eq:elliptic-interface-problem}.

{\it Step 3. The positivity of the minimizer.} This follows from (iii), the fact that $c_0 > 0$, 
and the strong maximum principle
\cite{GilbargTrudinger98}.
\end{proof}

\section{A reduced model for vertical expansion}
\label{s:1D}

\subsection{Model formulation}
\label{ss:1Dmodel}

To elucidate the dynamics of vertical growth, we introduce a simplified version of the original 
three-dimensional model. We assume that the system
region at time $t$ is the interval $(-a, h(t))$ along the $z$-axis, 
where $a > 0$, a given constant, is the agar thickness 
and $h(t)$ is the colony height; cf.\ Figure~\ref{fig:0} (a). 
The colony region, agar region, and agar-colony interface are 
$ \Omega_+(t) = (0, h(t))$, $\Omega_- = (-a, 0)$, and $  \Gamma = \{ 0 \}, $ respectively. 
The reaction and diffusion of nutrient is much faster than the increasing of the colony height $h(t)$.  
Hence, we consider quasi-stationary dynamics neglecting the terms
$\partial_tp$ and $\partial_tc$ in equations for the 
pressure $p$ and nutrient concentration $c$, respectively.
Note that the functions $c$ and $p$ still depend on time $t$ due to the evolving colony height $h(t).$
Taking into account all these assumptions along with repeating model derivation arguments from Section~\ref{ss:3Dmodel}, we obtain the following one-dimensional system
of equations for the colony height $h(t)$, the nutrient concentration $ c = c(z, t)$,
and the pressure $p = p(z, t)$, and the corresponding boundary and initial conditions: 
\begin{subequations}\label{1dmodel-with-dimensions}
\begin{empheq}[left=\empheqlbrace]{align}
\label{QuasiS1D+}
& D_+ \partial_{zz}  c  =  \frac{\rho \lambda_B}{Y} \, \dfrac{c}{c+K}  & & 
\mbox{ for } z \in (0, h(t)) \ \mbox{ and } \  t > 0,  \\
\label{QuasiS1D-}
&D_- \partial_{zz} c = 0 & &  \mbox{ for } z \in (-a, 0) \ \mbox{ and } \ t > 0,  \\
\label{QuasiS1Djump}
& \quad D_- \partial_z c(0^-, t) = D_+ \partial_z c(0^+, t) & & \mbox{ for } t > 0, \\
\label{1DBC}
    & \quad c(-a, t) = c_0 \quad \mbox{and} \quad  \partial_z c(h(t), t) = 0 & &
    \mbox{ for } t > 0, \\
& \partial_{zz}p = -\dfrac{\lambda_B}{\xi}\,\dfrac{c}{c+K} && 
\mbox{ for } z \in (0, h(t))  \mbox{ and } \  t > 0. \label{1DPDE-for-p} \\
& \quad \partial_z p(0, t) = 0 \quad \mbox{and} \quad  p(h(t), t) = 0 & & 
\mbox{ for } t > 0, \label{1DBC-p}\\ 
& h'(t) = -\xi \partial_z p(h(t),t) & & \mbox{ for } t>0, 
\label{hprimelaw}
\\
& \quad h(0) = h_0, & & \label{1Dh0}
\end{empheq}
\end{subequations}
where $h_0 > 0$ is a given initial height. 
Here, all the parameters $D_{\pm}$, $\rho$, $Y$, $\lambda_B$, $K$, $\xi$, and $a$ have the same meaning as for the three-dimensional model described in Section~\ref{ss:3Dmodel}. We note that a related simpler model was previously introduced in \cite[Appendix 2.3]{Warren_eLife2019} to study the properties of the nutrient concentration profile in the formal large-time limit.


We introduce the dimensionless variables and parameters
\[
\hat{z}=\dfrac{z}{\mathcal{Z}}, \quad \hat{a}=\dfrac{a}{\mathcal{Z}}, \quad 
\hat{h}=\dfrac{h}{\mathcal{Z}}, \quad \hat{h}_0 = \frac{h_0}{\mathcal Z}, 
\quad \hat{t}=\dfrac{t}{\mathcal{T}}, \quad \hat{c}=\dfrac{c}{K}, \quad \hat{p}=\dfrac{p}{\mathcal{P}}, 
\quad d_0 = \frac{D_-}{D_+},
\]
where the characteristic scales are defined, consistent with scales in Subsection~\ref{ss:3Dmodel}, as 
\[
\mathcal{Z}=\sqrt{\dfrac{KD_+}{\alpha \lambda_B}}, 
\qquad \mathcal{P} = \dfrac{\lambda_B\mathcal{Z}^2}{\xi}, \qquad \mathcal{T} = \dfrac{1}{\lambda_B}.
\]
Similar to the non-dimensionalization in Subsection~\ref{ss:3Dmodel}, we obtain by dropping hats the dimensionless
system 
\begin{subequations}\label{1dmodel-intermediate-nondim}
\begin{empheq}[left=\empheqlbrace]{align}
\label{Interm_c+-nondim}
& \partial_{zz}  c  =  \dfrac{c}{c+1}  & & 
\mbox{for } z \in (0, h(t)) \ \mbox{ and } \  t > 0,  \\
\label{Interm_c--nondim}
& \partial_{zz} c = 0 & &   \mbox{for } z \in (-a, 0) \ \mbox{ and } \ t > 0,  \\
\label{Interm_c_Interface-nondim}
& \qquad d_0\partial_{z} c(0^-, t) = \partial_{z} c(0^+, t) & &  \mbox{for } t > 0, \\
\label{Interm_c_BC-nondim}
    &\qquad  c(-a, t) = c_0 \quad \mbox{and} \quad  \partial_{z} c(h(t), t) = 0 & &  \mbox{for } t > 0, \\
& \partial_{zz}p = -\dfrac{c}{c+1} &&  \mbox{for } z \in (0, h(t)) \ \mbox{ and } \  t > 0, \label{Interm-p-nondim} 
\\ 
 & \qquad \partial_{z}p(0, t) = 0 \quad \mbox{and} \quad  p(h(t), t) = 0 &&  \mbox{for } t > 0, \label{Interm-p-BC-nondim}\\
& h'(t) = - \partial_{z} p(h(t),t) & &  \mbox{for } t>0, \label{Interm_h-nondim}\\
& \qquad h(0) = h_0. & & \label{Interm_h0-nondim}
\end{empheq}
\end{subequations}


It follows from \eqref{Interm_h-nondim}, \eqref{Interm-p-BC-nondim}, \eqref{Interm-p-nondim}, and \eqref{Interm_c+-nondim} that
\begin{align}
\partial_t h &= - \partial_zp(h(t),t)= - \int_0^{h(t)}\partial_{zz}p(z,t)\,\text{d}z \nonumber 
\\
&= \int_0^{h(t)} \frac{c}{c+1} \, dz 
=\int_0^{h(t)}\partial_{zz}c(z,t) \,\text{d}z = - \partial_{z}c(0^+,t). \label{simplification-h}
\end{align}
By \eqref{Interm_c--nondim} and  \eqref{Interm_c_BC-nondim}, we have 
\begin{equation}
\label{cztagar}
c(z,t) = \dfrac{1}{a} (c(0,t)- c_0)z + c(0,t) \quad \mbox{for } -a<z<0 \ \mbox{and} \ t > 0
\end{equation}
Denoting $\kappa = d_0/a$, we can then rewrite the interface condition 
\eqref{Interm_c_Interface-nondim} as
\begin{equation}
\partial_zc(0^+,t) = \kappa (c(0,t)-c_0) \quad \mbox{for } t > 0.
\label{simplification-c}
\end{equation}
We can now eliminate the pressure variable and reduce the dimensionless system above into the following: 
\begin{subequations}\label{vfbp-nondim-expanded}
\begin{empheq}[left=\empheqlbrace]{align}
\label{1dsemifinal-c-expanded}&\partial_{zz}c=\dfrac{c}{c+1}& & \ \mbox{for } z\in (0,h(t)) \mbox{ and } t>0,\\
\label{1dsemifinal-c---expanded}& \partial_{zz}c = 0 & & \ \mbox{for } z\in(-a,0) \mbox{ and } t>0,\\
\label{1dsemifinal-BC_at_0-expanded}&\qquad \partial_z c(0^+,t) = d_0\partial_zc(0^-,t) & & \ \mbox{for } t>0, \\
\label{1dsemifinal-BC_at_a-expanded} & \qquad c(-a,t)=c_0 \quad \mbox{and} \quad \partial_z c(h(t),t)=0 & &  \ \mbox{for } t > 0, \\
\label{1dsemifinal-h-expanded}& h' (t) = - \kappa (c(0,t)-c_0) & &\ \mbox{for } t>0, \\
\label{1dsemifinal-h0} & \qquad h(0) = h_0. & & 
\end{empheq}
\end{subequations}
By \eqref{cztagar} and \eqref{simplification-c}, this system is equivalent to the following system for
the concentration $c = c(z,t)$ only for $z \in (0, h(t))$: 
\begin{subequations}\label{vfbp-nondim}
\begin{empheq}[left=\empheqlbrace]{align}
\label{1dfinal-c}&\partial_{zz}c=\dfrac{c}{c+1}& & \ \mbox{for } z\in (0,h(t))\mbox{ and } t>0,\\
\label{1dfinal-BC_at_0}&\qquad \partial_zc(0^+,t) = \kappa (c(0,t)-c_0),& &\ \mbox{for } t>0, \\
\label{1dfinal-BC-at-h}&\qquad \partial_zc(h(t),t)=0 & & \ \mbox{for } t>0,\\
\label{1dfinal-h}& h' (t)= - \kappa (c(0,t)-c_0) & & \ \mbox{for } t>0, \\
\label{1dfinal-h0h0}
& \qquad h(0) = h_0. &  & 
\end{empheq}
\end{subequations}

\subsection{Numerical simulations and perturbation analysis}
\label{ss:1Dsimulation}

We have numerically solved the problem \eqref{vfbp-nondim-expanded}, with 
$a = 1$, $c_0 = 1.5$ and $d_0 = 5$,
and our results are depicted in Figure~\ref{fig:vertical}. We observe
that the concentration profile $c(z, t)$ for a fixed $t$ 
is monotonically decreasing in $z$ and strictly positive for all times and the
bacterial colony height $h(t)$ eventually grows linearly in time $t$. Moreover,  
the value of the concentration profile at $z=0$ converges to a constant $U_0$ as $t\to \infty;$  see 
the insets in Figure~\ref{fig:vertical}. The meaning of this constant $U_0$
is shown in our
asymptotic analysis below; cf.\ \eqref{scalar_eqn_U_0}. 

\begin{figure}[h]
\begin{center}
\includegraphics[width=0.90\textwidth]{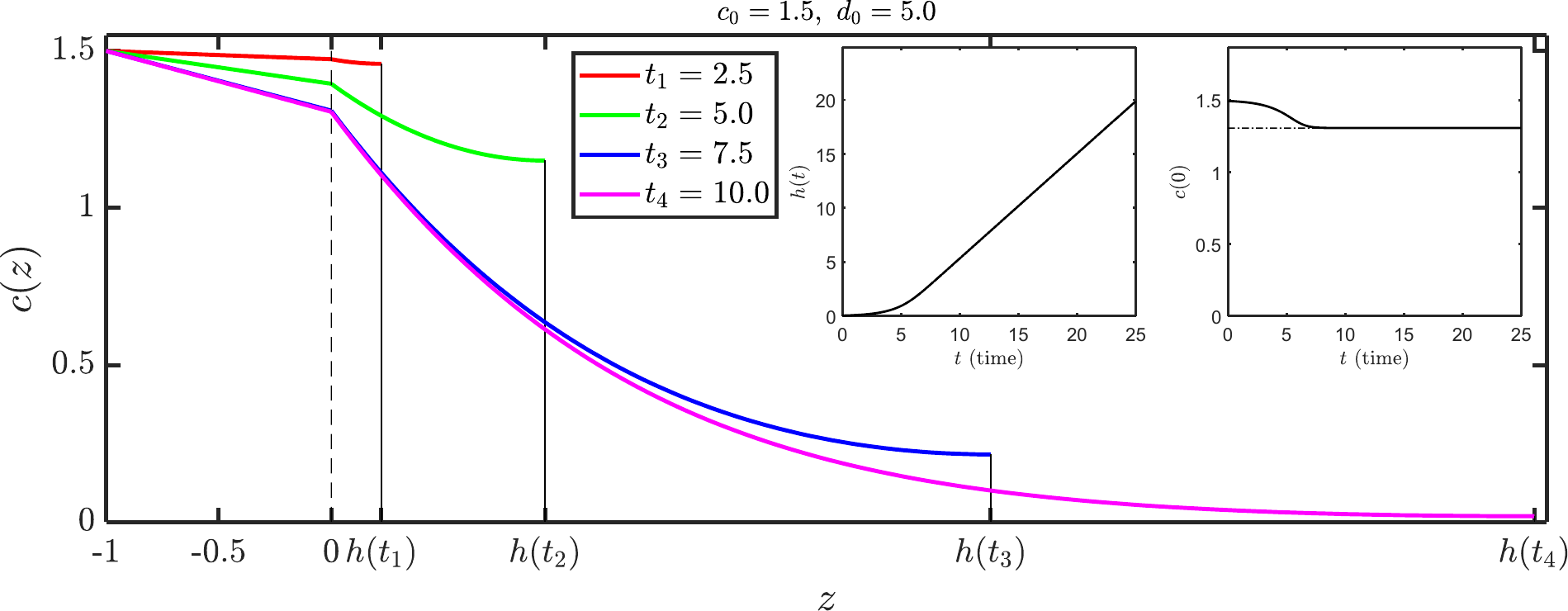}
\end{center}
\vspace{-4 mm}

\caption{Numerical simulation of the system \eqref{vfbp-nondim-expanded}.
The colony height $h  = h(t)$ and the profiles of the nutrient concentration $c = c(z, t)$ $(-a \le z \le h(t))$ are presented for several values of $t = t_i $ $(i=1, 2,3, 4)$. 
Insets: (Left) The bacterial colony height $h = h(t)$ 
as a function of time; (Right) The value of the nutrient concentration at $z=0$ vs.\ time.
The horizontal dashed line is given by $c=U_0$ where $U_0$ solves~\eqref{scalar_eqn_U_0}. }
\label{fig:vertical}
\end{figure}

As we will show below in Proposition~\ref{p:depletion}, the concentration 
$c = c(z, t)$ for fixed $t > 0$ decays quadratically for $0 \le z \le z_\ast(t)$ and 
decays exponentially for $z > z_\ast(t),$ where $z_\ast(t) $ is defined by $c(z_\ast(t), t) = 1$
which indicates the concentration reaches the Monod constant $K$ in the dimensional form. 
In Figure~\ref{fig:2vertical}, results of numerical simulations for $z_*(t)$ are presented. The definition of $z_*=z_\ast(t)$ is illustrated in Figure~\ref{fig:2vertical} (a). 
From Figure~\ref{fig:2vertical} (b), we observe that $z_\ast = z_*(t)$ converges exponentially 
to a limiting value $z_*^\infty$ as time evolves. This is proved below in Proposition~\ref{p:exponential} and 
Theorem~\ref{t:1Ddynamics}.

\begin{figure}[h]
\begin{center}
\includegraphics[width=0.85\textwidth]{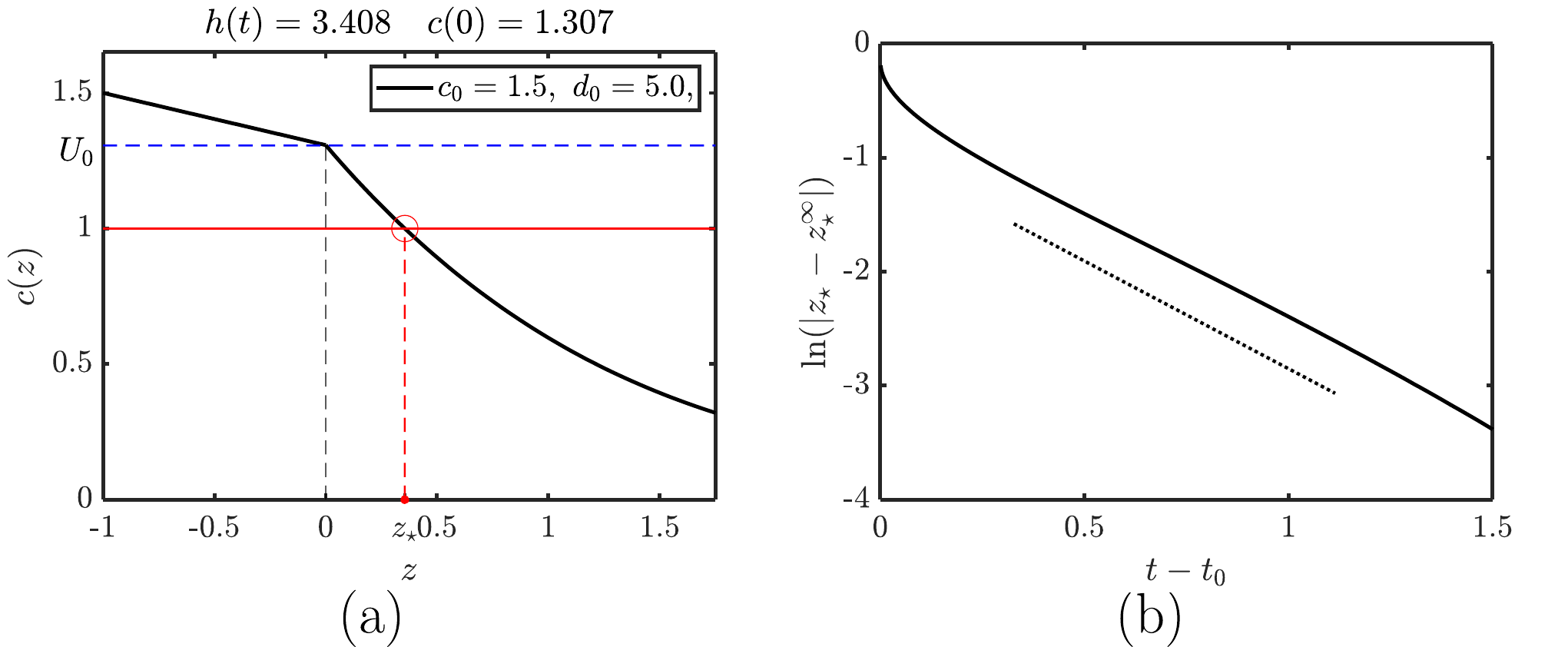}
\end{center}
\vspace{-4 mm}
\caption{
Numerical simulation of the system \eqref{vfbp-nondim-expanded}: the asymptotic dynamics 
of $z_{*}=z_{*}(t)$ defined by  $c(z_*, t)=1$ for the case  
$U_0>1$. Here, $U_0$ is the limiting value of $c(0,t)$ as $ t \to \infty$; cf.
\eqref{scalar_eqn_U_0}.  
 (a) The profile $c(z,t)$ at $t=10$. Note that $c(0,t)\approx U_0$ for $t=10$. (b) The slope of $\ln|z_{*}-z_{*}^{\infty}|$ is linear, where $z_\ast = z_\ast(t)$
 and $z_{*}^{\infty}={\lim}_{t\to\infty}z_{*}(t)$. 
 Note that $t_0$ denotes the first time when $c(z_*,t)=1$ has a solution, so $z_*(t)$ is defined uniquely  for $t\geq t_0$
 due to the monotonicity of $c$ with respect to $t.$
}
\label{fig:2vertical}
\end{figure}

We now consider the solution $c_h = c_h(z)$ to the time-independent version
of the system 
\eqref{1dfinal-c}--\eqref{1dfinal-BC-at-h}, with $h $ and $c_h(z)$ replacing
$h(t)$ and $c(z,t)$, respectively, and perform an asymptotic analysis \cite{holmes2012introduction}
of $c_h$ in the limit \(h\to\infty\). 
We rescale the spatial variable 
\(
\hat z={z}/{h} \in [0, 1]
\)
and define $\hat{c}_h (\hat{z}) = c_h(z).$ 
Introducing the small parameter
\(
\varepsilon:=h^{-1}\ll 1,
\)
we rewrite the function $\hat{c} = \hat{c}_h$ which now depends implicitly on 
$\varepsilon$.  It is the unique solution of
the following boundary-value problem on the fixed interval $(0, 1)$: 
\begin{subequations}\label{vfbp-nondim-rescaled}
\begin{empheq}[left=\empheqlbrace]{align}
    &  \varepsilon^2\hat{c}'' =  \dfrac{\hat{c}}{\hat{c}+1} \qquad  \forall \hat{z} \in (0, 1), & & \label{vfbp-nondim-rescaled-a} \\
  & \quad \varepsilon \hat{c}'(0) =\kappa(\hat{c}'(0) - c_0),  &&  \\
 & \quad \hat{c}'(1)  = 0.&&
  \end{empheq}
  \end{subequations}

We seek an asymptotic
representation of \(\hat{c}\) as the sum of an outer solution $\hat{c}_{\text{out}}$, which is regular in both \(\hat z\) and \(\varepsilon\), and an inner
(boundary-layer) solution \(\hat{c}_{\text{in}}\), which  depends on the stretched variable $\hat{z}/\varepsilon$: 
\begin{equation}
    \label{hatcexp}
 \hat{c}(\hat{z}) \approx \hat{c}_{\text{out}}(\hat{z},\varepsilon) \quad 
 \text{ for }\hat{z}=O(1) \qquad \text{ and } \qquad \hat{c}(\hat{z}) \approx \hat{c}_{\text{in}}\left(\dfrac{\hat{z}}{\varepsilon},\varepsilon\right)\quad \text{ for } \hat{z}=O(\varepsilon),
\end{equation}
where 
\begin{equation}
\hat{c}_{\text{out}}(\hat{z}, \varepsilon) =\sum\limits_{n=0}^{\infty}\varepsilon^n \hat{c}_{\text{out},n}(\hat{z})
\qquad \text{ and }\qquad \hat{c}_{\text{in}}\left(\dfrac{\hat{z}}{\varepsilon},\varepsilon\right)=\sum\limits_{n=0}^{\infty}\varepsilon^n\hat{c}_{\text{in},n}\left(\dfrac{\hat{z}}{\varepsilon}\right), \label{s3-series}
\end{equation}
Note that the stretched variable $\hat{z}/\varepsilon$ coincides with the original variable $z$.
Here, to understand the limiting behavior of the profile of $\hat{c}$ as $\varepsilon\to 0$, we are only interested in the first terms of these two series above, with the matching condition 
\begin{equation}
\label{matchcondition}
\lim\limits_{z\to\infty}\hat{c}_{\text{in},0}(z)=\lim\limits_{\hat{z}\to 0}\hat{c}_{\text{out},0}(\hat{z}).
\end{equation}

We substitute  expressions \eqref{s3-series} into \eqref{vfbp-nondim-rescaled} and the matching condition \eqref{matchcondition} to obtain that  $\hat{c}_{{\rm in}, n}\equiv 0$ for all $n\geq 1$ and $\hat{c}_{{\rm out}, n}\equiv 0$ for all $n \ge 0$.   
The only non-zero term in \eqref{s3-series}, which is $U(z):=\hat{c}_{\text{in},0}\left({\hat{z}}/{\varepsilon}\right)$, satisfies the following boundary-value problem
\begin{subequations}\label{vfbp-boundary-layer}
\begin{empheq}[left=\empheqlbrace]{align}
    & U'' = \dfrac{U}{U+1} \qquad \forall z>0, \label{eqnU}\\
  &  \qquad U'(0)=\kappa(U(0)-c_0), \label{bcU0}\\
 &\qquad U \to 0\text{ as }z\to\infty.\label{bcInf}
  \end{empheq}
\end{subequations}
Here, the condition at $z\to \infty$ is due to the matching condition \eqref{matchcondition}. Note that the fact that there is only one term, $U(\hat{z}/\varepsilon)$, in the multiscale asymptotic expansion \eqref{hatcexp}--\eqref{s3-series}, does not imply that $\hat{c}(\hat{z})=U(\hat{z}/\varepsilon)$, but rather $\hat{c}(\hat{z})=U(\hat{z}/\varepsilon)+O(\varepsilon^n)$ for any $n>0$ (for example, we can have $\hat{c}(\hat{z})=U(\hat{z}/\varepsilon)+O(e^{-1/\varepsilon})$).

Eq.~\eqref{eqnU} has the first integral
$ (U')^2 / 2- \Lambda(U)= \text{const}.$ 
This and \eqref{bcInf} imply 
  \begin{equation}
  \label{UpU0}
   U' = -\sqrt{2\Lambda(U)} \quad \forall z > 0 \qquad \mbox{and} \qquad 
  U(0)=U_0 > 0,
  \end{equation}
 where $U_0$ is to be determined. 
 Since $\Lambda (U) \sim U$ for $U \sim 0$, the zero function is a stable state for the 
 equation in \eqref{UpU0} and the
 initial-value problem \eqref{UpU0} has a unique 
 solution. 
Using $\partial_z U = -\sqrt{2\Lambda(U)}$, 
the boundary condition \eqref{bcU0} and the equation in \eqref{UpU0}, we obtain
  \begin{equation}
 \sqrt{2\Lambda(U_0)}=\kappa (c_0 - U_0).\label{scalar_eqn_U_0}
  \end{equation}
  Here, the left-hand side is strictly monotonically increasing in $U_0$, vanishes at $U_0=0$, and grows to infinity as $U_0\to \infty$. Therefore, there exists a unique $U_0$ solving \eqref{scalar_eqn_U_0}, which in turn implies 
  the unique solvability of the system \eqref{vfbp-boundary-layer}.

In conclusion of the perturbation analysis, we have $c_h(z) \sim U(z)$ for large $h$ and $c_h(0) \to U_0$ for $h\to\infty$, where $U_0>0$ solves \eqref{scalar_eqn_U_0}. Due to \eqref{1dfinal-BC-at-h}, it implies that $h(t)$ grows linearly as $t\to \infty$. Convergence of the concentration profile $c_h(z)$ to a constant-in-time profile $U(z)$ and the linear growth of $h(t)$ for $t\to\infty$ is in the agreement with the numerical observations illustrated by Figure~\ref{fig:vertical}. These conclusions, along with other statements, will be made rigorous in the subsections below.

\subsection{Profile of nutrient concentration}
\label{ss:1Dsteadystate}

In this subsection, we study the following boundary-value problem for the equilibrium concentration $c = c(z)$, 
given $a>0$, $h>0$, and $c_0 > 0$: 
\begin{subequations}\label{ss-vfbp-nondim-expanded}
\begin{empheq}[left=\empheqlbrace]{align}
\label{ss-1dfinal-c---expanded}& 
c'' = 0 & & \ \mbox{for } z\in(-a,0), & \\
\label{ss-1dfinal-c-expanded}&c''=\lambda(c)& & \ \mbox{for } z\in (0,h), & \\
\label{ss-1dfinal-BC_at_0-expanded}& \qquad c'(0^+) = d_0c'(0^-), & &  & \\
\label{ss-1dfinal-BC_at_ah-expanded} &\qquad  c(-a)=c_0 \quad \mbox{and} \quad c'(h) = 0. & & & 
\end{empheq}
\end{subequations}
Here, we show that this problem has a unique solution $c \in H^1(-a, h)$, 
 characterize the solution profile $c = c(z)$, compare such profiles for different values of $a$ and $h$, 
show that the concentration depletes into the colony region, and obtain the large $h$ asymptotic behaviors of 
the concentration.


\begin{proposition}
    \label{p:1Dc}
    Given $a > 0$, $h > 0$, and $c_0 > 0.$  The boundary-value problem 
   \eqref{ss-vfbp-nondim-expanded} admits a unique, nonnegative, and continuous solution $c = c (z)$ $(-a \le z \le h)$. 
   Moreover, $c(z)$ satisfies the following:
 \begin{compactenum}
    \item[{\rm (i)}]
        $c$ is linear in $[-a, 0]$, given by 
        \[
        c(z,t) = \frac{(c(0)- c_0)}{a}z + c(0) \qquad  (-a\le z \le 0);
        \]
       \item[{\rm (ii)}]
       $c \in C^\infty ([0, h])$, and 
       $c$ is strictly decreasing and $c'$ is strictly increasing in $[0, h]${\rm ;}
        \item[{\rm (iii)}]
        $0 < c(h) < c(0) < c_0$. 
\end{compactenum}
\end{proposition}

Figure~\ref{f:1Dsolution} shows a 
typical profile of the solution to \eqref{ss-vfbp-nondim-expanded} 
(see also numerical solutions in Figure~\ref{fig:vertical} (a)).

\begin{figure}[hptb]
\begin{center}
\includegraphics[width=0.6\textwidth]{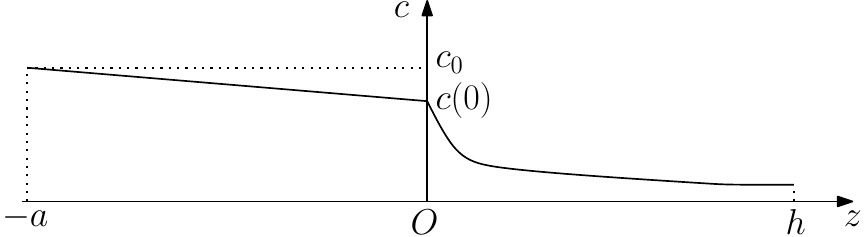}
\end{center}

\vspace{-4 mm}

\caption{Sketch of the solution profile $c = c(z)$ to the boundary-value problem \eqref{ss-vfbp-nondim-expanded}. 
}
\label{f:1Dsolution}
\end{figure}

\begin{proof}[Proof of Proposition~\ref{p:1Dc}] The existence and uniqueness of 
a nonnegative and continuous solution can be obtained by the direct method; cf.\ Theorem~\ref{t:3Dminimizer}.

Part (i) follows  from \eqref{ss-1dfinal-c---expanded}
and \eqref{ss-1dfinal-BC_at_ah-expanded}.

{Part (ii).} The regularity $c|_{[0, h]} \in C^\infty([0, h])$ follows from a usual bootstraping argument. 
Since $c(z) > 0$ for $z \in (0, h)$, we have by \eqref{ss-1dfinal-c-expanded} that 
\[
c'(z) = - \int_z^h \Lambda (c(\zeta))\,\text{d}\zeta < 0 \qquad \mbox{if }0<z<h. 
\]
(The function $\lambda$ and $\Lambda$ are defined in
\eqref{nondim-def-lambda} and \eqref{Fc}, respectively.)
Hence, $c$ is strictly decreasing. This leads to  $c(z)> c(h) \ge  0$ for all $z \in [0,h]$, 
and further to $c''(z) = \lambda(c(z)) > 0$ for all $z \in [0, h]. $ Hence, $c'$ is strictly increasing on $[0, h].$

 Part (iii). 
Suppose $c(h)=0$. Then, $c = c(z)$ is the unique nonnegative solution to 
the initial-value problem of a second-order ordinary differential 
equation (ODE): $c''= \lambda(c)$ in $(0, h)$ and $c(h) = c'(h) = 0$. 
Since the initial-value has a unique solution and $\lambda(0)=0$, 
we have that $c(z)=0$ for all $z\in(0,h)$, contradicting the fact that $c > 0$ in $(0, h).$
The other inequalities follow since $c$ decreases strictly.
\end{proof}

We now study  the steady-state concentration $c = c(z)$ $(-a < z < h)$ as we vary
the agar thickness $a$ and the colony height $h.$ We denote 
the solution to 
the system of equations \eqref{ss-vfbp-nondim-expanded} by $c = c_{a, h}(z)$ $(-a < z < h)$  to emphasize the dependence on the agar 
thickness $a$ and colony height $h$. The statement of the next 
proposition is illustrated by Figure~\ref{f:ah}.
See also Figure~\ref{fig:vertical} for the related numerical simulations for comparison.

\begin{proposition}
    \label{p:ah} 
    \begin{compactenum} 
    \item[{\rm (i)}]
    If $0 < a_1 < a_2$ and $h > 0$ then 
    $c_{a_2, h}(z) < c_{a_1, h}(z)$ for all $z \in [-a_1, h]$
    and $c_{a_1, h}'(z) < c_{a_2, h}'(z)$ for all $z \in [-a_1, 0]$ and  all $z \in [0, h).$ 
    \item[{\rm (ii)}]
    If $ a > 0$ and $0 < h_1 < h_2$ then $c_{a, h_2}(z) < c_{a, h_1}(z)$ 
    for all $z \in (-a, h_1]$ and 
    $c_{a, h_2}'(z) < c_{a, h_1}'(z)$ for all $z \in [-a,0] $ and all $z \in [0,  h_1]$. 
    \end{compactenum}
\end{proposition}


\begin{figure}[hpt]
\begin{center}
\includegraphics[width=0.9\textwidth]{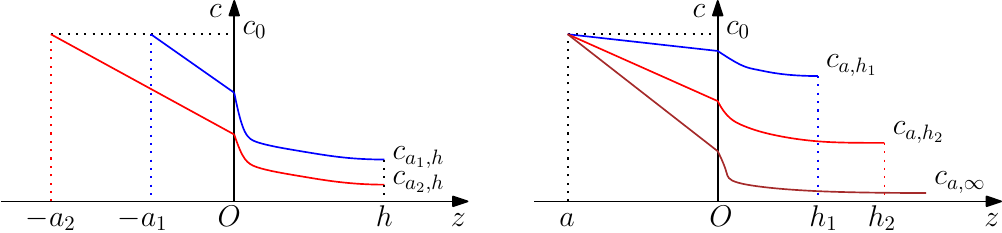}
\caption{A schematic illustration of Proposition~\ref{p:ah}.
Left: The concentration profiles $c_{a_1,h}$ and $c_{a_2,h}$ with $a_1<a_2$. 
Right: The concentration profiles $c_{a,h_1}$ and $c_{a,h_2}$ with $h_1<h_2$. 
The limiting profile $c_{a,\infty}$ is also shown in the right panel. 
}
\label{f:ah}
\end{center}
\end{figure}

\begin{proof}
{Part (i).} We first show that $c_{a_2, h}(z) < c_{a_1, h}(z)$ if $ -a_1 \le z < 0.$ Suppose that this is not true. 
Since by Proposition~\ref{p:1Dc} both $c_{a_1, h}$ and $c_{a_2,h}$ are linear in $ [-a_1, 0]$ and 
$
c_0 = c_{a_1, h}(-a_1) > c_{a_2, h}(-a_1) > 0,
$
there exists $a \in (0, a_1)$ such that 
\[
\hat{c}_0 := c_{a_1, h}(-a) = c_{a_2, h}(-a) > 0 \qquad \mbox{and} \qquad 
c_{a_1, h}(z) < c_{a_2,h}(z) \qquad \mbox{if } -a < z < 0.
\]
This means that both $c_{a_1, h}$ and $c_{a_2, h}$ are solutions to the system
of equations \eqref{ss-vfbp-nondim-expanded} with $c_0$ 
replaced by $\hat{c}_0 > 0$. Thus, the uniqueness of solution implies that 
$c_{a_1, h} = c_{a_2, h}$ on $[-a, 0]$. This is a contradiction. 
Similarly, if $c_{a_1, h}(0) = c_{a_2, h}(0) \in (0, c_0)$, then, 
since both $c_{a_1, h}$ and $c_{a_2, h}$ satisfy
the same equation on $(0, h)$ and the same boundary conditions at $z = 0$ and $z = h$, they are the same unique
solution on $(0, h)$. Hence, $c_{a_1, h}'(0^+) = c_{a_2, h}'(0^+),$ 
and thus by the interface condition
$c_{a_1, h}'(0^-) = c_{a_2, h}'(0^-).$ This, the assumption that $c_{a_1, h}(0) = c_{a_2, h}(0),$ 
and the fact that 
$c_{a_1, h}(-a_1) = c_{a_2, h}(-a_2) = c_0$ imply that $a_1 = a_2$, 
which is impossible. Hence, $c_{a_2, h}(z) < 
c_{a_1, h}(z)$ for all $z \in [-a_1, 0].$

Next, we show that $c_{a_2, h}(z) < c_{a_1, h}(z)$ for any $z \in (0, h].$ Assume this is not true. 
Then, there exists $z_0 \in (0, h]$ such that $c_{a_2, h}(z) < c_{a_1, h}(z)$ for all $z \in [0, z_0)$ 
but $c_{a_2, h} (z_0) = c_{a_1, h}(z_0). $ Let $u = c_{a_1, h} - c_{a_2, h}$ on $[-a, h].$ 
Thus, $u(z) > 0$ in $[0, z_0)$ and $u(z_0) = 0.$ Hence $u'(z_0) \le 0.$ 
There are two different cases: either $u'(z_0) = 0$ or $u'(z_0) < 0.$ 
We shall get a contradiction for each of these two cases. 

First, we observe that 
 $u \in C^2([0, h])$ is the unique solution to the following boundary-value problem
\begin{align}
\label{probu}
u'' = A(z) u \quad \mbox{in } (0, h), \qquad 
 u(0) > 0, 
 \quad \mbox{and} \quad u'(h) = 0, 
\end{align}
where 
\begin{equation}
\label{Az}
A(z) = \frac{1}{ (c_{a_1, h}(z) + 1) ( c_{a_2, h}(z) + 1)}. 
\end{equation}
By Proposition~\ref{p:1Dc}, $c_{a_1, h}$ and $c_{a_2, h}$ are nonnegative and decreasing in $[0, h].$ Hence, 
\begin{equation}
    \label{A}
    0 < A(0) \le A(z) \le A(h) \qquad \forall z \in [0, h].
\end{equation} 

Now consider the case $u'(z_0) =0.$
By \eqref{A}, the fact that $u > 0$ on $[0, z_0)$, and \eqref{probu}, 
we have 
\begin{equation}
\label{AAA}
0< A(0) u (z) \le A(z) u (z) = 
u'' (z) \le A(h) u (z)  \qquad \forall z \in (0, z_0).
\end{equation}
Hence, $u''> 0 $ in $(0, z_0) $ and thus
$u'(z) \le u'(z_0) = 0 $ if $ z \in (0, z_0).$
Multiplying all parts in \eqref{AAA} by $2 u'(z) \le 0$ and then integrating 
from $z \in (0, z_0)$ to $z_0$ and using the fact that 
$ u(z_0) = u'(z_0) = 0 $ and $ u'\le 0 $ and $ u > 0 $ in $ (0, z_0),$ we get 
\[
- \sqrt{A(0)} \, u (z) \ge 
u' (z) \ge - \sqrt{A(h)}\, u (z) \qquad \forall z\in (0, z_0).
\]
By dividing all sides by $u (z) > 0$ for $z \in (0, z_0)$ and integrating the resulting sides
from $\varepsilon$ with $0< \varepsilon < z_0$ to  $z'\in (\varepsilon, z_0)$, we obtain by sending $\varepsilon\to 0^+$ that 
\[
-\sqrt{A(0)} \, z' \ge 
\ln u(z') - \ln u(0)  \ge - \sqrt{A(h)}\, z' \qquad \forall z' \in (0, z_0).
\]
As $z'\to z_0^-$, $u(z') \to 0^+$. We thus get a contradiction as the middle term above goes to $-\infty$ while the first
and last term tend to constants as $z'\to z_0^-.$  
Therefore, we have ruled out the case that $u'(z_0) = 0.$ 

Now consider the case that $u'(z_0) < 0.$ This implies that  $z_0 < h$, since $u'(h) = 0,$
and that  $u(z) < 0$ if $z > z_0$ and $z$ is close to 
$z_0$.  Suppose $u(z) < 0 $ for all $z \in (z_0, h)$, then $u'' < 0$ in $(z_0, h)$ by 
\eqref{probu}.
 Hence $u'$ decreases in $(z_0, h)$ and thus $u'(h) \le u'(z_0) < 0$, contradicting the fact that 
$u'(h) = 0.$ Thus, there exists $z_1 \in (z_0, h)$ such that $u(z_1) = 0$ and $u(z) < 0$ if $z \in (z_0, z_1).$
By \eqref{probu} again, $u''< 0$ in $ (z_0, z_1)$ and hence 
$u'(z) \le u'(z_0) < 0 $ if $  z \in (z_0, z_1).$
Thus $u$ decreases in $(z_0, z_1)$. But $u(z) < u(z_0) = 0$ if $z_0 < z < z_1$ and $z $ is close to $z_0$. For such 
$z$, we thus have 
$
0 = u(z_1) \le  u(z) < u(z_0) = 0,
$
a contradiction. We have thus ruled out the case $u'(z_0) < 0.$
Therefore, $c_{a_2, h}(z) < c_{a_1, h}(z)$ for any $z \in (0, h]$ and hence for any $z \in [-a_1, h].$

Finally, we show that $c_{a_1,h}'(z) < c_{a_2, h}'(z)$ for any $z \in [-a_1, 0]$ and for 
any $z \in [0, h).$ 
 By what has been proved, $c_{a_2, h}(\hat{z}) < c_{a_1, h}(\hat{z})$ for any $\hat{z} 
\in [0, h]. $ Hence, $\lambda(c_{a_2, h}(\hat{z})) < \lambda (c_{a_1, h}(\hat{z}))$ for any $\hat{z}
\in [0, h].$ It then follows from the equation for $c_{a_i, h}$ on $(0, h)$ and the boundary condition 
$c_{a_i, h}'(h) = 0$ $(i=1,2)$ that for any $z \in (0, h)$
\begin{align}
\label{cpa2a1}
    -c_{a_2,h}'(z) &= \int_z^h  c_{a_2, h}''(\hat{z})\,\text{d}\hat{z}
    =  \int_z^h  \lambda (c_{a_2, h}(\hat{z}))\, \text{d}\hat{z} \nonumber \\
    & <  \int_z^h  \lambda (c_{a_1, h}(\hat{z}))\, \text{d}\hat{z} = 
    \int_z^h  c_{a_1, h}''(\hat{z})\, \text{d}\hat{z} = -  c_{a_1, h}'(z).
\end{align}
Hence $c_{a_1, h}'(z) < c_{a_2, h}'(z)$ for any $z \in (0, h).$ The above inequality \eqref{cpa2a1}
also holds true if we replace $z$ by $0$. Thus, 
$c_{a_1, h}'(0^+) < c_{a_2, h}'(0^+)$ and hence
$c_{a_1, h}'(0^-) < c_{a_2,h}'(0^-).$ Now for any $z \in [-a_1, 0)$, we have 
\[
c_{a_1, h}'(z) = c_{a_1, h}'(0^-) < c_{a_2,h}'(0^-) = c_{a_2, h}'(z).
\]

{Part (ii).} We first show that $c_{a, h_1}(0) \ne c_{a,h_2}(0).$ Suppose $c_{a,h_1}(0) = c_{a, h_2}(0).$
Then, since both $c_{a, h_1} $ and $c_{a, h_2}$ are linear in $[-a, 0]$ and $c_{a, h_1}(-a) = c_{a, h_2}(-a) = c_0$, 
we have $c_{a,h_1}'(0^-) = c_{a, h_2}'(0^-)$ and hence $c_{a, h_1}'(0^+) = c_{a, h_2}'(0^+).$ Since 
$ c_{a, h_i}'' = \lambda (c_{a, h_i})$ on $ (0, h_i)$ for $i = 1, 2,$ we can multiply both sides of
the equation by $c_{a,h_i}'$ and then integrate from $0$ to any $z \in (0, h_1) \subset (0, h_2)$ to get 
\[
  [c_{a, h_i}'(z)]^2 - [c_{a,h_i}'(0^+)]^2 
= 2 \Lambda (c_{a, h_i}(z)) - 2  \Lambda (c_{a, h_i}(0)) \qquad \forall z \in (0, h_1). 
\]
Let 
\[
\beta = [c_{a,h_1}'(0^+)]^2- 2 \Lambda (c_{a, h_1}(0) )
= [c_{a,h_2}'(0^+)]^2- 2 \Lambda (c_{a, h_2}(0) ).
\]
Then, since $c_{a, h_i}' < 0$ in $(0, h_1)$ by Proposition~\ref{p:1Dc}, we have for $i  = 1$ or $2$ that
\[
c_{a, h_i}'(z) = -\sqrt{ 2 
\Lambda (c_{a, h_i}(z)) + \beta } \qquad \forall z \in (0, h_1).
\]
Note that the right-hand side is a $C^1$-function of $z \in (0, h_1).$ Moreover, its 
derivative is bounded in absolute value in any compact sub-interval of $(0, h_1).$ Thus, since
$c_{a, h_1}(0) = c_{a, h_2}(0),$ the uniqueness of an initial-value problem of a first-order ODE
implies that $c_{a, h_1}(z) = c_{a, h_2}(z)$ for any $z \in (0, h_1).$ 
Thus $c_{a,h_2}'(h_1) = c_{a, h_1}'(h_1) = 0$. This is in contradiction with the fact that 
$c_{a, h_2}'(z) < 0$ for $0 < z < h_2$ by Proposition~\ref{p:1Dc} and the assumption 
that $ 0 < h_1 < h_2.$ Thus, 
$c_{a, h_1}(0) \ne c_{a, h_2}(0). $

We now have two possibilities. One is $ c_{a,h_1}(0) > c_{a, h_2}(0)$ and the other is 
$c_{a, h_2}(0) > c_{a, h_1}(0).$ We claim the following:  $c_{a,h_1} (z) > c_{a, h_2}(z)$ for any $z \in (-a, h_1]$
if  $c_{a, h_1}(0) > c_{a, h_2}(0)$ and 
$c_{a,h_2} (z) > c_{a, h_1}(z)$ for any $z \in (-a, h_1]$
if  $c_{a, h_2}(0) > c_{a, h_1}(0).$
Assume  $c_{a, h_1}(0) > c_{a, h_2}(0).$
Clearly, $c_{a, h_1}(z) > c_{a, h_2}(z)$ if $z \in (-a, 0]$, since both $c_{a, h_1}$ and $c_{a, h_2}$ are linear 
on $(-a, 0]$ and they both equal $c_0$ at $z = -a.$ We thus need to show only that $c_{a,h_1}(z) > c_{a, h_2}(z)$
for any $z \in (0, h_1].$
If this is not true, then there exists 
$z_0 \in (0, h_1]$ such that $c_{a, h_1}(z) > c_{a, h_2}(z)$ if $0 \le z < z_0$ and $c_{a, h_1}(z_0)
= c_{a, h_2}(z_0).$ Then, by the weak formulation for $c_{a,h_i}$, we have for $i = 1 $ or $2$ that
\begin{equation}
\label{weakhi}
\int_{-a}^0 d_0 c_{a, h_i}' v' \, \text{d}z + \int_{0}^{z_0} 
\left[ c_{a, h_i}' v'  + \lambda(c_{a, h_i}) v \right] \text{d}z = 0
\end{equation}
for any $v \in H^1((-a, h_i))$ such that $v(-a) = 0$ and $v = 0 $ on $[z_0, h_i].$
Let $u = c_{a, h_1} - c_{a, h_2} \in H^1((-a, h_1)).$ Clearly $u(-a) = u(z_0) = 0. $
Let $v = u$ on $(-a, z_0)$ and $v = 0$ on $[z_0, h]$. Then $v \in H^1((-a, h_1))$
and $v(-a) = v(z_0) = 0$. By \eqref{weakhi}, we thus have 
\begin{equation}
\label{weakweak}
\int_{-a}^0 d_0 u'^2 \text{d}z + \int_0^{z_0} \left[  u'^2 +  \left( \lambda(c_{a, h_1})
- \lambda(c_{a, h_2}) \right) u \right] \text{d}z = 0.
\end{equation}
By \eqref{lambda} and Proposition~\ref{p:1Dc},
\[
\left( \lambda(c_{a, h_1}) - \lambda(c_{a, h_2}) \right) u 
= \frac{ u^2 }{(c_{a,h_1} + 1) 
(c_{a, h_2}+1)}\ge \frac{ u^2}{(c_0+1)^2} 
\qquad \mbox{in } (-a, z_0). 
\]
Therefore, it follows from \eqref{weakweak} that $u = 0$ and hence $c_{a,h_1} = c_{a, h_2}$ 
in $(-a, z_0)$. In particular, $c_{a, h_1}(0) = c_{a,h_2}(0)$. This contradicts the assumption
that $c_{a, h_1}(0) > c_{a, h_2}(0)$. Hence $c_{a, h_1}(z) > c_{a, h_2}(z)$ for any $z \in 
(-a, h_1]$ provided that $c_{a, h_1}(0) > c_{a, h_2}(0).$ Similarly, 
$c_{a, h_1}(z) < c_{a, h_2}(z)$ for any $z \in (-a, h_1]$ provided that $c_{a, h_1}(0) < c_{a, h_2}(0).$

We now prove that $c_{a,h_2}(z) < c_{a, h_1}(z)$ for any $z \in (-a, h_1].$ If this is not true, then 
$c_{a, h_2}(0) > c_{a, h_1}(0)$, since $c_{a, h_2}(0) \ne c_{a, h_1}(0)$. Hence,  following
from what we have proved above, $c_{a, h_2}(z) > c_{a, h_1}(z)$ for any $z \in (-a, h_1].$
Let again $u = c_{a, h_1} - c_{a, h_2}$ on $[-a, h_1].$ Then $u < 0$ on $(-a, h_1].$ Since $c_{a, h_1}(0) < c_{a, h_2}(0)$
and $c_{a, h_1}(-a) = c_{a, h_2}(-a) = c_0$, we have $0 > c_{a,h_2}'(0^-) > c_{a, h_1}'(0^-).$
Hence $0 > c_{a, h_2}'(0^+) > c_{a, h_1}'(0^+),$ i.e., $u'(0^+) < 0.$ On the other hand, 
by Proposition~\ref{p:1Dc}, $c_{a,h_2}'(h_1) < 0$. This and the fact that $c_{a, h_1}'(h_1) = 0$
imply that  $u'(h_1) = c_{a, h_1}'(h_1) - c_{a, h_2}'(h_1) > 0$. 
We can verify that 
$ u'' = A(z) u$ in $(0, h_1)$, where $A(z) $ is given in \eqref{Az} and 
it satisfies  $0 < A(0) \le A(z) \le A(z_1) $ if $z \in (0, h_1).$
Since $u < 0$ in $(0, h_1)$, we have  $u''<0$ in $(0, h_1).$ Hence 
$u'(z) \le u'(0^+) < 0 $ if $0 < z < h_1.$ Sending $z \to h_1^-$, we get
$u'(h_1) < 0$, in contradiction with $u'(h_1) > 0.$ 

Finally, we prove that 
 $c_{a, h_2}'(z) < c_{a, h_1}'(z)$ for all $z \in [-a,0] $ and  all $z \in [0,  h_1]$. 
 Since $c_{a, h_2}(0)
< c_{a, h_1}(0)$ and $c_{a, h_1}(-a) = c_{a, h_2}(-a) = c_0$, we have  
\[
c_{a, h_2}'(z) = c_{a, h_2}'(0^-)
< c_{a, h_1}'(0^-) = c_{a, h_1}'(z) \qquad \forall z \in [-a, 0].
\]
This also implies 
that  $ c_{a, h_2}'(0^+) < c_{a, h_1}'(0^+).$ Moreover, 
as we have shown, $0 < c_{a, h_2}(\hat{z}) < c_{a, h_1}(\hat{z})$ and hence 
$\lambda (c_{a, h_2}(\hat{z})) < \lambda (c_{a, h_1}(\hat{z}))$ for any $\hat{z}\in [0, h_1].$
Consequently, by the equation
for $c_{a,h_i}$ on $(0, h_i)$ $(i = 1, 2)$, we have for any $z \in (0, h_1]$ that 
\begin{align*}
c_{a, h_2}'(z) &=  c_{a, h_2}'(0^+) + \int_0^z   \lambda (c_{a, h_2}(\hat{z}))\, \text{d}\hat{z} 
<  c_{a, h_1}'(0^+) +  \int_0^z   \lambda (c_{a, h_1}(\hat{z}))\, \text{d}\hat{z} =  c_{a, h_1}'(z),
\end{align*}
as desired. 
\end{proof}


We now consider the asymptotic behaviors of the steady-state concentration $c_{a, h} = c_{a, h}(z)$ $(-a < z < h)$ 
in the limit of thick agar plate or large colony height. Recall that $\kappa = d_0/a.$

\begin{proposition}
    \label{p:1Dasymp} 
    \begin{compactenum}
        \item[{\rm (i)}]
    For any $h > 0$, $ c_{a,h}(0) \searrow 0$ as $a \to \infty$ and $c_{a, h}(0) \nearrow c_0$ as 
    $a \to 0$. 
    \item[{\rm (ii)}]
    For any $a > 0$, both  $c_{a, h}(0)$ and $c_{a, h}(h)$ decrease as $h$ increases, 
    \[
       \lim_{h\to \infty} c_{a,h}(h) = 0, 
      \qquad \mbox{and} \qquad 
    \lim_{h\to \infty} c_{a, h}(0) = \mu, 
    \]
    where $\mu = \mu (a, c_0) \in (0, c_0)$ is the unique solution to 
      \begin{equation}
   \label{definemu}
  \kappa^2 \left( c_0 - \mu \right)^2 = 2 \Lambda(\mu). 
   \end{equation}
   For a fixed $c_0 > 0$, $ \mu = \mu(a, c_0)$ is a strictly decreasing function
   of $a > 0$ with 
   \[
   \lim_{a\to 0} \mu(a, c_0) = c_0 
\qquad \mbox{and} \qquad    
   \lim_{a \to \infty} \mu(a, c_0) = 0. 
   \]
   For a fixed $a > 0$, $\mu = \mu(a, c_0)$ is a strictly increasing function 
   of $c_0 > 0$ with 
   \[
   \lim_{c_0 \to 0} \mu(a, c_0) = 0 \qquad \mbox{and} \qquad  
   \lim_{c_0\to \infty}
   \mu(a, c_0) = \infty. 
   \]
   Moreover, $\mu = 1$ when $\kappa^2 (c_0-1)^2 =2 \ln (e/2).$
   \item[{\rm (iii)}]
    For a fixed $a > 0$, the limit 
    \begin{equation} 
    \label{cainfcah}
    c_{a, \infty}(z) := \lim_{h\to \infty} c_{a, h} (z) \in [0, c_0] 
    \end{equation}
    exists for any $z \in [-a, \infty).$
    The function $c_{a, \infty}: [-a, \infty)\to (0, c_0]$ is continuous and decreasing.
    It is the unique nonnegative solution to the following boundary-value problem: 
    \begin{subequations}\label{vfbp:c-infinity}
    \begin{empheq}[left=\empheqlbrace]{align}   
    \label{cinfty-}
           &  c_{a, \infty}'' = 0 & & \mbox{for } z \in (-a, 0), \\
        \label{cinfty+}
        &  c_{a, \infty}'' =  \lambda (c_{a, \infty}) & & \mbox{for } z\in (0, \infty), \\
        \label{cinftyjump}
        & \qquad d_0 c_{a, \infty}'(0^-) =  c_{a, \infty}'(0^+),  &  & \\
        \label{cinftyBC}
        & \qquad c_{a, \infty}(-a) = c_0 \quad \mbox{and} \quad c_{a, \infty}(\infty) = 0. & &
    \end{empheq}
    \end{subequations}
  Moreover, $c_{a, \infty}'$ monotonically increases on $(0, \infty)$ and 
  \[
  c'_{a, \infty} (\infty):=\lim_{z \to \infty} c'_{a, \infty} (z)~=~0.
  \]
   \end{compactenum}
\end{proposition}

\begin{proof}
{Part (i).} By Proposition~\ref{p:1Dc} and Proposition~\ref{p:ah},  $ 0 < c_{a, h}(0) < c_0$ for all $a > 0 $ 
    and $c_{a, h}(0)$ is a decreasing function of $a > 0$. 
   Thus $\lim_{a\to \infty} c_{a, h}(0) $ exists and is nonnegative.   Note that 
   $ c'_{a, h}(0^-) =  ({c_{a,h}(0) - c_0})/{a} \to 0$ as $a \to \infty.$
   Hence, $c_{a, h}'(0^+) \to 0$ as $a \to \infty.$ Consequently, 
   \[
   \int_0^{h}  c''_{a, h}(z)\, dz =  c'_{a, h}(h) -  c'_{a, h}(0^+) = -  c'_{a, h}(0^+) \to 0
   \qquad \mbox{as } a \to \infty. 
   \]
   By Proposition~\ref{p:1Dc}, $c_{a, h}(z) > c_{a, h} (h) > 0$ if $0 \le z < h.$ Thus, 
   \[
   h  \lambda (c_{a, h}(h)) \le \int_0^h  \lambda (c_{a, h}(z)) \, \text{d}z 
  = \int_0^h c''_{a, h}(z)\, \text{d}z \to 0 \qquad \mbox{as } a \to \infty.
   \]
   This implies $c_{a, h}(h) \to 0$ and hence $\Lambda(c_{a, h}(h)) \to 0$ as $a \to \infty.$
   By the equation of $c_{a, h}$ on $(0, h)$, we have  
   \begin{equation}
   \dfrac{\text{d}}{\text{d}z}\left[(c'_{a,h}(z))^2-2\Lambda(c_{a,h}(z))\right]=0. 
\label{first-integral}
    \end{equation}
This and the fact that $c'_{a, h}(h) = 0$ lead to 
   \begin{equation}
   \label{cp0p}
   - [c'_{a, h}(0^+)]^2 = 2  \Lambda (c_{a, h}(h)) - 2  \Lambda (c_{a, h}(0)).
   \end{equation}
   Therefore, $\Lambda (c_{a, h}(0)) \to 0$ and $c_{a, h}(0) \to 0$ as $a \to \infty.$ 

By Proposition~\ref{p:ah}, $c_{a, h}(0)$ increases as $a $ decreases and is bounded above
by $c_0$. So, $\lim_{a \to 0} c_{a, h}(0)$ exists and is not bigger than $c_0$. If this limit is strictly less
than $c_0$, then 
\[
c'_{a, h}(0^-) =  \frac{c_{a,h}(0) - c_0}{a} \to -\infty \quad \mbox{as } a \to 0.
\]
The interface condition 
then implies that $c'_{a, h}(0^+)\to -\infty. $ However, by \eqref{cp0p}, 
$c_{a, h}'(0)$ should be bounded for all $a.$ This is a contradiction. Hence, $c_{a, h}(0) \to c_0$
as $a \to 0.$



{Part (ii).} By Proposition~\ref{p:1Dc} and Proposition~\ref{p:ah}, $c_{a, h}(0) \in (0, c_0)$
   and $c_{a, h}(h) \in (0, c_0) $ both decrease as
   $h $ increases.  
   Thus, the limits
   \begin{equation}
   \label{munu}
   \mu:=\lim_{h\to \infty} c_{a, h}(0)\qquad \mbox{and} \qquad  \nu:=\lim_{h\to \infty} c_{a, h} (h)
   \end{equation}
   exist, and $\mu \ge \nu\ge 0$ by Proposition~\ref{p:1Dc} and Proposition~\ref{p:ah}. 
   Since $c_{a,h}'(0^-) = ( c_{a, h}(0) - c_0)/a$ and $d_0 c_{a,h}'(0^-) = c_{a, h}'(0^+)$
   for all $h > 0, $
   we have the limit
   \begin{equation} 
   \label{hkprime}
   \lim_{h\to \infty} c_{a,h}'(0^+) = 
   \kappa (\mu-c_0). 
   \end{equation}
    Since $c_{a,h}(z) > c_{a, h} (h)$ for all $z \in (0, h)$ and $c_{a, h}'(h) = 0$ for
 all $h > 0$, we obtain by the equation for $c_{a, h}$ on $(0, h)$ that 
   \[
   -  c_{a, h}'(0^+) =  \int_0^{h}  c_{a, h}''(z)\, dz 
   = \int_0^{h}  \lambda(c_{a,h}(z))\, dz > \lambda
 (c_{a, h}(h)) h \ge   \lambda (\nu) h \qquad \forall h > 0. 
   \]
   This and \eqref{hkprime} imply that 
   $ 0 \le   \lambda(\nu) < - { c_{a, h}'(0^+)}/{h} \to 0 $ as  $ h \to \infty.$
    Thus $\lambda(\nu) = 0$ and hence $\nu = 0.$ 

   Sending $h\to \infty$ in \eqref{cp0p}, we obtain \eqref{definemu} 
   by \eqref{munu} with $\nu = 0$ and \eqref{hkprime}.  Define 
   \[
   \tau(x) = \kappa^2 \left( c_0 - x \right)^2 - 2  \Lambda(x) 
   \qquad   \forall x \in [0, c_0]. 
   \]
We verify that $\tau \in C^1([0, c_0])$, $\tau(0) > 0$, $\tau(c_0) < 0$, and 
$\tau'(x) < 0$ for any $x \in [0, c_0]$. Moreover, $\mu \in [0, c_0]$ and $\tau(\mu) = 0$ 
which is \eqref{definemu}.
Thus $\mu \in (0, c_0)$ is the unique solution to \eqref{definemu}. 

Since $\kappa = d_0/a$, we can rewrite \eqref{definemu} as 
\begin{equation}
\label{a2muc0mu}
a^2 = \frac{ d_0^2 (c_0-\mu)^2}{2  \Lambda (\mu)} 
\qquad \mbox{and} \qquad 
c_0 = \mu + \frac{ \sqrt{ 2\Lambda(\mu) }}{\kappa}, 
\end{equation}
respectively. Thus, with $c_0>0$ fixed, $a$ is a strictly decreasing function of $\mu \in (0, c_0)$ and 
hence, as an inverse function, $\mu = \mu(a, c_0)$ is a strictly decreasing function
of $a > 0$. Similarly, with $a $ fixed, $c_0$ is a strictly increasing function of $\mu > 0$ and
hence $\mu = \mu(a, c_0)$ is a strictly increasing function of $c_0 > 0$. 
The desired limits follow from \eqref{a2muc0mu}. The condition on
$c_0$ and $a$ for 
$\mu = 1$ can be directly verified. 



{Part (iii).}
Let $z \in [-a, \infty)$. By Proposition~\ref{p:ah},  $c_{a, h}(z)\in [0, c_0]$ 
decreases as $h > z$ increases. Thus, the limit \eqref{cainfcah} exists for any
$z \in [-a, \infty).$
exists. 
On $[-a, 0]$, each $c_{a, h}$ is linear with $c_{a,h}(-a) = c_0$
and $c_{a,h}(0) \to \mu $ as $h \to \infty.$ Thus, the convergence \eqref{cainfcah} is uniform
on $[-a, 0]$, and hence $c_{a, \infty}$ is continuous on $[-a, 0]$. 
Let $0 < A < h_0 < h$. By Proposition~\ref{p:ah}, 
$c_{a, h_0}'(\hat{z}) > c_{a, h}'(\hat{z})$ for any $\hat{z} \in (0, A).$
Integrating both sides of this inequality against $\hat{z}$ 
from $0$ to $z\in (0, A)$ and from 
$z$ to $A$, we obtain 
\[
c_{a,h_0} (0) - c_{a, h}(0) < c_{a, h_0}(z) - c_{a, h}(z) < c_{a, h_0}(A) - c_{a, h}(A)
\qquad \forall z \in (0, A).
\]
Hence, sending $h \to \infty,$ we obtain
\[
c_{a,h_0} (0) - c_{a, \infty}(0) \le c_{a, h_0}(z) - c_{a, \infty}(z) \le c_{a, h_0}(A) - c_{a, \infty}(A)
\qquad \forall z \in (0, A).
\]
Consequently, the convergence \eqref{cainfcah} is uniform on any $[0, A]$ 
and hence $c_{a, \infty}$ is continuous 
on $[0, A]$ since each $c_{a, h}$ is. Thus, $c_{a, \infty}$ is a continuous function 
on $[-a, \infty).$

Since $c_{a, h}$ is linear in $[-a, 0]$, $c_{a, h}(-a) = c_0$, $0 < c_{a, h}(0) < c_0$, 
and $c_{a, h}(0) \searrow \mu \in (0, c_0)$ as $ h \to \infty$, $c_{a, \infty}$ is linear 
in $[-a, 0]$ determined by $c_{a, \infty}(-a) = c_0$ and $c_{a, \infty} (0) = \mu.$ Hence, 
\eqref{cinfty-} and the first equation in \eqref{cinftyBC} hold true. Since
each $c_{a, h}$ decreases, by \eqref{cainfcah}, 
$c_{a, \infty}$ decreases in $[-a, \infty)$ and is bounded, the limit
 $
 c_{a, \infty}(\infty) = \lim_{z\to \infty} c_{a, \infty}(z) \in [0, c_0]
 $
 exists. Since $c_{a, h} \searrow c_{a, \infty}$ on $ (0, \infty)$ and 
 $c_{a, h} (h) \searrow 0$,   
 \[
 0 \le c_{a, \infty}(\infty)
 \le c_{a, \infty} (h) \le c_{a, h}(h) \to 0 \qquad \mbox{as } h \to \infty.
 \]
 Hence, the second equation in \eqref{cinftyBC} holds true.  

Consider now the limit of $c_{a, h}'$ on $[-a, 0]$ and $[0, \infty).$ We have 
\begin{equation} 
\label{qinfty-}
q_{a, \infty} (z) := \lim_{h\to \infty} c_{a, h}'(z) = 
\lim_{h \to \infty} \frac{c_{a, h}(0) - c_0}{a} = \frac{\mu - c_0}{a} \qquad \forall z \in [-a, 0].
\end{equation}
Let $z \in [0, \infty)$ and $h > z.$ By Proposition~\ref{p:ah}, 
$c_{a, h}'(z)$ decrease as $h$ increases. Moreover, 
\begin{equation} 
\label{cpahbound}
-c_0 \kappa < d_0 c_{a, h}'(0^-) = c_{a, h}'(0^+)
\le c_{a, h}'(z) < c_{a, h}'(h) = 0.
\end{equation}
Thus the limit 
\begin{equation} 
\label{qinfty+}
 q_{a, \infty}(z) := \lim_{h\to \infty } c_{a, h}'(z)\qquad \forall z \in [0, \infty)
 \end{equation}
exists. 
By \eqref{qinfty-},  $q_{a, \infty}< 0$ is a constant and hence
continuously differentiable  on $[-a, 0].$
By \eqref{cpahbound} and \eqref{qinfty+}, 
$-c_0 \kappa \le q_{a, \infty} \le 0$ 
on $[0, \infty)$. 
Note that $q_{a, \infty}: [-a, 0) \cup (0, \infty) \to \R$ is also a
decreasing function as each $c_{a, h}'$ is so. 

Since $c_{a, h}$ and 
$c_{a, \infty}$ are linear on $[-a, 0]$, we obtain by the fact that $c_{a, h}(0)\to c_{a, \infty}(0)
= \mu \in (0, c_0)$ that
\begin{align} 
\label{qcp-}
q_{a, \infty}(z) &= \lim_{h \to \infty}c_{a, h}'(z) = \lim_{h\to \infty} \frac{c_{a, h}(0) - c_0}{a} 
= \frac{c_{a, \infty}(0) - c_0}{a} = c_{a, \infty}'(z) \qquad \forall z \in [-a, 0].
\end{align}
Let $0 < z < A < h$. By Proposition~\ref{p:ah}, on $[0, A]$, $c_{a, h}'$ is bounded uniformly for all $h > A.$ 
By the Lebesgue Dominated
Convergence Theorem and the convergence \eqref{qinfty+}, we have 
\begin{align*}
c_{a, \infty} (z) - c_{a, \infty}(0) 
&= \lim_{h\to \infty} \left[ c_{a, h} (z) - c_{a, h}(0) \right]
= \lim_{h \to \infty} \int_0^z c_{a, h}'(\hat{z})\, d\hat{z} 
 = \int_0^z q_{a, \infty}(\hat{z})\, d\hat{z}. 
\end{align*}
Thus, 
\begin{equation} 
\label{qcp+}
q_{a, \infty} (z) = c_{a, \infty}'(z) \qquad \forall z \in [0, \infty)
\end{equation}

Since $d_0c_{a, h}'(0^-) =  c_{a, h}'(0^+)$ for all $h > 0$, we obtain
the interface condition \eqref{cinftyjump} by \eqref{qinfty-}, \eqref{qinfty+}, \eqref{qcp-}, and \eqref{qcp+}. 
Since $c_{a, h}''(z ) =  \lambda (c_{a, h}(z))$ for $0 < z < h$, we have 
\[
 c_{a,h}'(z) -  c_{a, h}'(0^+) = \int_0^z  \lambda (c_{a, h}(\hat{z}))\, \text{d}\hat{z}
 \qquad \mbox{if } 0 < z < h.
\]
Sending $h\to \infty$, we get 
\[
 c_{a,\infty}'(z) -  c_{a, \infty}'(0^+) = \int_0^z  \lambda (c_{a, \infty}(\hat{z}))\, \text{d}\hat{z}
 \qquad \forall z > 0.
\]
Thus, \eqref{cinfty+} holds true. 

We now prove the uniqueness of a continuous and nonnegative solution to \eqref{cinfty-}--\eqref{cinftyBC}. 
Let $\hat{c}_{a, \infty}$ be also a continuous and nonnegative solution to \eqref{cinfty-}--\eqref{cinftyBC}. 
Suppose $\hat{c}_{a,\infty}(0) > c_{a, \infty}(0)$.  
Let $u = \hat{c}_{a, \infty} - c_{a, \infty}.$ Then $u(0) > 0$. Since both $c_{a, \infty}$ and 
$\hat{c}_{a, \infty}$ are linear on $[-a, 0]$ and $a$ is fixed, we have  $u'(0^+) > 0.$ 
Suppose there exists $z_0 > 0$ such that $u > 0$ on $[0, z_0)$ and $u(z_0) = 0. $
Then, $\lambda (\hat{c}_{a,\infty}) \ge \lambda (c_{a, \infty})$ on $[0, z_0]$. Hence, 
for any $z \in (0, z_0)$, 
\begin{align}
\label{useful}
\hat{c}_{a, \infty}'(z) - \hat{c}_{a, \infty}'(0^+) & = 
\int_0^z \hat{c}_{a, \infty}''(\hat{z})\, \text{d} \hat{z} = \int_0^z
\lambda( \hat{c}_{a, \infty} (\hat{z})) \, \text{d} \hat{z} 
\nonumber \\
& \ge \int_0^z 
\lambda (c_{a, \infty}(\hat{z}))\, \text{d} \hat{z} 
= \int_0^z c_{a, \infty}''(\hat{z}) \, \text{d} \hat{z}
\nonumber \\
&
= c_{a, \infty}'(z) - c_{a, \infty}'(0^+). 
\end{align}
Consequently, $u'(z) \ge u'(0^+) > 0$ for any $z \in [0, z_0].$ Since $u(0) > 0$, we thus have 
\[
0 = u(z_0) = u(0) + \int_0^{z_0} u'({z}) \, \text{d}{z} \ge u(0) + u'(0^+) z_0 > 0,
\]
a contradiction. Thus, we have $\hat{c}_{a, \infty} > c_{a, \infty}$ on $[0, \infty).$ Moreover, 
the inequality in \eqref{useful} becomes strict. Thus, $u'(z) > u'(0^+) > 0$ for all $z > 0.$ Hence, 
$u(z) > u(0) > 0$ for all $z > 0.$ This contradicts $u(\infty) = 0.$
Therefore, it is impossible to have $\hat{c}_{a, \infty}(0) > c_{a, \infty}(0).$
Similarly, it is impossible to have $c_{a, \infty}(0) > \hat{c}_{a, \infty}(0).$
Hence, $\hat{c}_{a, \infty} (0) = c_{a, \infty}(0)$. Note this common value is in 
$(0, c_0).$ Since both $\hat{c}_{a, \infty}$ and 
$c_{a, \infty}$ are linear on $[-a, 0]$ and $\hat{c}_{a, \infty}(-a) = c_{a, \infty}(-a) 
= c_0$, we infer that $\hat{c}_{a, \infty} = c_{a, \infty}$  on $[-a, 0].$ 

We show now $\hat{c}_{a, \infty} = c_{a, \infty}$ on $(0, \infty).$ First, we observe
that both of them are decreasing in $(0, \infty).$ If not, we may assume there exists
$z_1 > 0$ such that, e.g., $\theta := \hat{c}_{a, \infty}'(z_1) > 0.$ Since $\hat{c}_{a, \infty }
\ge 0$ and 
$\hat{c}_{a, \infty}'' = \lambda (\hat{c}_{a, \infty}) \ge 0 $ in $(0, \infty), $
we have $\hat{c}_{a, \infty}'(z) \ge \theta$ for all 
$z \ge z_1.$ This leads to 
\[
\hat{c}_{a, \infty}(z) = \hat{c}_{a, \infty}(z_1) + \int_{z_1}^z 
\hat{c}_{a, \infty}'(\hat{z})\, \text{d}\hat{z} \ge \hat{c}_{a, \infty} (z_1) + \theta (z - z_1) \to \infty
\qquad \mbox{as } z \to \infty,
\]
contradicting $\hat{c}_{a, \infty}(\infty) = 0.$ Thus, both $\hat{c}_{a, \infty}$ and $c_{a, \infty}$
are positive and decrease on $(0, \infty).$
By the first integrals \eqref{first-integral} for both $\hat{c}_{a, \infty}$ and $c_{a, \infty}$, we observe 
that both 
of these functions are solutions to the following initial-value problem of 
differential equation for $v = v (z) > 0$ $(0 < z < \infty)$:
\begin{equation}  
\label{IVP}
 v'(z)  = - \sqrt{ 2 \Lambda (v(z) ) + \eta}\quad \forall z \in (0, \infty)
\qquad \mbox{and} \qquad v(0) = \hat{c}_{a, \infty}(0) = c_{a, \infty}(0),
\end{equation}
where 
\[
\eta =  [ \hat{c}_{a, \infty}'(0^+)]^2 - 2  \Lambda( \hat{c}_{a, \infty}(0))
=  [ c_{a, \infty}'(0^+)]^2 - 2  \Lambda ( c_{a, \infty}(0)).
\]
Since for any $m, M \in \R$ with $0 < m < M$, the function $v \mapsto \sqrt{2 
\Lambda (v) + \eta} $ is 
Lipschitz-continuous on $[m, M],$ the solution to the initial-value problem \eqref{IVP} is unique. 
Thus, $\hat{c}_{a, \infty} = c_{a, \infty}$ on $[0, \infty).$

Finally, the function $c_{a, \infty}: [-a, \infty)\to [0, c_0]$, as defined by the limit
\eqref{cainfcah}, 
is a bounded and decreasing function, 
since each $c_{a, h}: [-a, h]\to [0,c_0] $ is. 
 Since $c''_{a, \infty} > 0$ on $(0, \infty)$, $c'_{a, \infty}$ increases on $(0, \infty)$.
By the first integral for $c_{a, \infty}$ (cf.\ \eqref{first-integral}), 
\[
[c'_{a, \infty}(z)]^2 - [c'_{a, \infty}(0^+)]^2 = 2 \Lambda (c_{a, \infty}(z)) - 2 \Lambda(c_{a, \infty}(0)). 
\]
This and the fact that $c_{a, \infty}$ is bounded imply that $c'_{a, \infty}$ is bounded. Hence, 
$c'_{a, \infty} (\infty) = \lim_{z\to \infty} c'_{a, \infty}(z)$ exists.  Since $c_{a, \infty}$ decreases monotonically, 
we have $c'_{a, \infty}(z) \le 0$ for all $z \in (0, \infty)$. Hence, $c_{a, \infty}'(\infty) \le 0$. 
Suppose $c'_{a, \infty}(\infty) < 0$. Then, since $c''_{a,\infty}=\lambda(c_{a,\infty})>0$, we have $c'_{a,\infty}(z)\leq - |c'_{a,\infty}(\infty)|$ contradicting $c_{a,\infty}(\infty)=0$. 
Thus, $c_{a, \infty}'(\infty) = 0.$
\end{proof}



It is natural to expect that the nutrient concentration $c_h(z)$ decreases as one moves away from its source, the agar at $z=0$. The following proposition confirms this intuition and provides estimates
for the depletion rate.

\begin{proposition}
   \label{p:depletion}
   Given $a > 0$, $c_0 > 0$, and $h> 0.$ Let $c = c(z) $ $ (z \in [-a, h])$
   be the unique solution to \eqref{ss-vfbp-nondim-expanded}. Assume that 
 $c(0) \ge 1 \ge c(h)$. Then, there exists 
   a unique $z_\ast \in [0, h]$ such that 
   $c(z_\ast) = 1$ and
    \begin{align}
   \label{lezstar}
       & c(z) \le \left( 1 + \frac{1}{\sqrt{2}} (z_\ast - z ) \right)^2
       \qquad \mbox{if } 0 \le z \le z_\ast,
       \\
       \label{gezstar}
       & e^{-(z-z_\ast)} \le {c(z)} \le  e^{- (z - z_\ast)/\sqrt{3}}
       \qquad \mbox{if } z_\ast < z \le h. 
   \end{align}
Moreover, $z_\ast = z_\ast(q)$ depends uniquely on the flux  $q:= c'(0^+)$ with
\begin{equation}
\label{zstarq}
z_\ast(q) = 2 \left( \sqrt{\xi(q_0)} - \sqrt{\xi (q)}  \right)
= - 2  \left(  q_0 +\sqrt{\xi(q)}\right)
\qquad \forall q \in [q_0, q_h], 
\end{equation}
where $\xi(q) = [c'(z_\ast)]^2$ and 
\begin{equation}
\xi (q) =   q^2 +2  \left[ \Lambda(1) - \Lambda \left( c_0 + \frac{q}{\kappa}
\right) \right] \ge 0 \qquad \forall q \in [q_0, q_h], \label{def_of_xi}
\end{equation}
and 
$q_0$ and $q_h$ are the flux when $z_\ast = 0$ $($i.e., $c(0) = 1)$ 
and  $z_\ast = h$ 
$($i.e., $c(h) = 1)$, determined by 
$q_0 = \kappa (1-c_0) $ and $\xi(q_h) = 0$, respectively,  and $-\kappa c_0 \le q_h < q_0 < 0.$
\end{proposition}

\begin{proof}
 The existence and uniqueness of $z_\ast \in [0, h]$ follows from the fact that $c = c(z)$
is a strictly decreasing function on $[0, h]$ (cf.\ Proposition~\ref{p:1Dc}). 
Consider $z \in [0, h]$. By \eqref{first-integral} with $c$ replacing $c_{a, h}$ and $c'(h)=0$, we have the first integral 
\begin{equation}
\label{FirstIntegral}
[c'(z)]^2 = 2 \Lambda (c(z)) - 2 \Lambda (c(h))
\qquad \forall z \in [0, h]. 
\end{equation}
Consequently, by the definition of $\Lambda$ \eqref{Fc}, 
\[
[c'(z)]^2 \le 2 \Lambda (c(z)) \le  2  c(z)\qquad \forall  z \in [0, h].
\]
Since $c'(z) < 0$ for $z \in (0, h),$  we get 
$-( \sqrt{c(z)}\,)' \le 1/\sqrt{2}.$
Integrating both sides from $z  $ to $z_\ast$ and noting that 
$c(z_\ast) = 1$, we obtain 
\eqref{lezstar}. 

Consider now the case $z_\ast < z \le h$ which leads to 
$ 1 = c(z_\ast) > c(z)$.
Thus, it follows from the first integral
\eqref{FirstIntegral} and the fact
$x - \ln (1 + x) < x^2/2 $ for all $x \in (0, 1)$ that 
\[
 [c'(z)]^2 \le 2 \Lambda (c(z))\le  [c(z)]^2 
\qquad \forall z \in [z_\ast, h]. 
\]
Consequently, $-[\ln c(z) ]' \le 1 $ for all $z \in [z_\ast, h].$
Integrating both sides of this inequality from $z_\ast $ to $z$ and noting that 
$c(z_\ast)=1$, we obtain the first inequality in \eqref{gezstar}. 
Denoting 
$x = c(z) \in (0, 1)$, we have  
\[
\ln(1+x) <  x-\frac{1}{2} x^2 + \frac{1}{3} {x^3} < x - \frac{1}{6} x^2. 
\]
Therefore, by \eqref{FirstIntegral}, 
\[
 [c'(z)]^2  \ge 2  \Lambda (c(z)) = 2 
\left[ c(z) - \ln \left( 1 + c(z) \right) \right] 
\ge \frac{1}{3}  [c(z)]^2. 
\]
This leads to $- [\ln c(z) ]' \ge 1/\sqrt{3}$ for any $z \in [z_\ast, h]$. 
Integrating both sides from $z_\ast $ to $z$, we then get the second inequality in 
\eqref{gezstar}. 


By \eqref{ss-1dfinal-c-expanded}, we have for any $z \in (0, h)$ that 
$([c'(z)]^2)' = 2 [\Lambda(c(z))]'$. 
Therefore, integrating both sides of this equation from $z = 0$ to $z = z_\ast$ and noting that
$c(z_\ast) = 1$, $q= c'(0^+)$,  and 
$c(0) = c_0  + a c'(0^-) = c_0 + q/\kappa$, we have 
\begin{equation}
\label{zastq}
 [ c'(z_\ast)]^2 = {q^2}  + 2  \left[ 
 \Lambda(1) - \Lambda \left(  c_0 + \frac{q}{\kappa}\right) \right].
\end{equation}
This and the definition $\xi(q) = [c'(z_\ast)]^2$ imply \eqref{def_of_xi}. 
Note that the flux $q = c'(0^+)$ satisfies 
\begin{equation}
\label{qqq}
0 > q = c'(0^+) = d_0 c'(0^-) = \kappa [ c(0)  - c_0] > - c_0 \kappa. 
\end{equation}
Hence $\xi(q)$ is defined on 
$(-c_0 \kappa, 0).$ 
If $z_\ast  = 0$, then $c(0) = 1$. Hence, by \eqref{qqq}, 
the corresponding flux is $q_0 = \kappa (1-c_0).$
Since $c_0 > c(0) \ge 1$, we have 
$q_0 = -\kappa (c_0-1) \in (-c_0 \kappa, 0).$
If $z_\ast = h$, then $c'(z_\ast) = c'(h) = 0.$ Hence, by \eqref{zastq}, the corresponding
flux $q_h$ satisfies $\xi(q_h) = [c'(z_\ast)]^2 = 0.$
It is directly verified that $\xi'(q) < 0$ in $(-c_0 \kappa, 0)$, 
$\xi(q_0) = q_0^2>0$, 
and $\xi(0^-) < 0.$ Therefore, $q_h \in (q_0, 0) \subset (- c_0\kappa , 0)$ is the unique zero of $\xi: (- c_0 \kappa, 0) \to \R. $

Now,  $c'(z_\ast(q))= - \sqrt{\xi(q)}$ is a continuously differentiable and monotonic function of $q
\in [q_0, q_h]$. Since $c'$ is a smooth and monotonically decreasing function, it follows
that $z_\ast= z_\ast(q)$ is a smooth and monotonic function. Taking the derivative of both sides of
\eqref{zastq} with respect to $q$, we get 
\[
2 c'(z_\ast) c''(z_\ast) z_\ast'(q) = \xi'(q). 
\]
This, the fact that 
$
c''(z_\ast) =  \lambda (c(z_\ast)) = \lambda (1) = 1/2$
and \eqref{zastq} then imply 
\[
z_\ast'(q) =  \frac{\text{d}}{\text{d}q} \left( 
- 2 \sqrt{ \xi (q)} \right).
\]
Finally, noting that $0 = z_\ast (q_0)$ and  $\xi(q_0)  = [c'(z_\ast(q_0))]^2 = [c'(0^+)]^2 = q_0^2,$
we obtain \eqref{zstarq}. 
\end{proof}

We now provide 
some exponential decay estimates for the fixed time problem \eqref{vfbp-nondim} in terms of large $h$
and for the infinite-height problem \eqref{vfbp:c-infinity}. These will be used in the
next subsection to obtain
the estimate of the vertical expansion for the dynamic problem. 


\begin{proposition}
\label{p:exponential}
Let $c_h = c_h(z)$ for any $h > 0$ and $c_\infty = 
c_{\infty}(z)$ be the unique solutions to \eqref{ss-vfbp-nondim-expanded} $($with $c_h$ replacing $c)$
and \eqref{vfbp:c-infinity} $($with $c_\infty$ replacing $c_{a, \infty})$, respectively, 
given $a, $ $c_0$, $d_0$, and $\kappa=d_0/a$. 
Let $\gamma := \sqrt{2 \Lambda(c_0)}/c_0 > 0$ and 
$C_1 := 4 c_0/ \gamma>0.$
\begin{compactenum}
        \item[{\rm (i)}]
The function $c_\infty=c_\infty(z)$ satisfies 
\begin{equation}
| c_\infty'(z)| \le  c_\infty (z) \le c_0 e^{- \gamma z} \qquad \forall z > 0.
\label{exp_decay_on_c_prime}
\end{equation}
\item[{\rm (ii)}] 
We have 
\begin{equation}
        \max\limits_{-a\leq z \leq h}|c_h(z)-c_\infty(z)|\leq 
        C_1 e^{-\gamma h/2} \qquad \forall 
        h > \frac{1}{\kappa}. 
        \label{exponential-decay-of-profile-2}
        \end{equation}
        \item[{\rm (iii)}] 
       Let $\nu \in (0,c_0)$. There exists a unique $z_\ast^\infty \in (-a, \infty)$ such that 
$ c_\infty (z_\ast^\infty) = \nu.$ Assume 
        \begin{equation}\label{def_of_h_nu}
        h > h_\nu := \max \left\{\dfrac{\kappa(c_0-c_\infty(0))}{\lambda(\nu)},\dfrac{1}{\kappa}\right\}.
        \end{equation}
        Then there
        exists a unique $z_*^h \in (-a, h) $ such that $ c_h(z_*^h) = \nu$ 
        and $-a < z_*^\infty < z_*^h< h_\nu$. Moreover, 
        \begin{equation}
        |z_*^{h} - z_{*}^{\infty}|\leq C_2 e^{-\gamma  h/2} \qquad \forall h > h_\nu, 
        \label{exponential-decay-of-z_star}
        \end{equation}
        where $C_2:=C_1/|c'_\infty(h_\nu)|$ is independent of $h$. 
        
\end{compactenum}
\end{proposition}

\begin{proof} 
Part (i).   Note that  $0 < \Lambda(s)< s^2/2$ for $s > 0$. 
Let $\alpha(s) = \Lambda(s)/s^2$ for $s > 0$ (cf.\ \eqref{Fc} for the definition of $\Lambda$). 
We verify that $\alpha'(s) = s^{-3} \beta(s)$ with $\beta(0^+) = 0$ and $\beta'(s) < 0$ for all $s > 0.$
Hence, $\alpha'(s) < 0$ and $\alpha(s)$ monotonically decreases on $(0, \infty).$
Since $c_0> c_\infty(z)> 0$ for $0<z<\infty$ (cf.\ Proposition~\ref{p:1Dasymp}) (iii)), we obtain
\begin{equation}
\dfrac{\Lambda(c_0)}{c_0^2}[c_\infty(z)]^2 <\Lambda(c_\infty(z) )<\dfrac{[c_\infty(z)]^2}{2}
\qquad \forall z > 0.
\label{eqn_Lambda_estimate}
\end{equation}
The first integral for $c_\infty$, the fact that $c_\infty(\infty)  = c'_\infty(\infty) = 0$, and the 
first inequality in 
\eqref{eqn_Lambda_estimate} then lead to 
\begin{equation}
\label{CLCL}
\dfrac{[c'_\infty(z)]^2}{2}=\Lambda (c_\infty(z))>\dfrac{\Lambda(c_0)}{c_0^2}[c_\infty(z)]^2 \qquad \forall z > 0.
\end{equation}
Taking the square root and noting $c'_\infty(z) <0$ and $c_\infty(z) >0$,  we get
$
c'_\infty(z) + \gamma c_\infty(z) < 0 $ for all $z > 0,$
where $\gamma=\sqrt{2\Lambda(c_0)}/{c_0}.$
Therefore, 
\[
c_\infty(z)< c_\infty(0) e^{-\gamma z} \le c_0 e^{-\gamma z} \qquad \forall z > 0.
\]
Finally, by  \eqref{eqn_Lambda_estimate} and \eqref{CLCL}, we get 
$
[c'_\infty(z)]^2=2\Lambda(c_\infty(z))< [c_\infty(z)]^2$, leading to 
\[
 |c'_\infty(z)| \le c_\infty(z) \qquad \forall z > 0.
\]


{Part (ii).} The function $u(z) := c_h(z)-c_\infty(z)$ ($0<z<h$) satisfies
\begin{equation}
    \label{ucc}
  u''=A(z)u \quad \mbox{if } z \in (0, h), \qquad u'(0)=\kappa u(0),   \qquad 
  u'(h) = - c'_\infty(h),
\end{equation}
where $A(z)=(1+c_\infty(z))^{-1}(1+c_h(z))^{-1}$. Multiplying the first equation above by $u$
and integrating the resulting equation from $z=0$ to $z=h$, and using the integration by parts and the boundary conditions above, we 
obtain 
\begin{equation}
\kappa [u(0)]^2 +\int_0^h \left(u'\right)^2 \,\text{d}z + \int_{0}^{h}A(z) u^2 \,\text{d}z = -c_\infty'(h) u(h).\label{energy_relation_for_u}
\end{equation}
Consequently, since  $A(z)\geq  {1}/{(1+c_0)^2}$ and $|c'_\infty(h)|\leq c_0e^{-\gamma h}$ by
\eqref{exp_decay_on_c_prime} with $z=h$, we have 
\begin{equation}
\kappa [u(0)]^2 +\int_0^h \left(u'\right)^2 \,\text{d}z + \dfrac{1}{(1+c_0)^2}\int_{0}^{h} u^2 \,\text{d}z \leq c_0
e^{-\gamma h} |u(h)|.\label{prop13_estimate2}
\end{equation}
We now estimate $|u(h)|$. 
It follows from the fundamental theorem of calculus, a basic algebraic inequality, and 
the Cauchy--Schwarz inequality that 
\begin{equation}
|u(h)|^2\leq 2|u(0)|^2+ 2h \int_0^h (u')^2\,\text{d}z.\label{ftc-ineq}
\end{equation}
Thus, the right-hand side of \eqref{prop13_estimate2} can be bounded as follows: 
\[
c_0e^{-\gamma h} |u(h)|\leq c_0^2he^{-2\gamma h}+\dfrac{1}{4 h}|u(h)|^2\leq c_0^2he^{-2\gamma h}+
\dfrac{1}{2h}|u(0)|^2+ \dfrac{1}{2}\int_0^h(u')^2\,\text{d}z.
\]
This estimate and \eqref{prop13_estimate2} imply
\[
\kappa [u(0)]^2 +\int_0^h \left(u'\right)^2 \,\text{d}z + \dfrac{1}{(1+c_0)^2}\int_{0}^{h} u^2 \,\text{d}z \leq  c_0^2he^{-2\gamma h}+
\dfrac{1}{2h}|u(0)|^2+ \dfrac{1}{2}\int_0^h(u')^2\,\text{d}z.
\]
Therefore, if $h >  1/\kappa,$ then
\begin{equation}
\kappa[ u(0)]^2 + \int_0^h(u')^2\,\text{d}z \leq 2c_0^2 he^{-2\gamma h}.\label{prop13_estimate3}
\end{equation}
Similar to \eqref{ftc-ineq}, we obtain by \eqref{prop13_estimate3} and for $h > 1/\kappa$ that 
\[
\max_{0\leq z\leq h}|u(z)|^2\leq 2|u(0)|^2+ 2h \int_0^h (u')^2\,\text{d}z\leq 2h 
\left[\kappa |u(0)|^2 + \int_0^h(u')^2\,\text{d}z \right]
\leq 4c_0^2 h^2 e^{-2\gamma h}.
\]
Thus, by taking square roots of both sides, we have
\begin{equation}
\label{prop36:almost-there}
\max_{0\leq z\leq h}|u(z)|
\leq 2c_0 h e^{-\gamma h}. 
\end{equation}

Note for any $\varepsilon > 0$ that 
\[
he^{-\varepsilon h}\leq \max_{\bar{h}>0} \bar{h} e^{-\varepsilon \bar{h}}  = \left.
\bar{h}e^{-\varepsilon \bar{h}}\right|_{\bar{h}=\varepsilon^{-1}}= \varepsilon^{-1}e^{-1}.
\]
This implies that $h\leq \varepsilon^{-1}e^{-1}e^{\varepsilon h}<\varepsilon^{-1}e^{\varepsilon h}$.
Hence, by \eqref{prop36:almost-there}, we obtain that
\begin{equation}
\label{0zh}
\max\limits_{0\leq z\leq h}|u(z)|\leq 2c_0 \varepsilon^{-1} e^{-(\gamma-\varepsilon) h}. 
\end{equation}
 Taking $\varepsilon = \gamma/2$, we obtain \eqref{exponential-decay-of-profile-2}.
By the linearity of both $c_\infty$ and $c_h$ on $[a,0]$ (cf.\ Proposition~\ref{p:1Dc} (i)), 
we get 
\[
|u(z)|=|c_h(z)-c_\infty(z)| = |u(0)|\left(1-\dfrac{|z|}{a}\right),\quad -a\le z \le 0.
\]
This and \eqref{0zh} leads to \eqref{exponential-decay-of-profile-2}. 




{Part (iii).} 
The existence of $z_*^\infty$ such that $c_\infty(z_*^\infty)=\nu$ follows 
from that $c_\infty$ strictly decreases from $c_\infty(-a) = c_0$ to $c_\infty (\infty) = 0$
and that $\nu \in (0, c_0)$. 
Consider $h > h_\nu$, where $h_\nu$ is defined in \eqref{def_of_h_nu}. We show that $c_h(h_\nu) < \nu$. By contradiction assume that $c_h(h_\nu)\geq \nu$. Then $c_h(z) > c_h(h_\nu) \geq \nu$
for all $z \in [0, h_\nu)$. Hence, 
\begin{align*}
h_\nu\lambda(\nu) & <\int_0^{h_\nu} \lambda(c_h(z))\,\text{d}z
=\int_0^{h_\nu} c_h''(z)\,\text{d}z \nonumber \\
&= c'_h(h_\nu) +\kappa (c_0 -c_h(0))<\kappa (c_0 -c_\infty(0)).
\end{align*}
This leads to $h_\nu < \kappa (c_0 - c_\infty(0))/\lambda(\nu),$
contradicting the definition of $h_\nu$. 
The monotonicity of $c_h$ in $z$ then implies that there exists $z_\ast^h \in (-a, h_\nu)$
such that $c_h(z_\ast^h) = \nu$. 
By Proposition~\ref{p:1Dasymp}, $c_h (z) > c_\infty (z)$ for $z \in [-a, h]$. Thus, 
$-a < z_\ast^\infty < z_\ast^h < h_\nu.$ 

Let $ m := \min_{0 \le z \le h_\nu} | c'_\infty(z)| = |c'_\infty(h_\nu)|> 0$. 
Note that $-a < z_*^\infty < z_*^h < h_\nu <h$. 
It follows from  the fundamental theorem of calculus, the fact that 
$c_\infty(z^\infty_\ast) = c_h(z^h_\ast) = \nu$, and Part (ii) that 
\begin{align*}
m|z^h_*-z_*^\infty| \leq |c_\infty(z^h_*)-c_\infty(z^\infty_*)|
=|c_\infty(z^h_*)-c_h(z^h_*)|
\le C_1 e^{-\gamma h / 2}. 
\end{align*}
Setting $C_2=C_1/m$,  we obtain \eqref{exponential-decay-of-z_star}. 
\end{proof}

\begin{rmk}\label{rmk-z-star}
For $\nu = 1$, the
dimensionless equation $c_h(z_*^h) = 1$ means that the nutrient concentration reaches the level $K$ (the Monod constant) 
at the critical, dimensional height (from the agar-colony interface). In a spatial region 
of the colony where the nutrient is below
this level, the local growth rate is too small for cells to grow; cf.\ \cite{Warren_eLife2019,Kannan_NatCommun2025}. 
Proposition~\ref{p:exponential} implies that $z_*^h$ decreases to $z_*^\infty$ exponentially as $h \to \infty.$ 
This indicates the vertical growth is limited as only a thin layer of cells near the agar-colony interface
can grow effectively,
agreeing with the agent-based simulations \cite{Warren_eLife2019,Kannan_NatCommun2025}. 
\end{rmk}

\subsection{Dynamics of vertical expansion}
\label{ss:1Ddynamics}

\begin{theorem}[Well-posedness of the moving-boundary problem \eqref{vfbp-nondim-expanded}]
\label{t:verticalmoving} 
Given $a, c_0, d_0, h_0$ all positive and set $\kappa = d_0/a.$
The system \eqref{vfbp-nondim-expanded} admits a unique global solution
\(
(c(z,t),h(t))
\)
with $ h\in C^2([0,\infty))$ and $h (t) > 0$ for all $t > 0$ and 
$c(\cdot, t)$ is linear on $[-a, 0]$ and 
$c(\cdot,t)|_{[0, \infty) } \in C^\infty([0,h(t)]) $ for all $t > 0.$ Moreover, 
$0< c(\cdot,t)< c_0$ and $c(\cdot, t)$ decreases strictly
on $(-a, h(t))$ for all $t > 0$ and $\partial_t c (\cdot, t) <0$ on $(0, h(t))$ for all $t > 0$. 
\end{theorem}

\begin{proof} We first prove the existence and uniqueness of 
$h(t)\in C^1([0,\infty))$. 
For any $h>0$, there exists a unique
$c_h = c_h(z)\ge 0$ $(-a \le z \le h)$ that satisfies the time-independent version of the 
first four equations \eqref{1dsemifinal-c-expanded}--\eqref{1dsemifinal-BC_at_a-expanded} of the system
\eqref{1dsemifinal-h-expanded}, with $c_h$ linear on $[-a, 0]$, $c_h|_{[0, h]} \in C^\infty ([0, h])$, 
$c_h$ strictly decreasing, 
and  $0 < c_h(z) < c_0$ for all $z \in (-a, h]$; cf.\ Proposition~\ref{p:1Dc}. 
The initial-value problem for $h = h(t)$ can be written as $h' = G(h)$ with $h(0) = h_0 > 0,$ where by 
\eqref{1dsemifinal-h-expanded}, \eqref{simplification-c}, and \eqref{simplification-h}
\[
G(h) = -\kappa (c_h(0) - c_0) = -c_h'(0^+) = \int_0^h \lambda (c_h(z)) \, dz. 
\]
Let $0 < h_1 < h_2$. By Proposition~\ref{p:ah}, 
$c_{h_1}(z) > c_{h_2}(z)$ if $z \in [-a, h_1].$ In particular, $c_{h_1}(0) > c_{h_2}(0)$ and hence
$G(h_1) < G(h_2).$ Consequently, 
\begin{align}
\label{GG}
 |G(h_2) - G(h_1) | & = G(h_2) - G(h_1) \\
 &= \int_0^{h_2} \lambda (c_{h_2}(z))\, dz -  \int_0^{h_1} \lambda (c_{h_1} (z))\, dz \\
& = \int_{h_1}^{h_2} \lambda (c_{h_2}(z))\, dz + \int_0^{h_1}
\left[ \lambda (c_{h_2}(z)) - \lambda (c_{h_1}(z) )\right] dz\nonumber \\
& \le \int_{h_1}^{h_2} \lambda (c_{h_2}(z))\, dz \nonumber \\
& \le |h_2 - h_1|.
\end{align}
Therefore, $G$ is a globally Lipschitz-continuous function on $[0, \infty).$ Thus, 
the initial-value problem $h'(t) = G(h(t))$ with $h(0) = h_0$ has a unique global
solution $h(t) $ for all $t > 0.$ 

Since $h'(t) = G(h(t)) > 0$ for all $t > 0$, we have 
$h(t) > h_0 > 0$ for all $t > 0$. Moreover, for any $t > 0$, $c(\cdot, t) = c_{h(t)}(\cdot)$. 
By Proposition~\ref{p:1Dc}, 
$c(\cdot, t)$ is linear on $[-a, 0]$, 
$c(\cdot,t)|_{[0, \infty) } \in C^\infty([0,h(t)]) $, and 
$0< c(\cdot,t)< c_0$ and $c(\cdot, t)$ strictly decreases on $(-a, h(t))$ for all $t > 0.$
 Moreover, by  Propositions~\ref{p:ah}, $\partial_t c(\cdot, t) < 0$ on $(0, h(t))$ for all $t > 0.$

It remains now to show that $G$ is a $C^1$-function of $h > 0$, which implies that 
$h(t)\in C^2([0,\infty))$. We proceed in three steps. 

{\it Step 1.} For each $h>0$, the restriction of $c_h$ onto $[0, h]$ satisfies that
$c_h\in C^\infty([0, h])$, $0 < c_h < c_h(0) < c_0$ and $c_h'<0$ on $(0, h)$, and 
\[
c_h'' = \lambda (c_h) \quad \mbox{in } (0, h), \qquad c_h'(0) = \kappa (c_h(0) - c_0), \qquad c_h'(h) = 0; 
\]
cf.\ \eqref{1dfinal-c}--\eqref{1dfinal-BC-at-h}. Define $\hat{c}_h(\hat{z}) = c_h(z)$ with 
$\hat{z} = z/h \in [0, 1]$ and $z \in [0, h]$. Then $\hat{c}_h \in C^\infty([0, 1])$, 
$0 < \hat{c}_h <  \hat{c}_h(0) < c_0$, $\hat{c}_h' < 0$ on $(0, 1)$, and 
\begin{equation}\label{s3:ifp-bvp}
\hat{c}_h'' = h^2 \lambda (\hat{c}_h) \quad \mbox{in } (0, 1), \qquad \hat{c}_h'(0) = h \kappa (\hat{c}_h(0) - c_0), 
\qquad \hat{c}_h'(1) = 0.
\end{equation}
It suffices to show that $\hat{c}_h(0)=c_h(0)=c_0-\kappa^{-1}G(h)$ is a $C^1$-function of $h > 0$. 

{\it Step 2.} We claim the following: for any $h > 0$,  there exists  $\delta > 0$ such 
that the following initial-value problem of a second-order ODE
\begin{equation}\label{s3:ifp-ivp}
\hat{w}_{h,b}''=h^2 \lambda (\hat{w}_{h,b}) \quad\text{in }(0,1],
\qquad \hat{w}_{h,b}(0)=b, \qquad \hat{w}_{h,b}'(0)=h \kappa (b-c_0) 
\end{equation}
has a unique solution 
$\hat{w}_{h, b}\in C^2([0, 1])$ for any $b \in (\hat{c}_h(0) - \delta, \hat{c}_h(0) + \delta)$, and the solution
$\hat{w}_{h,b}$ is positive. 
To prove this claim, we introduce $\hat{v}_{h, b} = \hat{w}_{h,b}'$ and consider the following equivalent initial-value problem
for a system of first-order ODEs: 
\begin{equation}
    \label{wv}
    \left\{
    \begin{aligned}
        & \hat{w}_{h,b}' = \hat{v}_{h, b} \qquad \mbox{on } (0, 1], \\
        & \hat{v}_{h,b}' = h^2 \lambda( \hat{w}_{h, b}) \qquad \mbox{on } (0, 1], 
    \end{aligned}
    \right. \qquad \mbox{and} \qquad 
    \left\{
    \begin{aligned}
        &\hat{w}_{h, b}(0) = b, \\
        &\hat{v}_{h,b}(0) = h \kappa (b-c_0). 
    \end{aligned}
    \right.
\end{equation}


Consider the open box $B = (\hat{c}_h(1)/2, c_0)\times (-h \kappa c_0, 1).$
By Proposition~\ref{p:1Dc}, the
trajectory $\mathcal{T}:=\{(\hat{c}_h(\hat{z}),\hat{c}'_h(\hat{z})):0\leq \hat{z}\leq 1\}$ is contained in $B$. 
Set 
\begin{align*}
   & \varepsilon =  \min \left\{ 1, \frac12 \mbox{dist}({\mathcal T}, \partial B), c_0 - \hat{c}_h(0), \hat{c}_h(0) 
    - \hat{c}_h(1) \right\}>0,\\
&\delta =\dfrac{\varepsilon}{\sqrt{e^{h^2+1}(1+h^2 \kappa^2)}} \in (0, \varepsilon).
\end{align*}
Let  $b \in (\hat{c}_h(0) - \delta, \hat{c}_h(0) + \delta)$. Then, 
the initial value in \eqref{wv} satisfies
$(b, h \kappa (b - c_0)) \in B$. 
Since the right-hand side of the system of differential equations in 
\eqref{wv} is $C^{\infty}$  with respect to the unknowns $(\hat{w}_{h,b},\hat{v}_{h,b})$ in $B$, the solution
$(\hat{w}_{h,b}, \hat{v}_{h,b})$ exists locally. 
Let $[0, \hat{z}_\infty) $ for some $\hat{z}_\infty\in (0, \infty]$ be the maximal interval of the solution. 
It suffices to show the following two claims: 
\begin{compactenum} 
\item[(1)] $\hat{z}_\infty > 1$, which implies the solution 
$\hat{w}_{h, b} $ is defined on $[0, 1]$; and 
\item[(2)] the trajectory $\mathcal S := \{
(\hat{w}_{h, b}(\hat{z}), \hat{v}_{h,b}(\hat{z})): 0 \le \hat{z} \le 1\} $ is contained in $\overline{B},$
which implies that the solution $\hat{w}_{h,b}$ is positive on $[0, 1].$
\end{compactenum}

We will prove these two claims by contradiction. 
Assume that either  $\hat{z}_\infty \le 1$ or the trajectory $\mathcal S$ 
exists but is not entirely contained in $\overline{B}$. In either case, there exists $\hat{z}_*\in (0,1) $ such that the trajectory $\mathcal S$ reaches $\partial{B}$ at $\hat{z}_*\in(0,1)$ in the following sense  
\begin{equation}
\label{hhzstar}
(\hat{w}_{h,b}(\hat{z}_*),\hat{w}'_{h,b}(\hat{z}_*))\in \partial B \quad \mbox{and} \quad 
(\hat{w}_{h,b}(\hat{z}),\hat{w}'_{h,b}(\hat{z}))\in B \quad \mbox{for } 0\leq \hat{z}<\hat{z}_*.
\end{equation}  
Indeed, \eqref{hhzstar} is precisely the negation of the second claim. On the other hand, if $\hat{z}_\infty$ is finite, in particular, $\hat{z}_\infty\leq 1$, then we have that  
\(
\lim\limits_{\hat{z}\to \hat{z}_\infty}\|(\hat{w}_{h,b}(\hat{z}),\hat{w}'_{h,b}(\hat{z})\|_{\R^2}=\infty.
\) 
Therefore, the trajectory $\{
(\hat{w}_{h, b}(\hat{z}), \hat{v}_{h,b}(\hat{z}))\}$ reaches $\partial B$ at some $\hat{z}_*\in(0,1)$ in the sense of \eqref{hhzstar}.

To arrive at a contradiction, we will obtain a priori estimates on $\hat{w}_{h,b}-\hat{c}_h$ for $0\leq \hat{z}\leq \hat{z}_*$. To this end, write 
\[
\hat{w}_{h,b}''-\hat{c}_{h}''=h^2\left(\lambda(\hat{w}_{h,b})-\lambda(\hat{c}_h)\right).
\]
Multiplying this equality by $\hat{w}_{h,b}-\hat{c}_{h}$, noting that $|\lambda (u) - \lambda(v)|\le |u - v|$ 
for any $u, v > 0$, and applying the Cauchy--Schwarz inequality, we  obtain
\begin{align*}
&\dfrac{\text{d}}{\text{d}\hat{z}}\left[|\hat{w}_{h,b}-\hat{c}_h|^2+|\hat{w}'_{h,b}-\hat{c}'_h|^2\right]\\
&\qquad=2(\hat{w}_{h,b}'-\hat{c}_h')(\hat{w}_{h,b}-\hat{c}_h)
+2 h^2(\lambda(\hat{w}_{h,b})-\lambda(\hat{c}_h))(\hat{w}_{h,b}'-\hat{c}_h')
\\
&\qquad \le (h^2+1) \left[|\hat{w}_{h,b}-\hat{c}_h|^2+|\hat{w}'_{h,b}-\hat{c}'_h|^2\right]
\qquad \forall \hat{z} \in [0, \hat{z}_\ast]. 
\end{align*}
By the Gronwall inequality, we get
\begin{align*}
 &|\hat{w}_{h,b}(\hat{z})-\hat{c}_h(\hat{z})|^2+|\hat{w}'_{h,b}(\hat{z})-\hat{c}'_h(\hat{z})|^2
 \nonumber \\
 &\qquad \leq e^{(h^2+1)\hat{z}}\left[|\hat{w}_{h,b}(0)-\hat{c}_h(0)|^2+|\hat{w}'_{h,b}(0)-\hat{c}'_h(0)|^2\right]
 \nonumber \\
 &\qquad \leq e^{(h^2+1)}(1+\kappa^2 h^2)\delta^2\nonumber \\
 &\qquad = \varepsilon^2 
 \qquad \forall \hat{z} \in [0, \hat{z}_\ast].
\end{align*}
In particular, 
\[
|\hat{w}_{h,b}(\hat{z}_*)-\hat{c}_h(\hat{z}_*)|^2+|\hat{w}'_{h,b}(\hat{z}_*)-\hat{c}'_h(\hat{z}_*)|^2\leq \varepsilon^2.
\]
This contradicts the fact that
$(\hat{w}_{h,b}(\hat{z}_*),\hat{w}'_{h,b}(\hat{z}_*))\in \partial B$ and $\text{dist}(\mathcal{T},\partial B) \ge 2\varepsilon$.

{\it Step 3.} 
We define
\[
\mathcal{F}(h,b) :=\hat{w}_{h,b}'(1),
\]
where $\hat{w}_{h,b}\in C^2([0, 1])$ is the (positive) solution to \eqref{s3:ifp-ivp}.
By the standard theorems about continuous dependence of ODE solutions on initial conditions and parameters
(cf.\ \cite[Chapter 5]{hartman2002ordinary}, \cite[Chapter 3.5]{brauer1989qualitative}, and \cite[Chapter 2.3]{Perko2001}), 
we deduce that, given arbitrary $h_*>0$, there exist $\delta_1>0$ and $\delta_2>0$ such that $\hat{w}_{h,b}'(1)$ exists and continuously differentiable for all $(h,b)\in I_{\delta_1,\delta_2}$, where
\[
I_{\delta_1, \delta_2} :=\{ (h, b) \in \R^2: |h-h_*|<\delta_1\ \mbox{and} \  |b-\hat{c}_{h_*}(0)|<\delta_2\}.
\]
Hence, $\mathcal F \in C^1(I_{\delta_1,\delta_2})$. 
Moreover, by the definitions of $\mathcal{F}$ and $\hat{c}_h(\hat{z})$ we have 
\[
\mathcal{F}(h,\hat{c}_h(0))=0 \qquad \forall h >0.
\]
Finally, we show that $\partial_b\mathcal{F}(h_*,b)\neq 0$ with $b=\hat{c}_h(0)$. Indeed, by \eqref{s3:ifp-ivp}, the function $u:= \partial_b \hat{w}_{h_*,b}$ satisfies 
\[
u''=A(\hat{z})u \text{ on }(0,1], \qquad u(0)=1, \qquad u'(0)=\kappa h_*. 
\]
with 
$A(\hat{z}) = h^2_* \lambda'(\hat{w}_{h_*, b}(\hat{z}))\ge A_0$ on $[0, 1]$ for some constant $A_0 > 0,$
and both $u(0)$ and $u'(0)$ are strictly positive. Then, simply due to the continuity of $u$, 
we have $u(\hat{z})>0$ at least for some small interval $\hat{z}\in (0,\hat{z}_0)$. 
However, due to $u''=A(\hat{z})u$, we also have  $u''(\hat{z})>0$ on this interval, 
so $u'(\hat{z})>u'(0)=\kappa h_*$ for $\hat{z}\in (0,\hat{z}_0)$ and both $u(\hat{z})$ 
and ${u}'(\hat{z})$ can only grow staying strictly positive.  Therefore, $u'(\hat{z}) > 0$ for all $\hat{z}\in [0, 1]$. 
Hence, 
\[
\partial_b \mathcal{F}(h_\ast, \hat{c}_{h_\ast}(0))
= u'(1) > 0
\qquad \forall\, h_*>0.
\]
Consequently, we are in a position to apply the implicit function theorem and 
conclude that $h\mapsto \hat{c}_h(0)$ is a $C^1$-function of $h$.
Hence, $G \in C^1((0, \infty)).$
\end{proof}

\begin{theorem}
    \label{t:1Ddynamics}  Let \((c(z,t),h(t))\) with $h\in C^2([0,\infty))$ and 
$c(\cdot,t)\in C^\infty([h_0,h(t)])$  for all $t \geq 0$ be the solution of \eqref{vfbp-nondim-expanded}, given $a$, $d_0$, $c_0$, 
and $h_0,$ with $\kappa = d_0/a.$ Let $c_{a, \infty} = c_{a, \infty}(z)$ $(z \ge -a)$ be the solution 
to \eqref{ss-vfbp-nondim-expanded}, given in Proposition~\ref{p:1Dasymp}.
    \begin{compactenum}
    \item[{\rm (i)}] We have $h'(t) > 0$ and $h''(t) > 0$ for any $t > 0$. Moreover, 
    \begin{align} 
    \label{3lim-1}
    &\lim_{t\to \infty} c(0, t) = c_{a, \infty}(0), \\
    \label{3lim-2}
    &
    \lim_{t\to \infty}h'(t) = \kappa(c_0-c_{a, \infty}(0)). 
    \end{align}
\item[{\rm (ii)}] 
Assume $c_0>1$. Then there exists $t_0\geq 0$ and $z_*:[t_0,\infty)\to(-a,h(t)]$ 
such that $c(z_*(t),t)~=~1$ for all $t\geq t_0$. Moreover, 
\begin{equation}\label{zastexp-MP}
|z_*(t)-z^\infty_*|\leq C_3e^{-\gamma_1(t-t_0)} \quad \forall t>t_0,
\end{equation}
where $C_3=C_2e^{\gamma h(t_0)/2}$ and $\gamma_1=\gamma h'(t_0)/2$, with
$C_2$ and $\gamma$ same as in Proposition~\ref{p:exponential}.

\end{compactenum} 
\end{theorem}

\begin{proof} 
{Part (i).} By \eqref{1dfinal-h} and Proposition~\ref{p:1Dc} (iii),
we have $h'(t) > 0$ for any $t > 0$. 
Thus $h(t_1) < h(t_2)$ if $0 < t_1 < t_2$. Consequently, since $c(z, t) = c_{a, h(t)}(z)$ for any $z \in [-a, h(t)]$, 
we get by Proposition~\ref{p:ah} (ii)  that 
 $c(0, t_1) > c(0, t_2)$. Thus, by \eqref{1dfinal-h}, 
$h'(t_1) < h'(t_2)$. Hence  $h''(t)> 0$ for any $t > 0$.


If we fix $\tilde{t}> 0$ and note that $c(0, \tilde{t})< c_0$, then we have by \eqref{1dfinal-h} that 
$h'(t) \ge \kappa (c_0 - c(0, \tilde{t})) > 0$ for all $t \ge \tilde{t}. $ Thus, 
$h(t)\to \infty$ as $t\to \infty$. Hence, by Proposition~\ref{p:1Dasymp} (i), we obtain
  $c(0, t) = c_{a, h(t)}(0) \to c_{a, \infty}(0)$, implying
  \eqref{3lim-1}. This and \eqref{1dfinal-h} then imply \eqref{3lim-2}. 



{Part (ii).} Note that $c(z, t) = c_{a,h(t)}(z)$ for all $z \in [-a, h(t)]$ for all $t > 0.$
The existence of $t_0 > 0$ such that 
$c(h(t_0), t_0) = 1$ follows from that fact that $c(z,t)$ monotonically decreases
as $t$ increases, $c(0, t) \ge c_{a, \infty}(0)$ for all $t > 0$, $h(t) \to \infty$ and 
$c(h(t), t) \to 0$ as $t \to \infty,$ 
 and Proposition~\ref{p:exponential}. The existence and uniqueness of
 $z_\ast(t)\in (-a,h(t)]$ for each $t > t_0$ such that $c(t, z_\ast(t)) = 1$ follows from
Proposition~\ref{p:exponential}. 

 Assume $c_0>1$. From $h''(t)>0$, we obtain that $h'(t)>h'(t_0)>0$ for $t>t_0$. Then by integration we get  
\[
h(t)>h'(t_0)(t-t_0)+h(t_0) \qquad \forall t > t_0.
\]
Applying \eqref{exponential-decay-of-z_star} 
to $h=h(t)$, we then conclude 
\[
|z_\ast (t)-z_{\ast}^{\infty}|  \le C_2 e^{\gamma h(t_0)/2} e^{-\gamma h'(t_0)(t-t_0)/2} \qquad \forall t > t_0, 
\]
where $C_2$ and $\gamma$ are given in Proposition~\ref{p:exponential}. This leads to \eqref{zastexp-MP}.
\end{proof}


\begin{rmk} 
If $c_0< 1$, 
then neither $z_*(t)$ nor $z_*^{\infty}$ exist for all $t>0$, since both $c(z, t)$ and $c_\infty(z)$
are less than $c_0.$
If $c_0>1$, then Theorem~\ref{t:1Ddynamics} (ii) applies. If $c_\infty(0)\geq 1$, then $z_*(t)>z_\infty\geq 0$ for all $t\geq t_0$. 
In this case, the nutrient concentration reaches the Monod constant 
level (see also Remark~\ref{rmk-z-star}) in the colony region $[0, h]$.
Theorem~\ref{t:1Ddynamics} above implies that this threshold concentration for effective cell growth is reached
within the bacterial colony only during a finite time interval, predicting the vertical slowdown as experimentally observed.
\end{rmk}

\section{A reduced model for radial expansion}
\label{s:disk}

\subsection{Model formulation}
\label{ss:diskmodel}

In this section, we consider a reduction of the original model described in
Section~\ref{s:ContinuumModel}, focusing on the radial
expansion of the bacterial colony. Our two main simplifying assumptions are as follows: 
First, we neglect the vertical coordinate $z$, thereby assuming that the
bacterial colony is a disk with a time-dependent radius $R(t)$.
This assumption of flatness is appropriate for a monolayer colony during an early stage of growth.
 Second, since the colony is much smaller than the substrate in size, 
we assume that $W = \infty$, allowing the colony to grow without bound. 
Both nutrient concentration $c$ and pressure $p$ are functions of $r=\sqrt{x^2+y^2}$ (we drop prime from notations of radius vector)
and  time $t$.  Analogously to the vertical growth from Section~\ref{s:1D}, we shall
consider the quasi-stationary  dynamics, 
i.e., setting the $\eta$-terms  in \eqref{eq:nondimensional_model}(a,b) to be $0$. 

Repeating arguments of Section~\ref{ss:3Dmodel}, we reduce  
\eqref{eq:nondimensional_model} to the following set of differential equations and boundary conditions:  
\begin{subequations}\label{vfbp2d-dim}
\begin{empheq}[left=\empheqlbrace]{align}
&D_+\Delta_r c = \alpha \lambda_B\dfrac{c}{c+K}, & & 
\mbox{for } r\in (0,R(t)) \text{ and }t>0,\label{vfbp2d-dim-eq-for-c}\\
& \qquad \partial_rc(0,t)=0 \text{ and }c(R(t),t) = c_0, & & 
\mbox{for } t>0,\label{vfbp2d-dim-bc-for-c}\\
& \Delta_r p = -\dfrac{\lambda_B}{\xi}\dfrac{c}{c+K}, & & 
\mbox{for } r\in (0,R(t)) \text{ and }t>0, \label{vfbp2d-dim-eq-for-p}\\
& \qquad \partial_rp(0,t)=0 \text{ and }p(R(t),t) = 0, & & 
\mbox{for } t>0,\label{vfbp2d-dim-bc-for-p}\\
& R'(t) = - \xi\partial_r p (R(t), t), 
& & \mbox{for } t>0,\label{vfbp2d-dim-eq-for-R}\\
& \qquad R(0)=R_0.\label{vfbp2d-dim-ic-for-R}
\end{empheq}
\end{subequations}
Here, the parameters $D_+$, $\alpha=\rho/Y$, $\lambda_B$, $K$ and $\xi$ have 
the same meaning as in Subsection~\ref{ss:1Dmodel}. 
Also, $\Delta_r u(r)= r^{-1}\partial_r\left(r\partial_r u(r)\right)$ is the Laplacian in 
polar coordinates for a radially symmetric function $u=u(r)$. 
Conditions at $r=0$ in \eqref{vfbp2d-dim-bc-for-c} 
and \eqref{vfbp2d-dim-bc-for-p} results from the radial symmetry of $c$ and $p$, respectively. 


Let us introduce the dimensionless variables, 
\[
\hat{r}=\dfrac{r}{\mathcal{R}}, \quad 
\hat{R}=\dfrac{R}{\mathcal{R}}, \quad 
\hat{c}=\dfrac{c}{K}, \quad 
\hat{t}=\dfrac{t}{\mathcal{T}}, \quad 
\hat{p}=\dfrac{p}{\mathcal{P}}.
\]
with the following characteristic scales
\[
\mathcal{R} = \sqrt{\dfrac{KD_+}{\alpha \lambda_B}},\qquad \mathcal{P}=\dfrac{\lambda_B \mathcal{R}^2}{\xi},\qquad \mathcal{T}=\dfrac{1}{\lambda_B}.
\]
Without using new notations, we obtain the following system of equation in the dimensionless form: 
\begin{subequations}\label{vfbp2d-nondim}
\begin{empheq}[left=\empheqlbrace]{align}
&\Delta_{r} c = \dfrac{c}{c+1} & & 
\mbox{for } r\in (0,R(t)) \text{ and }t>0,\label{vfbp2d-nondim-eq-for-c}\\
& \qquad \partial_rc(0,t)=0 \text{ and } c(R(t),t) = c_0 & & 
\mbox{for } t>0,\label{vfbp2d-nondim-bc-for-c}\\
& \Delta_{r} p = -\dfrac{c}{c+1} & & 
\mbox{for } r\in (0,R(t)) \text{ and }t>0, \label{vfbp2d-nondim-eq-for-p}\\
& \qquad \partial_rp(0,t)=0 \text{ and }p(R(t),t) = 0 & & 
\mbox{for } t>0,\label{vfbp2d-nondim-bc-for-p}\\
& R'(t) = - \partial_{r} p (R(t), t) & & 
\mbox{for } t>0,\label{vfbp2d-nondim-eq-for-R} \\
& \qquad R(0)=R_0.\label{vfbp2d-nondim-ic-for-R}
\end{empheq}
\end{subequations}
Note by \eqref{vfbp2d-nondim-eq-for-c} and \eqref{vfbp2d-nondim-eq-for-p} that the function $c+p$ is harmonic. Moreover, it follows from \eqref{vfbp2d-nondim-bc-for-c} and \eqref{vfbp2d-nondim-bc-for-p} that the function $c+p$ satisfies the Dirichlet boundary condition $c+p=c_0 $ at $r = R$. 
 Hence, $c+p = c_0$ on the disk $r < R(t)$ for all $t > 0$. 
Eliminating the pressure $p = c_0 - c$, we obtain the following compact form of the dimensionless system: 
\begin{subequations}\label{vfbp2d-nondim-compact}
\begin{empheq}[left=\empheqlbrace]{align}
&\Delta_{r} c = \dfrac{c}{c+1} & & 
\mbox{for } r\in (0,R(t)) \text{ and }t>0,\label{vfbp2d-nondimc-eq-for-c}\\
& \qquad \partial_rc(0,t)=0 \text{ and } c(R(t),t) = c_0 & & 
\mbox{for } t>0,\label{vfbp2d-nondimc-bc-for-c}\\
& R'(t) = \partial_{r} c(R(t), t) & & 
\mbox{for } t>0,\label{vfbp2d-nondimc-eq-for-R}\\
& \qquad R(0)=R_0.\label{vfbp2d-nondimc-ic-for-R}
\end{empheq}
\end{subequations}

 \subsection{Numerical simulations and perturbation analysis}
\label{ss:numersimul-disk}

We solved numerically the system of equations \eqref{vfbp2d-nondim-compact}
and summarized our results in Figure~\ref{fig:1D}. We observe
from Figure~\ref{fig:1D} (a) that the nutrient 
concentration $c$ is a monotonically increasing function of $r$ for a fixed $t$. 
The concentration at the center of the colony $c(0, t)$ decreases to $0$ as 
$t\nearrow\infty$. 
In Figure~\ref{fig:1D} (b), we compare the solution $c(r, t)$ at 
$t = T_2$ for some $T_2 > 0$ with the solution $\bar{\bar{c}}(R(T_2) - r)$ which 
is obtained from our asymptotic analysis (cf.\ \eqref{bar_bar_C}). A good agreement 
is reached. 
Figure~\ref{fig:1D} (c) shows 
that the radius
$R(t)$ grows linearly as time evolves. 

 \begin{figure}[tp]
\begin{center}
\includegraphics[width=\textwidth]{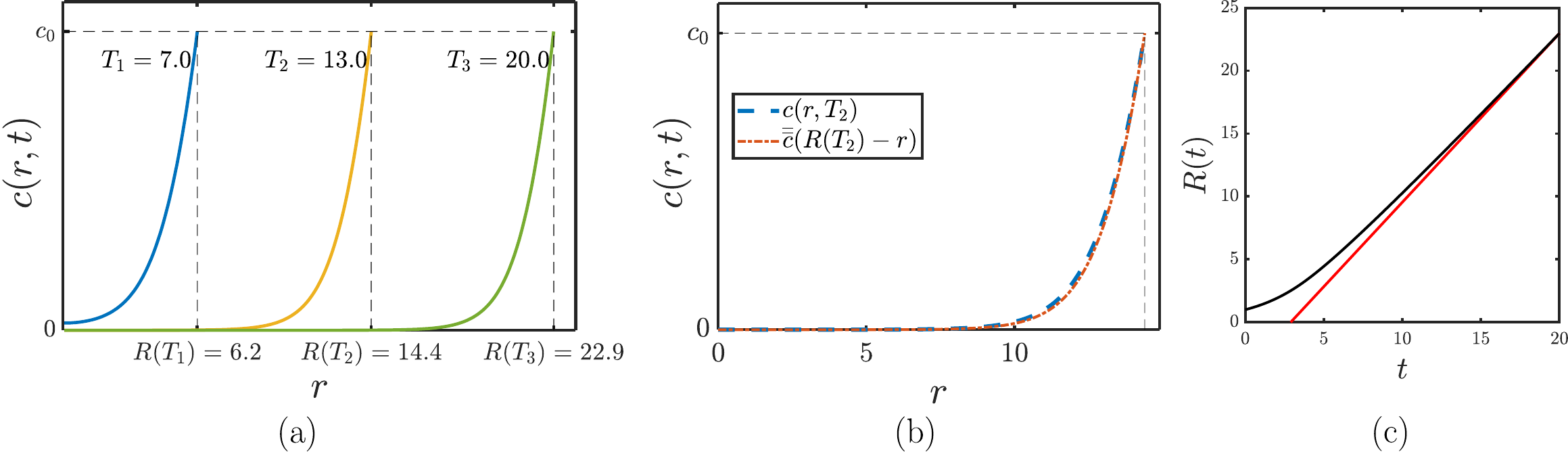}
\end{center}
\vspace{-4 mm}

\caption{ 
Results of the numerical simulation of the system \eqref{vfbp2d-nondim-compact} with
 $c_0=2$ and $R(0)=1.0$. 
(a) Concentration profiles $c(r,t)$ for  $t = T_1,T_2,T_3$.
 (b) Comparison between the concentration profile $c(r,t)$  and its approximation via \eqref{approx_C_R} for $t=T_2$. (c) The colony radius
$R(t)$ (black solid line) and the slope 
$-\partial_z\bar{\bar{c}}(0^+)$
(red line; namely, it is plot for $R=R(20)-\partial_z\bar{\bar{c}}(0^+)(t-20)$) vs.\ time $t$.}
\label{fig:1D}
\end{figure} 

We now use the standard perturbation analysis \cite{holmes2012introduction} to address the large-time (or, equivalently, large-$R$) behavior of the solution to the system \eqref{vfbp2d-nondim-compact}. 
 Introduce new variable $\xi := r/R$ and parameter $\varepsilon = R^{-2}$. Then the 
 time-independent version of the 
 problem \eqref{vfbp2d-nondimc-eq-for-c}--\eqref{vfbp2d-nondimc-bc-for-c}, 
 with $c(r)$ and $R$ replacing $c(r, t)$ and $R(t)$, respectively, is rewritten for $\bar{c}(\xi)= c(R\xi)$ as 
 \begin{equation}
 \label{barbarc}
 \left\{
 \begin{array}{l}
 \varepsilon\left(\xi  \bar{c}'(\xi) \right)' = 
 \xi \lambda(\bar{c}(\xi)) \qquad \mbox{for } 0<\xi<1, \\
 \bar{c}'(0) = 0 \qquad \mbox{and} \qquad 
 \bar{c}(1)  = c_0.
 \end{array}
 \right.
 \end{equation}
This is a singularly perturbed problem with vanishing outer solution that cannot satisfy the boundary condition at $\xi = 1$.
Therefore, similar to the analysis in Subsection~\ref{ss:1Dsimulation}, we represent the solution $\bar{c}$ as the sum of the inner (boundary layer) solution of the fast variable $\eta:=\varepsilon^{-\gamma}(1-\xi)$ for some $\gamma>0$ and the outer solution, depending regularly on both $\xi$ and $\varepsilon$. The inner and outer solutions are coupled with the matching conditions: the limit as $\eta\to \infty$ of the inner solution equals the limit as $\xi\to 1$ of the outer solution. As in Subsection~\ref{ss:1Dsimulation}, we find
that the outer solution vanishes.

To determine the value of $\gamma$, we substitute
$\bar{\bar{c}}(\eta) = \bar{c}(\xi) = \bar{c}(1-\varepsilon^{\gamma}\eta)$  into \eqref{barbarc} to derive the equation $\bar{\bar{c}}(\eta)$ 
\begin{equation*}
\varepsilon^{1-2\gamma}\partial^{2}_{\eta}\bar{\bar{c}}-\varepsilon^{1-\gamma}\left(\eta \partial^2_\eta \bar{\bar{c}}+ \partial_\eta \bar{\bar{c}}\right)=\lambda (\bar{\bar{c}}) -\varepsilon^{\gamma} \eta \lambda(\bar{\bar{c}}).
\end{equation*}
To have a non-trivial equation for $\bar{\bar{c}}$ at the leading order of $\varepsilon$, we choose $\gamma = 1/2$, and then the boundary value-problem for $\bar{\bar{c}}$ is 
\begin{equation}\label{bar_bar_C}
\left\{
\begin{array}{l}
\bar{\bar{c}}''= \lambda(\bar{\bar{c}})  \qquad 
\mbox{for } \eta>0,\\
\bar{\bar{c}}(0)  = c_0 \qquad 
\mbox{and} \qquad 
\bar{\bar{c}}(\eta)\to 0 \quad \text{as }\eta \to \infty.
\end{array}
\right.
\end{equation}
Here, the condition at $\eta\to\infty$ is derived from  the matching condition. 
Noting that 
\[
\eta = \varepsilon^{-1/2}(1-\xi) = R \left( 1-\frac{r}{R} \right)=R-r,
\]
we have an approximation of the solution of 
\eqref{vfbp2d-nondimc-eq-for-c}--\eqref{vfbp2d-nondimc-bc-for-c} for large $R$:
\begin{equation}\label{approx_C_R}
c(r) \approx \bar{\bar{c}}(R-r). 
\end{equation}
Here, $\bar{\bar{c}}(\eta)$ is independent of $R$. 
To examine
the inner solution $\bar{\bar{c}}(\eta)$,
we replace $\lambda(\bar{\bar{c}})$ by $\bar{\bar{c}}$. Then the exact solution for \eqref{bar_bar_C} is $\bar{\bar{c}}(\eta)=c_0e^{-\eta}$.
The expression \eqref{approx_C_R} gives the exact large-time (equivalently, large $R$) asymptotic for $\partial_rc(R,t)$: 
\begin{equation}\label{exact_asymptotics}
\lim\limits_{R\to\infty}\partial_rc(R,t)=-\bar{\bar{c}}'(0).
\end{equation}
The formula \eqref{exact_asymptotics} with \eqref{vfbp2d-nondimc-eq-for-R} implies that the bacterial colony grows linearly for large $R$, and the linear growth rate is $-\bar{\bar{c}}'(0)$.
We note that the numerical results are in the agreement with the approximation formula \eqref{approx_C_R}, 
and the growth rate $R'(t)$ is close to $-\bar{\bar{c}}'(0+)$; cf.\ Figure~\ref{fig:1D}~(c). 


\subsection{Profile of the nutrient concentration}
\label{ss:Monotonicity}

 In this subsection, we consider 
 the following system which is the time-independent version of 
 the system of equations \eqref{vfbp2d-nondimc-eq-for-c} and \eqref{vfbp2d-nondimc-bc-for-c}: 
\begin{equation}\label{eq:radial-simple}
\left\{
\begin{array}{l}
(rc'(r))' =  r \lambda (c(r)) \qquad \mbox{for } 0 < r < R,\\
c'(0) = 0 \qquad \mbox{and} \qquad 
c(R) =c_0.
\end{array}
\right.
\end{equation}
Here, $R>0$ is a given fixed number and $c_0 > 0$ is the given boundary value of the nutrient concentration. 
We shall consider a nonnegative solution $c = c(r)$ to \eqref{eq:radial-simple}. 

\begin{definition}
    \label{d:radialsol}
    A function $c = c(r)$ $(0 \le r \le R)$ is a solution to the boundary-value problem \eqref{eq:radial-simple},
    if $c \in C^2((0, R))\cap C^1([0, R])$ and $c$ satisfies \eqref{eq:radial-simple}.
\end{definition} 



\begin{theorem}\label{thm-mono}
Given $R> 0$ and $c_0 > 0$, there exists a unique  nonnegative solution to the boundary-value problem 
\eqref{eq:radial-simple}. Moreover, $c \in C^\infty((0, R]) \cap C^2([0, R])$, 
and  $0 < c(r) < c_0$ and $c'(r) > 0$
 for all $r \in (0, R).$
\end{theorem}
\begin{proof} 
We denote 
\begin{align*}
&B_r = \left\{ x = (x_1, x_2) \in \R^2: |x| = \sqrt{x_1^2 + x_2^2} < r\right\} \qquad \mbox{if } r > 0; \\
&\mathcal A_{R} = \{ u \in H^1(B_R): u = c_0 \mbox{ on }  \partial B_R\}; \\
&
\mathcal E_R[u] = \int_{B_R} \left[ \frac12 |\nabla u |^2 + \Lambda( u) \right] dx
\qquad \forall u \in H^1(B_R). 
\end{align*}
By the direct method, we infer the existence of a unique nonnegative minimizer $u $ of $\mathcal E_R$ over 
$\mathcal A_R;$  cf.\ the proof of Theorem~\ref{t:3Dminimizer}. The minimizer satisfies 
\begin{equation}
\label{uuu}
\Delta u = \lambda (u) \quad  \mbox{in } B_R 
\qquad \mbox{and} \qquad 
u = c_0  \quad \mbox{on } \partial B_R. 
\end{equation}
The elliptic regularity theory \cite{GilbargTrudinger98} implies that 
$u \in C^\infty (\overline{B}_R)$. Since  $\Delta u = a(u) u$, where
$a(u) := 1/(1+u)$, 
it follows from the Strong Maximum Principle 
that $0 < u < c_0$ in $\overline{B}_R$,

Define $c: [0, R] \to \R$ by 
\[
c(r) = \left\{ 
\begin{aligned}
    & u(0)& &  \quad \mbox{if } r = 0, \\
    & \frac{1}{2\pi r} \int_{\partial B_r} u(x) d\ell(x) & & \quad \mbox{if }  0 < r \le R. 
\end{aligned}
\right.
\]
Clearly, $c\in C^2((0, R])\cap C([0, R])$, $c(R) = c_0$, and $0 < c < c_0$ on $[0, R)$. 
 Let $0 <r < R.$ We have 
\begin{align*}
\frac{d}{dr} \int_{\partial B_r} u(x) \, d\ell (x) &= \frac{d}{dr}\int_0^{2 \pi} u(r \cos \theta, r \sin \theta )r\, d\theta
\\
&
=\int_0^{2 \pi} \nabla u(r \cos \theta, r \sin \theta) \cdot (\cos \theta, \sin \theta)\, r\, d\theta+ 
\int_0^{2 \pi} u(r \cos \theta, r \sin \theta )\, d\theta
\\
& =\int_{\partial B_r} \nabla u(x) \cdot \nu (x) \, d \ell (x) + \frac{1}{r} \int_{\partial B_r}
u(x) \, d\ell (x),
\end{align*}
where $\nu(x)$ is the unit normal at $x \in \partial B_r$ exterior to $B_r.$ This leads to 
\begin{align}
\label{cpr}
c'(r) &= \frac{1}{2 \pi r} \int_{\partial B_r} \nabla u(x) \cdot \nu (x)  \, d \ell (x)
\nonumber \\
&= \frac{1}{2 \pi r} \int_{\partial B_r} \left[ \nabla u(x) - \nabla u(0) \right] \cdot \nu (x) 
\, d \ell (x)\nonumber \\
&
\to 0 \quad \mbox{as } r \to 0^+. 
\end{align}
Thus, 
\begin{equation} 
\label{cp0+}
c'(0^+) = \lim_{r \to 0^+} c'(r)= 0
\end{equation}
exists and $c \in C^1([0, R])$. By Taylor's expansion and the fact that 
$x = |x| \nu (x)$ on $\partial B_r$, we get
\begin{align*}
\frac{c'(r)}{r} &= \frac{1}{2 \pi r^2} 
\int_{\partial B_r} \left[ D^2u(0) x \cdot \nu(x) + o(|x|)^2 \right] d\ell (x)
\\
& = o(r) + \int_0^{2 \pi} D^2 u(0) \nu_0(\theta) \cdot \nu_0(\theta) \, d \theta\qquad 
\mbox{as } r \to 0^+,
\end{align*}
where $D^2u$ is the Hessian of $u$ and $\nu_0(\theta) = (\cos \theta, \sin \theta)^T$
with $T$ denoting the transpose. This and \eqref{cp0+} imply that 
\begin{equation}
    \label{cpp0}
 c''(0^+) =    \lim_{r \to 0^+} \frac{c'(r)}{r} = \int_0^{2 \pi} D^2 u(0) \nu_0(\theta) \cdot \nu_0(\theta) \, d \theta. 
\end{equation}
Hence $c \in C^2([0, R]).$

Now,  by integration using polar coordinates, the definition of $c(r)$,  \eqref{cpr}, 
and the Cauchy--Schwarz inequality, we get
\begin{align*}
    \int_{B_R} |\nabla_x ( c(|x|) )|^2 dx 
    = \int_0^R \frac{1}{2 \pi r} \left[ \int_{\partial B_r} \nabla u(x) \cdot \nu (x) \, d\ell (x)
    \right]^2 dr
 \le \int_{B_R} |\nabla u(x)|^2 dx. 
\end{align*}
Moreover, the convexity of $\Lambda$ and Jensen's inequality lead to
\begin{align*}
\int_{B_R} \Lambda (c(|x|)) dx  
=  \int_0^R 2 \pi  \Lambda \left( \frac{1}{2 \pi r} \int_{\partial B_r } u(x) \, d \ell (x) \right) rdr
\le \int_{B_R} \Lambda (u(x))\, dx.
\end{align*}
Therefore, treating $c = c(r) = c(|x|) $ as a function of $x \in B_R$, we have 
$c \in \mathcal A_R$ and $\mathcal E_R[c] \le \mathcal E_R [u].$ The uniqueness of the minimizer
then leads to $u(x) = c(|x|)$ for all $x \in \overline{B_R}.$ Thus, $c \in C^\infty((0, R])$.

Assume $\hat{c} \in C^2((0, R])\cap C^1([0, R]) $ is also a  nonnegative solution to \eqref{eq:radial-simple}. Let $h = c - \hat{c}.$
Then, 
\[
\left\{
\begin{aligned} 
& (r h'(r))' = r \alpha (r) h(r) \qquad \mbox{for } 0 < r < R, \\
& 
h'(0) = 0 \qquad \mbox{and} \qquad h(R) = 0, 
\end{aligned}
\right.
\]
where $\alpha(r) = 1/\left[(1+c)(1+\hat{c})\right]> 0$ on $[0, R]$. Multiplying the equation by $h(r)$, integrating
the resulting equation with respect to $r$ in interval $(\varepsilon,R)$ for $0 < \varepsilon < R$, and using 
integration by parts, we get 
\[
Rh'(R)h(R)-\varepsilon h'(\varepsilon)h(\varepsilon)=\int_\varepsilon^{R}r (h'(r))^2\,\text{d}r+\int_\varepsilon^{R}r \alpha (r) (h(r))^2\,\text{d}r
\]
Using the boundary conditions and sending $\varepsilon \to 0$, we obtain that $h(r) = 0$ for all $r \in [0, R],$ leading
to $\hat{c} = c$ on $[0, R]$.

We finally show that $c'(r) > 0$ for all $r \in (0, R).$
First, we have by \eqref{eq:radial-simple} that 
$(rc'(r))'\ge 0$ for any $r \in (0, R).$ Thus, $m(r):= rc'(r)$ is an increasing function
on $(0, R)$.
Hence, $m(r) \ge m(0) =0$ 
for any $r \in (0, R)$. Thus, $c'(r) \ge 0$ for any
$r \in (0, R).$
Assume that there exists $r_0 \in (0, R)$ such that $c'(r_0)=0$. Integrating both sides of 
$(r c'(r))'= r \lambda (c (r))$ over $[0, r_0]$ and using the fact that $c'(0^+) = 0$, we obtain 
\[
0 = r_0 c'(r_0)=\int_0^{r_0}\zeta \lambda (c(\zeta))\,\text{d}\zeta. 
\]
Since $c \ge 0$ on $[0, R]$, this leads to $c(r)= 0$ for all $ r\in [0, r_0]$. 
Now $u(x) = c(|x|)> 0$ $(x \in B_R)$ satisfies \eqref{uuu}. Hence $\Delta u = \beta u$ in $B_R$ with 
$\beta  = 1/(1  + u) > 0$ in $B_R.$ Since $u = 0$ if $|x| \le r_0$, the Strong Maximum Principle 
implies that $u$ is a constant in $B_R$m contradicting the fact that $u = c_0 > 0$ on $\partial B_R.$
\end{proof}

\begin{proposition}
\label{p:cDisk}
Assume $c_0>0$ and $R > 0.$ Let $c_R = c_R(r)$ $(0 \le r \le R)$ denote the solution of \eqref{eq:radial-simple}
with $c_R$ replacing $c$. Then 
the following hold true: 
\begin{itemize}
\item[{{\rm (i)}}] 
If $0 <  R_1 < R_2$, then
 $c_{R_2}(r)<c_{R_1}(r)$ for all $r \in [0, R_1]$, 
 $c'_{R_2}(r) < c'_{R_1}(r)$ for all $r \in (0, R_1]$, 
 and $R_1 c_{R_1}'(R_1) < R_2 c_{R_2}'(R_2)$; 

\item[{\rm (ii)}] 
If $R_0 > 0$, then
$\lim_{R\to\infty}\max_{0\leq r \leq R_0} c_R(r)=\lim\limits_{R\to\infty} c_R(R_0)=0.$ 
\end{itemize}
\end{proposition}


\begin{proof} Part (i). 
Let $\tilde{c}(r) = c_{R_1} (r) - c_{R_2}(r)$ $(0 \le r \le R_1)$. We need to show that 
$\tilde{c}(r) > 0$ and $\tilde{c}'(r) > 0$ for any $r \in [0, R_1].$
Note that 
\[
c_{R_1}(R_1)=c_0 = c_{R_2} (R_2) > c_{R_2}(R_1).
\]
We then have 
\begin{equation}\label{c-tilda}
\left\{
\begin{array}{l}
(r\tilde{c}'(r))'= r \tilde{\lambda}(r)\tilde{c}(r) \qquad\forall r \in (0, R_1),\\
\tilde{c}'(0)=0 \qquad \mbox{and} \qquad 
\tilde{c}(R_1)=c_0 - c_{R_2}(R_1) >0, 
\end{array}
\right.
\end{equation}
where 
\begin{equation*}
0 < \tilde{\lambda}(r):=\dfrac{1}{(1+c_{R_1}(r))(1+c_{R_2}(r))}<1\qquad \forall r \in [0, R_1].
\end{equation*}

First, let us show $\tilde{c}(r)\geq 0$ for all $r\in [0,R_1]$. 
Let $r_\ast \in [0, R]$ be such that $\tilde{c}(r_\ast) = \min_{0 \le r \le R} \tilde{c}(r).$ 
It suffices to show that $\tilde{c}(r_\ast) \ge 0$. Suppose $\tilde{c}(r_\ast) < 0.$ 
Note that $\tilde{c}(R_1)>0$. Thus, $0 \le r_\ast < R_1.$
We have $\tilde{c}'(r_\ast) = 0 $ regardless $0 < r_\ast < R_1$ or $r_\ast = 0.$
Suppose $0 < r_\ast < R_1$. Then, 
the equation in \eqref{c-tilda} at $r=r_*$ reduces to $\tilde{c}''(r_*)=\tilde{\lambda}(r_*)\tilde{c}(r_*)< 0$, 
since $\tilde{\lambda}(r_*)>0$.  Thus, 
$\tilde{c}$ attains a strict local maximum at~$r_*$, contradicting \ that $r_\ast$ is a minimum point of $\tilde{c}$
on $[0, R_1].$ 
Suppose $r_*=0$. Note that 
\[
\dfrac{\tilde{c}'(r)}{r}=\dfrac{\tilde{c}'(r)-\tilde{c}'(0^+)}{r-0}\to \tilde{c}''(0^+) \qquad \text{ as }r\to 0.
\]
Dividing the differential equation in \eqref{c-tilda} by $r>0$ and sending $r \to 0$, we get 
$2\tilde{c}''(0^+)=\tilde{\lambda}(0)\tilde{c}(0)< 0$. Hence $\tilde{c}'(r)$ decreases 
for $r \in [0, \delta]$ for some small $\delta > 0$, and thus $\tilde{c}(\delta) < \tilde{c}(0).$ This contradicts the fact that 
$r_\ast = 0$ is a minimum of $\tilde{c}$ on $[0, R_1].$ 

Next, we show that $\tilde{c}(r)>0$ for any $r \in [0, R_1]$. Since $\tilde{c} \ge 0$ on 
$[0, R]$, we have by the equation in \eqref{c-tilda} that 
$(r \tilde{c}'(r))' \ge 0$ on $(0, R)$. Integrating both sides and using the condition $\tilde{c}'(0) = 0$, we
then have $r \tilde{c}'(r) \ge 0$ and hence $\tilde{c}'(r) \ge 0$ for any $r \in (0, R). $
Thus, $\tilde{c}$ increases on $[0, R]$. It suffices to show now $\tilde{c}(0) > 0.$ 
Suppose $\tilde{c}(0) = 0$. 
Let $0<r_0<\min\{R,2\}$ and integrate the equation in \eqref{c-tilda} twice to get
\[
\tilde{c}(r_0)=\int_0^{r_0}\dfrac{1}{r}\int_0^{r}s\tilde{\lambda}(s)\tilde{c}(s)\,\text{d}s\text{d}r.
\]
If $\max_{0 \le s \le r_0} \tilde{c}(s) > 0$, then the fact that
$\tilde{\lambda}(s)<1$ and  that $\tilde{c}'\ge 0$ on $(0, R_1)$ imply
\[
0 \le \tilde{c}(r_0) \leq \dfrac{r_0^2}{4}\max_{0\leq s \leq r_0} \tilde{c}(s) < \max_{0\leq s \leq r_0}\tilde{c}(s)
= \tilde{c}(r_0), 
\]
a contradiction. If $\max_{0 \le s \le r_0} \tilde{c}(s) = 0$, then both $\tilde{c}$ and the zero function
are solutions to the initial-value problem of the equation in \eqref{c-tilda} 
for $r \in (r_0, R_1)$ with the initial values $\tilde{c}(r_0) = \tilde{c}'(r_0) = 0.$
The uniqueness of such a solution implies that $\tilde{c} = 0$ on $(0, R_1)$, contradicting
$\tilde{c} (R_1) > 0$.  Hence, $\tilde{c}(0) > 0$, and $\tilde{c}(r) > 0$ for any $r \in [0, R].$

Integrating both sides of the equation in \eqref{c-tilda}, we get
\[
r\tilde{c}'(r) = \int_0^r s \lambda (\tilde{c}(s))\, ds > 0 \qquad\forall r \in (0, R_1].
\]
Hence, $\tilde{c}'(r) > 0$ if $r \in (0, R_1].$


For any $ R > 0$, let $c_R\in C^2([0, R])$ be the unique positive solution of 
\eqref{eq:radial-simple} with $c$ replaced by $c_R.$ Define $\hat{c}_R(\hat{r}) = c_R(r)$ 
with $\hat{r} = r/R \in [0, 1]$ and $r \in [0, R]$. Then, $\hat{c}_R \in C^2([0, 1])$ is the 
unique solution to 
\[
(\hat{r} \hat{c}_{R}'(\hat{r}))' = R^2 \hat{r}\lambda (\hat{c}_R(\hat{r})) 
\quad \forall \hat{r} \in (0, 1) \qquad \mbox{and} \qquad 
\hat{c}_R'(0)= 0, \qquad \hat{c}_R(1) = c_0.
\]
Let $\hat{u}_R (x) = \hat{c}_R(|x|)$ for $x = (x_1, x_2) \in \overline{B}$, where 
$B = B_1(0)$ is the the unit disk centered
at the origin. Then, similar to the proof of Theorem~\ref{thm-mono}, we infer
that $\hat{u}_R \in C^\infty(\overline{B})$ satisfies $\hat{u}_R> 0$ in $B$ and 
\[
\Delta \hat{u}_R = R^2 \lambda (\hat{u}_R) \quad \mbox{in } B 
\qquad \mbox{and} \qquad \hat{u}_R = c_0 \quad \mbox{on } \partial B. 
\]
With $R = R_1$ and $R = R_2$ for $0 < R_1 < R_2$, we get the solutions $c_{R_i}$, 
$\hat{c}_{R_i}$, and $\hat{u}_{R_i}$ $(i = 1, 2)$ to their respective boundary-value
problems. Let $\hat{u}= \hat{u}_{R_1} - \hat{u}_{R_2} \in C^\infty(\overline{B}).$ Then, 
\begin{equation}
\label{hatubf}
\Delta \hat{u}(x)  - b(x) \hat{u}(x) = f(x) \quad \forall x \in B
\qquad \mbox{and} \qquad \hat{u} (x) = 0 \quad \forall x \in\partial B,
\end{equation}
where 
\[
b(x) = \dfrac{R_1^2}{(1+\hat{u}_{R_1})(1+\hat{u}_{R_2})}\quad \text{ and }\quad 
f(x)=\dfrac{(R_1^2-R_2^2)\hat{u}_{R_2}}{1+\hat{u}_{R_2}}. 
\]
Both $b$ and $f$ are smooth and bounded with $b> 0$ and $f < 0$ in $\overline{B}.$
The Maximal Principle \cite{GilbargTrudinger98,EvansBook2010} implies that $\hat{u} \ge 0$ in $B$ and
the Strong Maximum Principle further implies that
$\hat{u}>0$ in $B$. Consequently, by the boundary condition 
$\hat{u} = 0$ on $\partial B$ and Hopf's lemma \cite{GilbargTrudinger98}, 
we get that $\partial_\nu\hat{u}<0$ on $\partial B$.  

 Hence, $\hat{c}_{R_1}'(1) < \hat{c}_{R_2}'(1)$, which is the same 
 as $R_1 c_{R_1}'(R_1) < R_2 c_{R_2}'(R_2).$

Part (ii). 
By Theorem~\ref{thm-mono}, we need only to show that $\lim_{R\to \infty} c_{R}(R_0) = 0. $ Suppose this were not true. Then, there 
exist $c_\ast > 0$ and $R_k > R_0$ $(k = 1, 2, \dots)$  such that $R_k\nearrow \infty$ and 
$c_{R_k}(R_0)\ge c_{*}$ for all $k \ge 1.$
It then follows from Theorem~\ref{thm-mono} that $c_{R_k}(r)
\ge c_{*}$ for all $r\in [R_0,R_k]$ and all $k \ge 1$. By the equation for $c_{R} = c_{R}(r)$, we get
\[
(rc_{R_k}'(r))' = r \lambda(c_{R_k}(r)) \geq  r \lambda(c_{*}) \qquad \forall r \in [R_0, R_k] \ \forall k \ge 1.
\]
Integrating this inequality from $R_0$ to $r \in (R_0, R_k)$ with $k \ge 1$ leads to 
\begin{equation*}
rc_{R_k}'(r)\geq R_0c_{R_k}'(R_0)+ \frac{\lambda(c_{*})}{2} \left(r^2 - R_0^2\right). 
\end{equation*}
Noting that $c_{R_k}'(R_0)\geq 0$ and 
dividing by $r$, we obtain  
\begin{equation*}
c_{R_k}'(r)\geq \dfrac{\lambda(c_*)}{2}\left(r-\dfrac{R_0^2}{r}\right) \qquad \forall r \in [R_0, R_k] \ \forall k \ge 1.
\end{equation*}
Integrating this inequality from $R_0$ to $R_k$ with respect to $r$ and using the fact that  $c_{R_k}(R_k)=c_0$, we obtain
\begin{equation*}
c_0 \geq c_{R_k}(R_0) + \dfrac{ \lambda(c_*)}{2}\left[\dfrac{R_k^2}{2}-\dfrac{R_0^2}{2}-R_0^2\ln\left(\dfrac{R_k}{R_0}\right)\right].
\end{equation*}
Since each $c_{R_k} (R_0) \ge c_\ast > 0$ and $\lambda(c_\ast) > 0$,  we arrive at a contradiction
by sending $k$ to infinity. 
\end{proof}






We recall that the modified Bessel function of first kind $I_k = I_k (x) $ $(x > 0)$ of order $k$ is a bounded solution of 
the equation 
$(xI_k')'=[(x^2+k^2)/{x}]I_k$ \cite[Section 5.7]{silverman1972special}. It has the integral representation 
\cite[10.23.2]{DLMF} 
\[
I_k(x)=\dfrac{1}{\pi}\int_0^{\pi}e^{x\cos(\theta)}\cos(k\theta)\,\text{d}\theta > 0 \qquad \forall x > 0. 
\]
One verifies that $I_0'(x) = I_1(x)$ and by the Cauchy--Schwarz inequality
$[I_1(x)]^2 < I_0(x) I_1'(x) $ for all $x > 0$, leading to $(I_1(x)/I_0(x))'>0$. Hence, 
$I_1(x)/I_0(x)$ increases strictly for $x > 0$.


\begin{theorem}
\label{t:radialexp}
    Let $c_0>0$ and $R>0$. Let $c_{R} = c_{R}(r)$ $(0 \le r \le R)$ 
    be the solution of \eqref{eq:radial-simple} with $c_R$ replacing $c$. Then 
    \begin{equation}
    \label{cRp}
    C_1 \leq c_{R}'(R)\leq C_2 \qquad \forall R > 1, 
    \end{equation}
    where 
    \begin{equation}
    \label{C1C2}
C_1 =  \dfrac{ c_0\sqrt{\tilde{\alpha}}I_1(\sqrt{\tilde{\alpha}})}{2I_0(\sqrt{\tilde{\alpha}})}>0
\quad \mbox{with}
\quad \tilde{\alpha} = \frac{1}{1 + c_0}
\qquad 
\mbox{and} \qquad  C_2= c_0 + 2.
\end{equation}
\end{theorem}

\begin{proof}
Let us denote $c (r) = c_R(r)$ $(0 \le r \le R)$ for simplicity. 
Define 
\begin{equation}\label{radial-energy-functional}
\hat{\mathcal{E}}_{R}[u]=2\pi
\int_{0}^{R}\left[ \dfrac{1}{2}(u'(r))^2 +  \Lambda(u(r))\right]r\,\text{d}r \qquad 
\forall u \in H^1((0, R)).
\end{equation}
We shall prove 
\begin{eqnarray}
\label{I1}
&\dfrac{\hat{\mathcal{E}}_{R} [c]}{2 \pi c_0 R}\leq c'(R) \leq  \dfrac{ \hat{\mathcal E}_{R} [c]}{\pi c_0 R},& \\
\label{I2}
&\dfrac{\pi c_0^2\sqrt{\tilde{\alpha}}I_1(\sqrt{\tilde{\alpha}})}{I_0(\sqrt{\tilde{\alpha}})}R
\leq \hat{\mathcal{E}}_R  [c] \leq \pi c_0 (c_0 + 2) R, &
\end{eqnarray}
which will imply \eqref{cRp} and complete the proof.

{\it Proof of \eqref{I1}.} Multiply the first equation in \eqref{eq:radial-simple} by $c(r)$ and integrate over $r$ to get
\begin{equation*}
\int_0^R (rc'(r))'c(r) \,\text{d}r = \int_0^R r\lambda(c(r)) c(r)\,\text{d}r.
\end{equation*}
By integration by parts for the left-hand side and the boundary condition $c(R) = c_0$, we obtain
\begin{equation}
\label{c0Rc}
c_0Rc'(R)=\int_0^R \left[ (c'(r))^2+\lambda(c(r))c(r) \right] r\, \text{d}r.
\end{equation}
Note that the function $f(x):=x\lambda(x)/\Lambda(x)$ is strictly positive, 
$f(x) \to 2$ as $x \to 0^+$, and $f(x) \to 1 $ as $x \to \infty$. Moreover, 
$f'(x) = -\ln (x+1) [\lambda(x)/\Lambda(x)]^2 < 0$ on $(0, \infty)$, and hence, 
$f(x) $ is decreasing on $(0, \infty)$. Thus,
\begin{equation} 
\label{LLambda}
\Lambda (x) \le x \lambda (x) \le 2 \Lambda (x) \qquad \forall  x \in [0, \infty). 
\end{equation}
This, together with \eqref{c0Rc} with $x = c(r)$ and \eqref{radial-energy-functional}, implies \eqref{I1}. 

{\it Proof of \eqref{I2}.}
 Let $\hat{c}(r)=c_0e^{r-R}$ $(0 \le r \le R)$. Then, 
 \begin{align*}
\hat{\mathcal{E}}_R [\hat{c}] &= 
2\pi \int_0^R \left\{\dfrac{c_0^2}{2}
e^{2(r-R)}
+ \left[ c_0e^{r-R}
-\ln(1+c_0e^{r-R})
\right]\right\}\,r\text{d}r \\
& \le 2 \pi \int_0^R \left[ \frac12 c_0^2 e^{2 (r - R)} + c_0 e^{r - R} \right] r \, dr \\
&\leq  \pi c_0(c_0+2) \int_0^R e^{r-R}\,r\text{d}r \\
&= \pi c_0(c_0+2)\left( R+e^{-R}-1\right)\\
& <  \pi c_0(c_0+2) R. 
\end{align*}
 Let $\phi (r) = \hat{c}(r) - c(r)$ $(0 \le r \le R)$. Note that $\phi(R) = 0$. By the convexity
 of $F(s)=s^2/2$ and $F(s)=\Lambda(s)$ for $s > 0$, we have 
 $F(\hat{s})-F(s)\geq F'(s)(\hat{s}-s)$ if $s, \hat{s} > 0$. 
 Thus, it follows from integration by parts and the differential equation for $c$ in 
 \eqref{eq:radial-simple} that 
 \begin{align*}
     \frac{1}{2 \pi} \left( \hat{\mathcal E}_R[\hat{c}]  - \hat{\mathcal E}_R [c] \right)
     &\ge \int_0^R\left[ c'(r) \phi'(r) + \lambda (c(r)) \phi(r) \right] r dr 
    \\
    &=\int_0^R \left[ - (r c'(r))' + r\lambda(c(r)) \right] \phi(r) \, dr 
    \\
    & = 0. 
 \end{align*}
These two inequalities imply the upper bound in \eqref{I2}.


Note from \eqref{LLambda} that 
\[
\Lambda(x)\geq \frac{\lambda(x)x}{2} \ge \dfrac{1}{2(1+c_0)}x^2 \qquad \mbox{if } 0 \le x \le c_0.
\]
Thus, 
\begin{equation}
\label{lowB1}
\hat{\mathcal{E}}_R [c] \ge 2 \pi\int_0^R \left[ \frac12 \left( c'(r) \right)^2 + \dfrac{\tilde{\alpha}}{2} 
(c(r))^2 \right] r\, \text{d}r = \tilde{\mathcal E}_{R}[c], 
\end{equation}
where $\tilde{\alpha} = 1/(1 + c_0)$  and 
\[
\tilde{\mathcal E}_{R} [u ] = 2 \pi \int_0^R \left[\frac12 (u'(r))^2 + \frac{\tilde{\alpha}}{2} (u(r))^2 \right] r \, dr
\qquad \forall u \in H^1((0, R)).
\]
By the same argument used in proving Theorem~\ref{thm-mono}, we infer the existence of 
$\tilde{c} \in C^2([0, R])$ that solves the boundary-value problem
\begin{equation}\label{c-hat}
\left\{
\begin{array}{l}
(r\tilde{c}'(r))'=\tilde{\alpha} r \tilde{c}(r) \qquad \mbox{if } 0 < r < R,\\
\tilde{c}'(0) = 0\qquad \mbox{and} \qquad 
\tilde{c}(R) =c_0.
\end{array}
\right.
\end{equation}
Similar to the proof of $\hat{\mathcal E}_R[\hat{c}] \ge \hat{\mathcal E}_R[c]$, we have 
$\tilde{\mathcal E}_R[c] \ge \tilde{\mathcal E}_R[\tilde{c}].$ This and \eqref{lowB1} imply that
\begin{align}
\hat{\mathcal E}_R[c] &\ge \tilde{\mathcal E}_{R}[\tilde{c}]\nonumber \\
& = 
2\pi \int_0^R \left[\dfrac{1}{2}(\tilde{c}'(r))^2 + \dfrac{\tilde{\alpha}}{2} (\tilde{c}(r))^2 \right]r\, 
\text{d}r \nonumber \\
&=-\pi \int_0^R \left[(r\tilde{c}'(r))' - \tilde{\alpha} r \tilde{c}(r)\right]\tilde{c}(r)\,\text{d}r
+\pi c_0 R\tilde{c}'(R) \nonumber\\
&=\pi c_0R\tilde{c}'(R). \label{eqn:repr_poisson}
\end{align}

The solution $\tilde{c}$ to the boundary-value problem \eqref{c-hat} is given by 
\begin{equation*}
\tilde{c}(r)=c_0\dfrac{I_0(\sqrt{\tilde{\alpha}}r)}{I_0(\sqrt{\tilde{\alpha}}R)} \qquad \forall r \in [0, R].
\end{equation*}
By the fact that $I_0'(x) = I_1(x)$ and $I_1(x)/I_0(x)$ is a strictly increasing and smooth function for $x>0$ 
(see the discussion before Theorem~\ref{t:radialexp}), we have 
\begin{equation}
\label{tcpR}
\tilde{c}'(R)=c_0\sqrt{\tilde{\alpha}}\dfrac{I_1(\sqrt{\tilde{\alpha}} R)}{I_0(\sqrt{\tilde{\alpha}} R)}
\ge c_0\sqrt{\tilde{\alpha}}\frac{I_1(\sqrt{\tilde{\alpha}})}{I_0(\sqrt{\tilde{\alpha}})} \qquad \forall R > 1.
\end{equation}
Now, \eqref{eqn:repr_poisson} and \eqref{tcpR} imply the lower bound in \eqref{I2}. 
\end{proof}

\begin{corollary}
\label{c:boundarylayer}
Let $c_0>0$ and $R>1$. Let $c_R = c_R(r)$ $(0 \le r \le R)$ be the solution of \eqref{eq:radial-simple}
with $c_R$ replacing $c.$
Given $\alpha \in (0, 1)$ and assume $R_\alpha \in (0, R)$ satisfy  $c_R(R_{\alpha })=\alpha 
c_0$ and $R_\alpha > 1.$ Then, 
\begin{equation}
\label{alphac}
\frac{(1-\alpha)c_0}{C_2} \le R - R_\alpha \le \frac{(1-\alpha) c_0}{\alpha C_1},
\end{equation}
where $C_1 > 0$ and $C_2 > 0$ are the same constants in Theorem~\ref{t:radialexp}; cf.\ \eqref{C1C2}.
\end{corollary}

\begin{proof}
First, we note Theorem~\ref{t:radialexp} is applicable for $\hat{c}_R$, truncated on $(0,r)$ with $\hat{c}_R(r)$ instead of $c_0$. In particular, for any $r \in (R_\alpha, R)$, we have by Theorem~\ref{t:radialexp}, with $c_R(r)$ and $c_R'(r)$ 
replacing $c_0$ and $c_R'(R)$, respectively, that 
\begin{equation*}
C_1(r)  \le c_R'(r) \le C_2(r), 
\end{equation*}
where by \eqref{C1C2} and the monotonicity of $I_1(x)/I_0(x)$ $(x > 0)$ that 
\begin{align*}
    & C_1(r) = \frac{c_R(r)}{2} \sqrt{\frac{1}{1 + c_R(r)}} 
    \frac{I_1\left(\sqrt{1/(1+c_R(r))}\right)}{I_0\left(\sqrt{1/(1+c_R(r))}\right)}
    \ge \frac{c_R (R_\alpha)}{c_0} C_1 = \alpha C_1, 
    \\
    &C_2(r) = c_R (r) + 2 \le  C_2. 
\end{align*}
These  and the Mean-Value Theorem imply that  for some $r \in (R_\alpha, R)$
\[
\alpha C_1 \le C_1(r) \le \frac{c_0 - c_R (R_\alpha)}{ R - R_\alpha} = \frac{(1-\alpha) c_0}{ R - R_\alpha} \le C_2(r)
\le c_0 + 2 =  C_2,  
\]
leading to \eqref{alphac}. 
\end{proof}

\begin{rmk}
    Suppose $c_0 > 1$ which corresponds to the constant concentration of glucose being above the Monod constant $K$
    and $\alpha = 1/c_0$. Then, our result indicates that there is a peripheral ring-shaped region of monolayer
    of cells in the colony that can grow all the time. The width of this region is a constant for all the radius $R.$
    This agrees well with agent-based simulations and experiment \cite{Warren_eLife2019,Kannan_NatCommun2025}.
\end{rmk}

\subsection{Dynamics of radial expansion}
\label{ss:RadialDyn}

We first study the well-posedness of the moving-boundary problem \eqref{vfbp2d-nondim}. 

\begin{theorem}[Well-posedness of the free-boundary problem \eqref{vfbp2d-nondim-compact}]
\label{thm:hfbp-wellposed}
Let $c_0>0$ and  $R_0>0$. Then the system
\eqref{vfbp2d-nondim-compact} admits a unique global solution
\(
(c(r,t),R(t)),
\)
with
$
R\in C^2([0,\infty))$ and 
$c(\cdot, t ) \in C^2([0, R(t)])$ for all $t \geq 0.$ 
\end{theorem}

To prove this theorem, we need the following lemma:
\begin{lemma}\label{lemma:aux-existence}
Let $F,a\in C^1([0,1])$ be such that $a\geq a_0>0$ on $[0, 1]$ for some constant $a_0>0$. Then there exists a 
unique $u \in C^2([0, 1])$ that solves
\begin{equation}\label{s4-sys-lemma}
(\hat{r}u'(\hat{r}))'=\hat{r}\left[a(\hat{r})u(\hat{r})
+F(\hat{r})\right]\quad \forall \hat{r} \in (0,1),\qquad u'(0)=0, \qquad 
u(1)=0. 
\end{equation}
Moreover, there exists $C=C(a) >0$, independent of $F$, such that
\begin{equation}\label{lemma-shauder-bound}
\|u\|_{C^1([0,1])}\leq C\|F\|_{C([0,1])}.
\end{equation}
\end{lemma}

\begin{proof}
The existence, uniqueness, and regularity of solution to \eqref{s4-sys-lemma} can be obtained by using the same argument
used in the proof of Theorem~\ref{thm-mono}. Specifically, the solution $ u = u(r)$ is the radial average of 
the unique minimizer in $H_0^1(B)$ of the functional 
\[
J[v] = \int_B \left[ \frac{1}{2} |\nabla v|^2 + \frac{1}{2} a(|x|)  v^2 + F(|x|)v\right] dx \qquad \forall v \in H_0^1(B), 
\]
where $B$ is the unit disk of $\R^2$. 
The minimizer is in $H_0^1(B)\cap C^2(\overline{B})$. Hence $u \in C^2([0, 1]).$

We now prove \eqref{lemma-shauder-bound}. We may assume $F \ne 0$ for otherwise the minimizer of the above
functional $J$ is the zero function and hence $u = 0$ on $[0, 1]$. 
We first show  $\|u\|_{C([0,1])}\leq a_0^{-1}\|F\|_{C([0,1])}$.
Assume there exists $\hat{r}_0\in [0, 1]$ such that $u$ attains maximum at $\hat{r}_0$ 
on $[0, 1]$ and $u(\hat{r}_0)>a_0^{-1}\|F\|_{C([0,1])}>0$. Then, we have $\hat{r}_0 < 1$ as $u(1) = 0.$
If $\hat{r}_0\in(0,1)$, then $u'(\hat{r}_0)=0$ and $u''(\hat{r}_0)\leq 0$, 
so $\hat{r}^{-1}(\hat{r}u'(\hat{r}))'|_{\hat{r}=\hat{r}_0}\leq 0$. 
By the differential equation in \eqref{s4-sys-lemma}, we then 
have $a(\hat{r}_0)u(\hat{r}_0)+F(\hat{r}_0)\leq 0$, implying that $u(\hat{r}_0)\leq a_0^{-1} \|F\|_{C([0,1])}$. 
This is a contradiction.
If $\hat{r}_0=0$, then $u(\hat{r})>a_0^{-1}\|F\|_{C([0,1])}$ for $\hat{r}\in (0,\varepsilon)$ for some $\varepsilon>0$. Then the differential equation in \eqref{s4-sys-lemma} implies 
\begin{equation}\label{lemma:integral-formulations}
u'(\hat{r})=\int_0^{\hat{r}}\dfrac{s}{\hat{r}}\left(a(s) u(s)+ F(s)\right)\,\text{d}s>0 \qquad \mbox{if } 0< \hat{r}<\varepsilon,
\end{equation}
leading to that $u$ increases on $(0,\varepsilon)$. It contradicts the assumption that $u$ attains its maximum at $\hat{r}=\hat{r}_0=0$.
Thus, we proved that $u(\hat{r})\leq  a_0^{-1}\|F\|_{C([0,1])}$
for all $\hat{r} \in [0, 1]$. 
We obtain also $u(\hat{r})\geq - a_0^{-1}\|F\|_{C([0,1])}$ by observing that $-u$ solves \eqref{s4-sys-lemma} with $-F$ instead of $F$.

 We now show that $\|u'\|_{C([0,1])}\le C \| F\|_{C([0, 1])}$.
The integral equality \eqref{lemma:integral-formulations} applied for all $\hat{r}\in (0,1)$, together with the result just proved
above, implies that
\[
|u'(\hat{r})|\leq \int_0^{\hat{r}}\left(\|a\|_{C([0,1])}a_0^{-1}\|F\|_{C([0,1])}+\|F\|_{C([0,1])}\right)\,\text{d}s\leq (\|a\|_{C([0,1])}a_0^{-1}+1)\|F\|_{C([0,1])}.
\]
Setting $C:=a_0^{-1}+\|a\|_{C([0,1])}a_0^{-1}+1 $ completes the proof. 
\end{proof}

\begin{proof}[Proof of Theorem~\ref{thm:hfbp-wellposed}]
For each $R > 0$, let $c_R = c_R (r) \in C^\infty((0, R])\cap C^2([0, 1])$ be the unique positive solution to 
the boundary-value problem \eqref{eq:radial-simple} with $c_R$ replacing $c$; 
cf.\ Theorem~\ref{thm-mono}.
 Then, 
 \begin{equation}
\label{defineHR}
c_R'(R) = \frac{1}{R}  R c_R'(R ) = \frac{1}{R} \int_0^R ( r c_R'(r))'\, dr
 = \frac{1}{R} \int_0^R r \lambda (c_R (r))\, dr = :H(R). 
 \end{equation}
Hence, it suffices to show that the initial-value problem 
\[
R'(t) = H(R) \quad \forall t > 0 \qquad \mbox{and} \qquad R(0) = R_0 > 0
\]
has a unique solution $R \in C^1([0, \infty)).$ 

For any $R > 0$, we can write $H(R) = S(R) / R$, where 
\begin{equation*}
0 < S(R) := R c_R'(R) = \int_0^R r \lambda (c_R(r))\, dr  \le \int_0^R r \, dr = \frac12 R^2. 
\end{equation*}
Let $0 <  R_1 < R_2$. By Proposition~\ref{p:cDisk}, $c_{R_1}(r) > c_{R_2}(r)$ 
if $0 \le r \le R_1$ and $S(R_1) < S(R_2). $  Therefore, 
\begin{align*}
|S(R_1) - S(R_2)| &= S(R_2) - S(R_1)
\nonumber \\
& =\int_{R_1}^{R_2} r \lambda (c_{R_2}(r))\, dr + \int_{0}^{R_1} r \left[
\lambda(c_{R_2} (r)) - \lambda (c_{R_1} (r) )\right] dr 
\nonumber \\
& < \int_{R_1}^{R_2} r \, dr 
\nonumber \\
& = \frac12 (R_1 + R_2) | R_1 - R_2|. 
\end{align*}
It now follows that 
\begin{align*}
|H(R_1) - H(R_2)|  \le \left| \left( \frac{1}{R_1} - \frac{1}{R_2} \right) S(R_1) \right| 
+ \frac{1}{R_2}\left| S(R_1) - S(R_2)\right| 
< \frac{3}{2} | R_1 - R_2|. 
\end{align*}
Hence $H = H(R)$ is globally Lipschitz-continuous on $(0, \infty)$. Following the standard
theory of initial-value problems of ODE, we conclude that there exists a unique solution 
$R = R(t)$ on $(0, \infty).$

By Theorem~\ref{thm-mono}, $c(\cdot, t) \in C^2([0, R(t)])$ for each $ t \ge 0.$

Finally, we  show that $H \in C^1((0, \infty))$ which implies that 
$R\in C^{2}([0,\infty))$. 
Define $\hat{c}_R(\hat{r})=c_R(r)$ for $\hat{r}=r/R\in [0, 1].$ Then the function $\hat{c}_R\in C^2([0, 1])$ is the 
unique positive solution of the following boundary-value problem: 
\begin{equation}\label{s4:problem-for-c_hat}
\left\{
\begin{aligned}
&(\hat{r}\hat{c}'_R(\hat{r}))'=R^2 \hat{r}\lambda(\hat{c}_R(\hat{r})) \qquad \mbox{for } \hat{r} \in (0, 1), \\
& \hat{c}'_R(0)=0 \qquad \mbox{and} \qquad \hat{c}_R(1) = c_0. 
\end{aligned}
\right.
\end{equation}
Let $\varepsilon>0$ and $w_\varepsilon(\hat{r}):=\hat{c}_{R+\varepsilon}(\hat{r})-\hat{c}_R(\hat{r})$. 
We have
\begin{equation}
\label{problem-for-w_eps}
\left\{ 
\begin{aligned}
&(\hat{r}w'_\varepsilon(\hat{r}))'=R^2 \hat{r}a_1(\hat{r})w_\varepsilon+\varepsilon \hat{r} F_1(\hat{r}) \qquad 
\text{for } \hat{r} \in (0,1),\\
& w'_\varepsilon(0)=0 \qquad \mbox{and} \qquad w_\varepsilon(1) = 0, 
\end{aligned}
\right.
\end{equation}
where
\begin{align*}
&a_1(\hat{r})=\dfrac{1}{(1+\hat{c}_{R+\varepsilon}(\hat{r}))(1+\hat{c}_R(\hat{r}))},  
\\
&
F_1(\hat{r})=2R \lambda (\hat{c}_{R+\varepsilon}(\hat{r}))+\varepsilon \lambda(\hat{c}_{R+\varepsilon}(\hat{r})).
\end{align*}
Lemma~\ref{lemma:aux-existence} implies 
\begin{equation}
\label{s4:continuous dependence}
\|w_\varepsilon\|_{C^1([0,1])}\leq C\varepsilon. 
\end{equation}
Here and below, $C$ is a generic constant, independent of $\varepsilon$, 
that may change from line to line. 
As an immediate consequence, taking into account non-negativity of $\hat{c}_{R}$, we have that 
\begin{equation}
\label{bounds-on-lambda}
\|\lambda(\hat{c}_{R+\varepsilon})-\lambda(\hat{c}_R)\|_{C([0,1])}+\|\lambda'(\hat{c}_{R+\varepsilon})-\lambda'(\hat{c}_R)\|_{C([0,1])}\leq C\varepsilon.
\end{equation}

Introduce the boundary-value problem 
\begin{equation}
\label{problem-for-v_R}
\left\{
\begin{aligned}
& (\hat{r}v'_R(\hat{r}))'=R^2 \hat{r}\lambda'(\hat{c}_R(\hat{r}))v_R(\hat{r})+2R\hat{r}\lambda(\hat{c}_R(\hat{r})) \qquad 
\text{for} \ \hat{r} \in (0,1),\\
&v'_R(0)=0 \qquad  \mbox{and} \qquad v_R(1) = 0. 
\end{aligned}
\right.
\end{equation}
This problem is obtained by formally  differentiating \eqref{s4:problem-for-c_hat} with respect to $R$. 
By Lemma~\ref{lemma:aux-existence}, such $v_R$ exists and is unique. 
Using Lemma~\ref{lemma:aux-existence} and repeating arguments for the bound \eqref{s4:continuous dependence}, 
one can also show that 
\begin{equation}\label{last-straw}
\|v_{R+\varepsilon}-v_R\|_{C^1([0,1])}\leq C\varepsilon.
\end{equation} 
Hence,  $v_R$  depends on $R>0$ continuously. Next, define 
\[
u_\varepsilon (\hat{r}):= \dfrac{w_\varepsilon(\hat{r})}{\varepsilon}-v_R(\hat{r})=\dfrac{1}{\varepsilon}(\hat{c}_{R+\varepsilon}(\hat{r})-\hat{c}_R(\hat{r}))-v_R(\hat{r}).
\]
We verify that the function $u_\varepsilon\in C^2([0, 1])$ solves the following boundary-value problem:
\begin{equation}\label{problem-for-u-eps}
\left\{
\begin{aligned}
&(\hat{r}u_\varepsilon'(\hat{r}))'=R^2 \hat{r}\lambda'(\hat{c}_R(\hat{r}))u_\varepsilon(\hat{r})
+\varepsilon \hat{r} F_2(\hat{r}) \qquad 
\text{for } \hat{r} \in (0,1),\\
& u'_\varepsilon(0)=0 \qquad \mbox{and} \qquad u_\varepsilon(1) = 0, 
\end{aligned}
\right. 
\end{equation}
where 
\[
F_2(\hat{r})=R^2 \varepsilon^{-2}\left[\lambda(\hat{c}_{R+\varepsilon})-\lambda(\hat{c}_R) - \lambda'(\hat{c}_R)w_\varepsilon\right] + 2R\varepsilon^{-1}\left[\lambda(\hat{c}_{R+\varepsilon}
-\lambda(\hat{c}_R))\right]+\lambda(\hat{c}_{R+\varepsilon}). 
\]
Indeed, subtracting the differential equation in \eqref{problem-for-v_R} from 
that in \eqref{problem-for-w_eps}, divided by $\varepsilon$, we immediately have 
\begin{eqnarray*}
(\hat{r}u'_\varepsilon(\hat{r}))'&=& R^2 \hat{r}a_1(\hat{r})\left\{\dfrac{w_\varepsilon}{\varepsilon}\right\}+\hat{r}F_1(\hat{r})-R^2\hat{r}\lambda'(\hat{c}_R(\hat{r}))v_R(\hat{r})
-2R\hat{r}\lambda(\hat{c}_R(\hat{r}))\\
&=&R^2r\lambda'(\hat{c}_R(\hat{r}))\left\{\dfrac{w_\varepsilon}{\varepsilon}-v_R\right\}+R^2\hat{r}\left[a_1(\hat{r})-\lambda'(\hat{c}_R(\hat{r})\right]\left\{\dfrac{w_\varepsilon}{\varepsilon}\right\}\\
&& \hspace{20 pt}+\hat{r}\left[2R(\lambda(\hat{c}_{R+\varepsilon})-\lambda(\hat{c}_R))+\varepsilon\lambda(\hat{c}_{R+\varepsilon})\right].
\end{eqnarray*}
Using $a_1(\hat{r})w_\varepsilon =\lambda(\hat{c}_{R+\varepsilon})-\lambda(\hat{c}_R)$, we get  \eqref{problem-for-u-eps}.

Next, due to \eqref{s4:continuous dependence} and \eqref{bounds-on-lambda} we have that $\|F_2\|_{C([0,1])}\leq C$ and thus by applying Lemma~\ref{lemma:aux-existence} we finally get
$
\|u_\varepsilon\|_{C^1([0,1])}\leq C\varepsilon.
$
This implies that $v_R(\hat{r})=\partial_R\hat{c}_{R}(\hat{r})$ and, due to \eqref{last-straw}, $H(R)=R^{-1}\hat{c}_R'(1)$ is $C^1$ with respect to $R>0$, as desired. 
\end{proof}

\begin{theorem}
\label{t:dynamics}
Let $c_0>0$ and  $R_0>1$. Let \( (c(r,t),R(t)) \) be the solution of 
\eqref{vfbp2d-nondim-compact}. Let also $C_1$ and $C_2$ be positive constants defined in \eqref{C1C2}. Then 
$ C_1\leq R'(t) \leq C_2 $ for all $t > 0.$
Moreover, the function $([R(t)]^2)'$ increases as $t > 0$ increases and $([R(t)]^2)' = O(t)$ as 
$t \to\infty$.
\end{theorem}

\begin{rmk}
    This theorem shows that the sufficient nutrient in the peripheral region of a growing bacterial
    colony leads to the constant radial expansion as found in experiment and agent-based simulations
    \cite{Pirt_1967,Wimpenny_1981,Warren_eLife2019,Kannan_NatCommun2025}. Moreover, 
    the rate of change of the disk area is linear in time and the area increases to $\infty$ as  $t \to \infty. $
\end{rmk}

\begin{proof}[Proof of Theorem~\ref{t:dynamics}]
The bounds $C_1 \le R'(t) \le C_2$ $(t > 0)$
follows immediately from \eqref{vfbp2d-nondimc-eq-for-R} and Theorem~\ref{t:radialexp}. 
It then follows that $R(t)$ increases with $t$ and $R(t) = O(t)$ as $t \to \infty.$
By \eqref{vfbp2d-nondim}, \eqref{defineHR}, and Proposition~\ref{p:cDisk},  
the function $([R(t)]^2)' = 2 R(t) R'(t) = 2 R(t) \partial_rc(R(t), t)$ increases as time $t$ increases
and $([R(t)]^2)' = O(t)$ as $t \to\infty$.
\end{proof}

\section{Conclusion}
\label{s:Conclusion}

In this work, we studied the expansion of a bacterial colony on a hard agar substrate, focusing
on the kinetic behaviors of such expansion that have been well documented experimentally. 
Our studies are closely related to the
recent works by Warren et al.\ (2019) \cite{Warren_eLife2019} and Kannan et al.\ (2025)
\cite{Kannan_NatCommun2025} that developed a hybrid agent-based and 
reaction-diffusion equation model and predicted successfully 
the constant radial expansion of the colony for very long time (up to 65 hours)
and the near-constant expansion and then slow-down in the vertical direction. 
Here, we aimed at examining the role of the continuum reaction-diffusion equation modeling in 
the vertical and horizontal expansion. 

As the agent-based description of individual cell activities 
is limited to small colonies and short expansion time period, we constructed a three-dimensional continuum model
without the description of individual cells. The colony expansion is modeled by a moving boundary while
the local cell growth rate and the velocity field are connected through the conservation of biomass
and Darcy's law. For a fixed configuration of a colony-agar system, the equilibrium nutrient concentration is shown to 
exist uniquely.

We constructed, analyzed, and performed numerical simulations
of a one-dimensional model for the vertical expansion 
and a two-dimensional disk model for the radial expansion of the colony, both reduced from
the full three-dimensional model with biologically natural assumptions.
We showed the well-posedness of the underlying mathematical models that include the existence, uniqueness, and structure
of the nutrient concentration profile as well as the moving-boundary problem in both vertical and horizontal expansion models. 
The key findings from our detailed analysis that have significant biological consequences include:

    {\it The vertical expansion:}
    \begin{compactenum}
        \item[$\bullet$]
        The nutrient concentration drops quadratically and then exponentially in the colony. This 
  proves rigorously what has been discovered in \cite{Warren_eLife2019,Kannan_NatCommun2025} through agent-based simulations and
  experiment; 
        \item[$\bullet$]
    The height $z_\ast$ at which the nutrient concentration reaches the Monod constant, implying that above $z_\ast$ in the
    colony cells are hardly grow due to the lack of nutrient, drops exponentially fast to its limiting value during the colony vertical expansion; 
        \item[$\bullet$]
    Before slowing down, the colony height increases with acceleration and eventually gets close to a linear expansion. 
    \end{compactenum}
  
    {\it The horizontal expansion:}
    \begin{compactenum}
        \item[$\bullet$]
        Our disk model predicts a constant expansion in the large-time limit, agreeing with experiment \cite{Kannan_NatCommun2025}; 
        \item[$\bullet$]
        There is a peripheral ring region of the colony with a fixed size where there is sufficient nutrient for all the time
        during the radial expansion. This confirms the experiment reported in \cite{Kannan_NatCommun2025} on the existence 
        of such a region (about $20\, \mu m$ wide) during the expansion for more than $65$ hours. 
    \end{compactenum}


We now discuss several potential issues and possible future studies. 

(1) {\it The analysis of the full three-dimensional model with a contact-angle boundary condition}. 
While we have decoupled the vertical and horizontal dynamics of bacterial colony expansion into two separate models, 
we expect that the asymptotically linear expansion in both vertical and horizontal directions  
will persist in the three-dimensional moving-boundary model introduced in Section~\ref{ss:3Dmodel}.
The choice of a boundary condition at the colony edge is an important aspect of 
the three-dimensional continuum model; cf.\ Remark~\ref{rmk:edge-condition}.
One possibility is to impose a condition analogous to those used in capillarity problems \cite{finn1986,huh1971hydrodynamic}, such as prescribing the contact angle (or, equivalently, the slope of the colony profile at the boundary). However, experiments indicate that a bacterial colony initially grows as a nearly flat monolayer and undergoes buckling only after an extended period of approximately 10 $\sim $ 15 
hours \cite{Warren_eLife2019}, primarily due to cell-agar interactions that resist bacterial motion. Consequently, in a 
moving-boundary formulation, the colony profile is expected to meet the substrate tangentially, so that both the colony 
height and the first derivatives of
the height function will vanish at the edge. As a result, the boundary of the colony is generally non-Lipschitz, making the mathematical analysis of the nutrient concentration and pressure functions substantially more challenging.
The development and analysis of the three-dimensional model with a correct contact-angle condition 
is an immediate next step of our study.

(2) {\it Continuum modeling of mechanical interactions and analysis of buckling.}  
For a dense bacterial colony, it is known  that cell-cell and cell-environment mechanical interactions, together with metabolic interactions, are 
crucial to the formation and development of the colony 
\cite{Warren_eLife2019,Kannan_NatCommun2025,Tsimring_Ordering_PNAS2008,Waclaw_PRL2013,GrantWaclaw_JRSocInterface2014,Waclaw_Interface2017,Levine_PhaseSeparation_PNAS2015}. 
As shown in \cite{Warren_eLife2019,Kannan_NatCommun2025}, mechanical buckling of a few layers of growing cells in the colony during
an early growth phase determines the colony radial expansion rate through the ``buckling width", which is 
a $20 \,\mu m$ peripheral region of the colony. 
Elasticity models, particularly plate bending models, can be used for the mechanical interactions in colony
\cite{Tsimring_Ordering_PNAS2008,Amar_BiofilmPlates_JMPS2009,Warren_eLife2019,Kannan_NatCommun2025,Tsimring_BucklingInstab_PhysBiol2011}. 
These models have variational principles but additional assumptions may be needed as an expanding colony is a 
non-equilibrium system. 
An alternative and more interesting way is to derive the evolution of the stress field through coarse-graining
the agent-based description of individual cells in the colony.


(3) {\it The construction and analysis of a three-dimensional moving-boundary model that includes
additional biological effects, such as colony maintenance, cell starvation, dead cell zones, 
multiple bacterial strains, and multiple biochemical species.}
To understand the long-time dynamics and structures of an expanding colony, these effects need to be included in a continuum model 
\cite{Kannan_NatCommun2025}. Some of these effects, such as the colony maintenance that consumes  nutrients, 
can be modeled readily through the equation for the local cell growth. As in a continuum models, individual cells are
not represented, effects such as the cell starvation and death will need to be modeled through the metabolic interaction, with perhaps
additional assumptions. Including multiple bacterial strains and biochemical species in a continuum model can be done conceptually. But
the analysis of such a model will be very challenging. Computer simulations will then be needed to help detailed studies with
continuum modeling.

\section*{Acknowledgment}

BL was supported in part by the US National Science Foundation
through the grant  MCB-2029574 and DMS-2208465. MP was supported by the UC Riverside Regents Faculty Fellowships grant.

\bigskip

\bibliography{BiofilmsAnalysis}
\bibliographystyle{plain}


\end{document}